\documentclass{amsart}

\usepackage{amsmath}
\usepackage{amsfonts}
\usepackage{amssymb}
\usepackage{amsthm}
\usepackage{graphicx}
\usepackage{dsfont}
\usepackage{pdfcomment}
\usepackage{tikz}
\usepackage{tikz-3dplot}
\usepackage{subfigure}
\usepackage{soul}
\usepackage{enumerate}
\usepackage{cancel}

\usepackage{hyperref}
\usepackage{cleveref}
\usepackage{esvect}

\usetikzlibrary{decorations.markings}
\newtheorem{theorem}{Theorem}[section]

\newtheorem{corollary}[theorem]{Corollary}

\newtheorem{definition}[theorem]{Definition}

\allowdisplaybreaks

\newtheorem{lemma}[theorem]{Lemma}

\newtheorem{proposition}[theorem]{Proposition}
\newtheorem{remark}[theorem]{Remark}

\DeclareMathOperator{\Real}{Re}
\DeclareMathOperator{\Imaginary}{Im}

\DeclareMathOperator{\hf}{H^{\frac{1}{2}}(\partial \mathbb{D})}
\DeclareMathOperator{\dev}{dev}
\DeclareMathOperator{\Li}{Li_2}
\renewcommand{\Re}{\Real}
\renewcommand{\Im}{\Imaginary}
\newcommand{\ii}{\mathbf{i}}

\newcommand{\UT}{\mathcal{T}}
\newcommand{\WT}{\mathcal{T}_{\mathrm{WP}}}

\title{Infinite circle patterns in the Weil-Petersson class}
\author{Wai Yeung Lam}
\address{Mathematics Research Unit, Université du Luxembourg, L-4364 Esch-sur-Alzette} \email{wai.lam@uni.lu}
\begin{document}

\begin{abstract}

Analogous to Weil-Petersson quasicircles, we investigate infinite circle patterns in the Euclidean plane parameterized by discrete harmonic functions of finite Dirichlet energy. The space of such circle patterns forms an infinite-dimensional Hilbert manifold homeomorphic to the Sobolev space of half-differentiable functions on the unit circle. The Hilbert manifold is equipped with a Riemannian metric induced from the Hessian of a hyperbolic volume functional. We relate this Riemannian metric to the symplectic form on the Sobolev space of half-differentiable functions via an analogue of the Hilbert transform. Every such circle pattern induces a quasiconformal homeomorphism from the unit disk to itself, whose boundary extension belongs to the Weil-Petersson class of the universal Teichmüller space. Our results shed light on Jordan domains packed by infinite circle patterns of hyperbolic type, a subject highlighted by He and Schramm.
\end{abstract}
\maketitle
\section{Introduction}

Conformal maps between surfaces are fundamental in low-dimensional topology and have wide-ranging applications. They are characterized as mappings that locally preserve angles, sending infinitesimal circles to themselves. The uniformization theorem states that every non-compact simply-connected Riemann surface is conformally equivalent to either the complex plane or the unit disk. The former type is called parabolic and the latter is called hyperbolic. In particular, every simply-connected proper domain $\Omega \subsetneq \mathbb{C}$ is of hyperbolic type and conformally equivalent to the unit disk $\mathbb{D}$, admitting a Riemann mapping $g: \mathbb{D} \to \Omega$ as a biholomorphic map.

A refinement of this theory concerns the regularity of a Jordan domain's boundary. The boundary $\partial \Omega$ is a Weil-Petersson quasicircle if the logarithmic derivative of the Riemann mapping has finite Dirichlet energy \cite{Takhtajan2005}. That is, the conformal factor $u := \Re \log g'$ is a harmonic function on the unit disk satisfying
\[
\mathcal{E}_{\mathbb{D}}(u) := \iint_{\mathbb{D}} |\nabla u|^2 \, dx dy < \infty.
\]
Weil-Petersson quasicircles correspond to the Weil-Petersson class in the universal Teichmüller space via conformal welding. They have been extensively studied in the contexts of Teichmüller theory, mathematical physics, and geometric function theory. Recently, they have also been linked to the Loewner energy of Jordan curves \cite{Wang2019} and the renormalized volume of hyperbolic 3-manifolds \cite{Bridgeman2025}. For a comprehensive survey of Weil-Petersson curves, we refer to Bishop's recent works \cite{Bishop2022,Bishop2025}. 

In the discrete setting, circle patterns replace infinitesimal circles with configurations of finite-sized circles having prescribed intersection angles. Thurston proposed viewing the map induced by two circle patterns with the same combinatorics as a discrete conformal map \cite{Stephenson2005}. He conjectured that this approach provides a natural discretization of the Riemann mapping theorem, a fact later verified by Rodin and Sullivan \cite{Rodin1987}. Their proof of convergence relies on the rigidity of infinite circle patterns of \emph{parabolic type}.

In contrast, infinite circle patterns of \emph{hyperbolic type} are flexible and admit a rich deformation theory. He and Schramm \cite[Theorem 1.2]{HeSchramm1995} proved that for any fixed triangulation of an open disk of hyperbolic type, every simply connected proper domain in the plane admits an infinite circle packing with the prescribed combinatorics, where circles only accumulate at the boundary. While He and Schramm highlighted the importance of domains packed by circles of hyperbolic type, their characterization has remained undeveloped. In this paper, we investigate the deformation space of infinite circle patterns of hyperbolic type and establish its relation to the universal Teichmüller space.

 Formally, a circle pattern in the plane is a realization of the vertices of a planar graph $G=(V,E,F)$ such that each face has a circumcircle passing through its vertices. It also induces a realization of the dual cell decomposition $G^*=(V^*,E^*,F^*)$ where the dual vertices $V^*\cong F$ are realized as the circumcenters of the faces. For every dual edge $\phi \psi$ joining two circumcenters, we measure the intersection angle $\Theta_{\phi \psi} = \Theta_{\psi \phi}$ between the corresponding circumcircles, thereby defining a map $\Theta : E^* \to [0, 2\pi)$. Up to Euclidean transformations, a circle pattern is uniquely determined by the Euclidean radii of the circles $R:F \to \mathbb{R}_{>0}$ together with the intersection angles $\Theta$. Furthermore, a circle pattern is called \emph{Delaunay} if all intersection angles lie in $[0, \pi)$. The Delaunay condition is equivalent to the local convexity of an associated ideal polyhedral surface in hyperbolic 3-space \cite{Bobenko2004}.

Throughout the paper, we study infinite circle patterns on topological open disks. Every such circle pattern determines a flat Euclidean structure on the disk by gluing circumdisks along their common chords. By varying the Euclidean radii of the circles while keeping the intersection angles fixed, we obtain different flat Euclidean metrics on the underlying topological disk. Each such metric supports a circle pattern with the same combinatorics and the prescribed intersection angles, but the geometric realization in the plane may differ. Our goal is to investigate the deformation space of circle patterns. We recall that a graph is called \emph{transient} if every simple random walk on the graph escapes to infinity with non-zero probability.

\begin{definition}\label{def:infintheta}
	Suppose $(V,E,F)$ is a cell decomposition of a topological open disk $\Omega$ in the sense that the 1-skeleton graph $G=(V,E)$ has one infinite end. We further assume that $G$ is transient with uniformly bounded vertex degree and face degree and, for some small constant $\epsilon_0>0$, supports a function $\Theta : E^* \to (\epsilon_0,\pi-\epsilon_0)$ satisfying
	\begin{enumerate}
	\item[(A1)] For every closed loop $\gamma$ in the dual graph that bounds exactly one dual face,
	\[
	\sum_{\phi\psi \in \gamma} \Theta_{\phi\psi} = 2\pi.
	\]
	\item[(A2)] For every closed loop $\gamma$ in the dual graph that bounds more than one dual face,
	\[
	\sum_{\phi\psi \in \gamma} \Theta_{\phi\psi} > 2\pi + \epsilon_0.
	\]
	\end{enumerate}
	We denote by $\tilde{P}(\Theta)$ the collection of functions $R:F \to \mathbb{R}_{>0}$ where there exists a circle pattern in the plane with intersection angles $\Theta$ and Euclidean radii $R$. 
\end{definition}

\begin{definition}\label{def:metric}
	 By gluing circumdisks, every function $R \in \tilde{P}(\Theta)$ determines a smooth Euclidean metric $\sigma(\Theta,R)$ on the topological disk $\Omega$ which supports a circle pattern with intersection angles $\Theta$ and Euclidean radii $R$. We call $\sigma(\Theta,R)$ the associated circle pattern metric.
\end{definition}

For every circle pattern metric $\sigma(\Theta,R)$, there exists a local isometry to the Euclidean plane
\[
\dev: (\Omega, \sigma(\Theta,R)) \to \mathbb{R}^2
\]
called the \emph{developing map}. This map is unique up to Euclidean transformations but may or may not be globally injective. We say the circle pattern is \emph{embedded in} $\mathbb{R}^2$ if the developing map is injective.

The prescribed intersection angle $\Theta$ plays the role of a discrete conformal structure. By varying the radii, the set $\tilde{P}(\Theta)$ enumerates the collection of discrete conformal maps. It is known that the set $\tilde{P}(\Theta)$ is non-empty. 

\begin{theorem}[Discrete uniformization \cite{HeSchramm1993,Ge2025}]\label{thm:discrete_uniformization} 
	Let $(V,E,F)$ be a cell decomposition of a topological open disk $\Omega$ and $\Theta : E^* \to (\epsilon_0,\pi - \epsilon_0)$ satisfying Definition \ref{def:infintheta}. Then there exists a function $R^{\dagger}\in \tilde{P}(\Theta)$ such that the developing map $\dev$ is an isometry onto the unit open disk $\mathbb{D} \subset \mathbb{R}^2$. 

	Two such functions $R^{\dagger}, R^{\ddagger} \in \tilde{P}(\Theta)$ correspond to circle patterns that differ by a Möbius transformation preserving the unit disk.
\end{theorem}

The assumption on the graph being transient with bounded vertex degree ensures that we have a discrete uniformization to the unit open disk, instead of the entire plane. Figure \ref{fig:circle_patterns_disk} shows a circle pattern under the discrete uniformization. It is obtained from a circle packing together with its dual circle packing (See \cite{Stephenson2005}). The intersection angles are constant and all equal to $\pi/2$. The discrete uniformization for circle packings is established by He and Schramm \cite{HeSchramm1993} and extended to circle patterns with general intersection angles by Ge, Yu, and Zhou \cite{Ge2025}.

\begin{figure}
	\centering
	\includegraphics[width=0.6\textwidth]{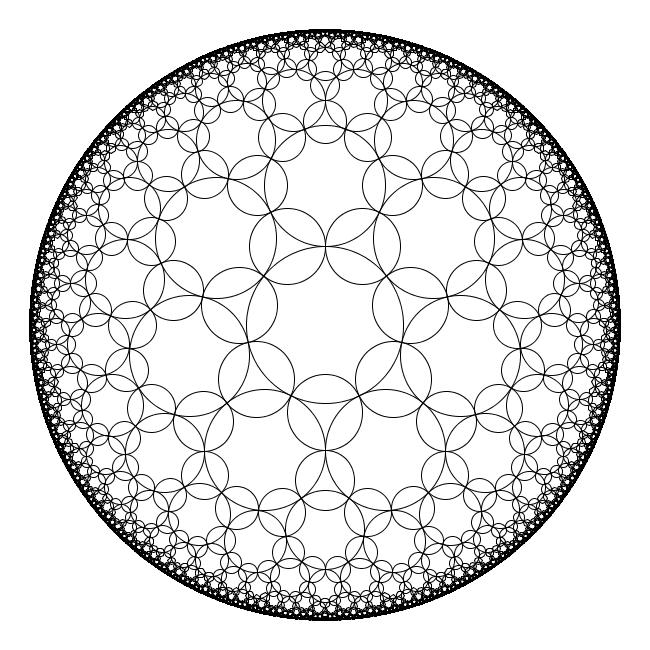}
	\caption{A uniformized circle pattern with constant intersection angle $\Theta\equiv \pi/2$ filling the unit disk. It is induced from a circle packing together with its dual circle packing.}
	\label{fig:circle_patterns_disk}
\end{figure}

\begin{figure}[h!]
\begin{minipage}{0.48\textwidth}
	\centering
	\includegraphics[width=\textwidth]{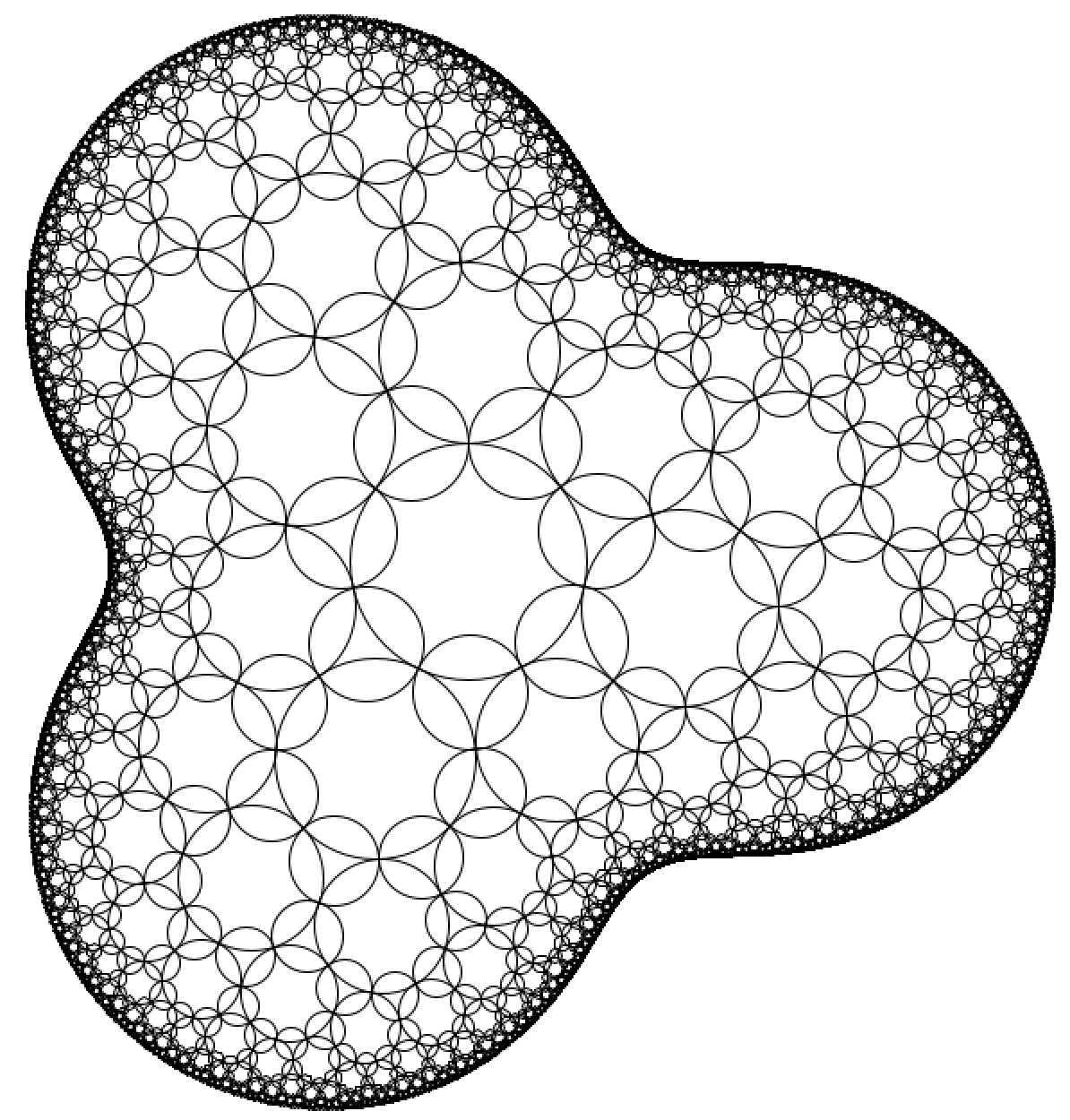}
  \end{minipage}\hfill
  \begin{minipage}{0.48\textwidth}
	\centering
	\includegraphics[width=\textwidth]{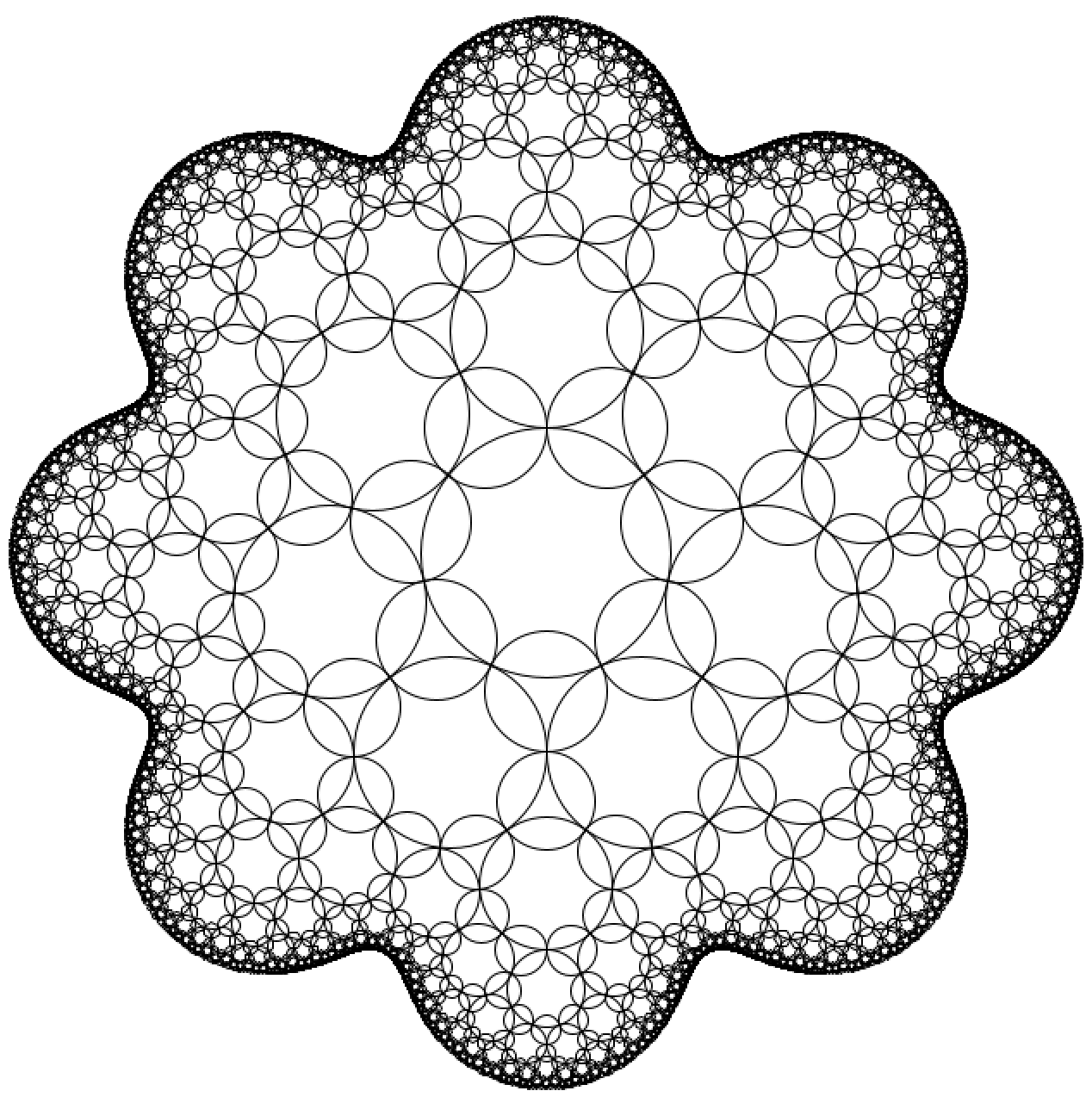}
  \end{minipage}
\begin{minipage}{0.4\textwidth}
	\centering
	\includegraphics[width=\textwidth]{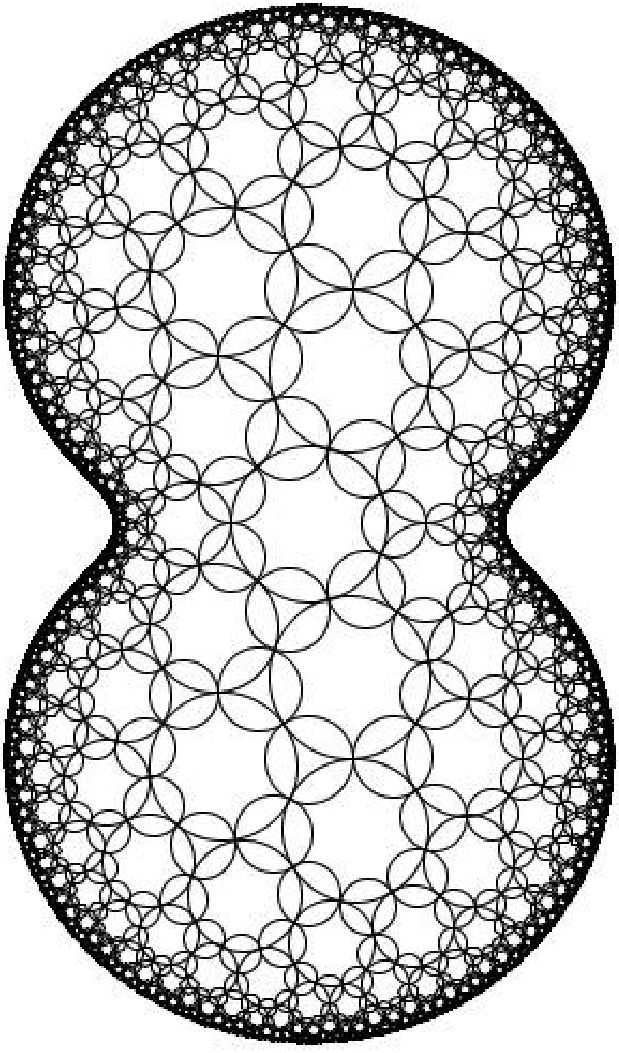}
  \end{minipage}\hfill
  \begin{minipage}{0.43\textwidth}
	\centering
	\includegraphics[width=\textwidth]{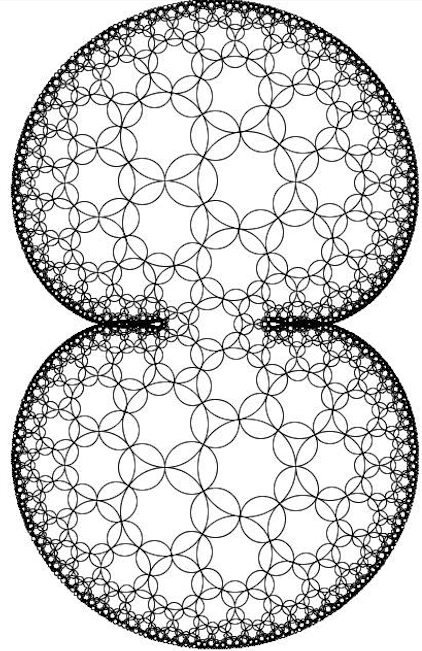}
  \end{minipage}

	\caption{Circle patterns in the deformation space $P(\Theta,R^{\dagger})$ with constant intersection angle $\Theta \equiv \pi/2$. Each arises from a superposition of a circle packing and its dual circle packing. The circle pattern on the bottom right is not embedded in $\mathbb{R}^2$.}
	\label{fig:circle_patterns_deformed}
\end{figure}

We denote the Hilbert space of Dirichlet-finite functions on the dual graph by
\[
\mathbf{D}(F) := \left\{ u: F \to \mathbb{R} \mid \mathcal{E}(u) := \sum_{\phi\psi \in E^*} |u_\phi - u_\psi|^2 < \infty \right\}
\]
where $\mathcal{E}(u)$ is the combinatorial Dirichlet energy of $u$. Given a function $u:F \to \mathbb{R}$, we modify the radii of the uniformized circle pattern by setting new radii for every face $\phi \in F$
\[R_\phi := e^{u_\phi} R^{\dagger}_\phi.\] 

\begin{definition}\label{def:wpclasscircle}
	Given a uniformized circle pattern as in Theorem \ref{thm:discrete_uniformization}, we define the space of circle patterns in the Weil-Petersson class
	\[
	P(\Theta,R^{\dagger}):= \left\{ u \in \mathbf{D}(F) \;\middle|\; e^{u} R^{\dagger} \in \tilde{P}(\Theta) \right\}.
	\]
\end{definition}

Elements in $P(\Theta,R^{\dagger})$ correspond to circle patterns that arise as deformations of the uniformized circle pattern on the unit disk and may or may not be embedded in $\mathbb{R}^2$. A function $u \in P(\Theta,R^{\dagger})$ represents the logarithmic change in radii and acts as a discrete conformal factor for the change in metrics. We establish several fundamental properties for the space $P(\Theta, R^{\dagger})$ as an infinite-dimensional Hilbert manifold and relate it to the Sobolev space of half-differentiable functions on the unit circle.

We investigate the global topology of $P(\Theta,R^{\dagger})$ by analyzing its projection onto the space of discrete harmonic functions $\mathbf{HD}(F)$. A function $h \in \mathbf{D}(F)$ is called discrete harmonic if it satisfies the discrete Laplace equation: for every face $\phi \in F$ \[\sum_{\psi} (h_\psi - h_\phi) = 0\] where the sum is over all dual vertices $\psi$ adjacent to $\phi$. Since the underlying graph is transient, the space of Dirichlet-finite functions admits the Royden decomposition into orthogonal subspaces
\[
\mathbf{D}(F) = \mathbf{HD}(F) \oplus \mathbf{D}_0(F),
\]
where $\mathbf{D}_0(F)$ is the closure of the space of finitely supported functions. Let $p_F : \mathbf{D}(F) \to \mathbf{HD}(F)$ denote the corresponding orthogonal projection onto the harmonic component.

\begin{theorem}\label{thm:homHDF}
    The space $P(\Theta,R^{\dagger})$ forms an infinite-dimensional Hilbert submanifold in $\mathbf{D}(F)$ and is homeomorphic to $\mathbf{HD}(F)$ via the projection $p_F$. 
\end{theorem}

The proof establishes that for every discrete harmonic function $h \in \mathbf{HD}(F)$, there exists a unique $v \in \mathbf{D}_0(F)$ such that the sum $h + v \in P(\Theta,R^\dagger)$. This is demonstrated using a variational method involving the volume of hyperbolic polyhedra associated with the circle patterns. Such a variational approach was pioneered by Colin de Verdière \cite{Colin1991} and developed by Rivin \cite{Rivin}, Leibon \cite{Leibon2002}, Bobenko and Springborn \cite{Bobenko2004}, and Guo and Luo \cite{Luo2009} for finite circle patterns on closed surfaces. In contrast, we extend their framework to infinite circle patterns and apply it to study the deformation space.

Elements in $P(\Theta,R^{\dagger})$ are analogous to harmonic functions in that they satisfy a discrete Laplace equation with respect to certain geometric edge weights. They admit conjugate harmonic functions defined on the primal graph $(V,E)$, leading to a dual parametrization of the space of circle patterns. Instead of specifying the change in radii, this dual parametrization describes the change in the central angles $\alpha^{\dagger}$ at circumcenters subtended by the common chords of adjacent circles. Similar to $\mathbf{D}(F)$, the space of Dirichlet-finite functions on the primal graph $\mathbf{D}(V)$ admits the Royden decomposition into orthogonal subspaces
\[
\mathbf{D}(V) = \mathbf{HD}(V) \oplus \mathbf{D}_0(V).
\]
We remark that two functions $u, \tilde{u} \in P(\Theta,R^{\dagger})$ represent circle patterns that differ by a constant scaling if and only if $u - \tilde{u}$ is a constant function. This defines an equivalence relation on $P(\Theta,R^{\dagger})$ and motivates our consideration of the quotient space $P(\Theta,R^{\dagger})/\mathbb{R}$, which parameterizes circle patterns up to global scaling.

\begin{theorem}\label{thm:hilbertcentral}
	There is a Hilbert submanifold $P(\Theta,\alpha^{\dagger}) \subset \mathbf{D}(V)$ parameterizing circle patterns with intersection angles $\Theta$ via deformations of central angles. The space $P(\Theta,\alpha^{\dagger})$ is homeomorphic to $\mathbf{HD}(V)$ via the projection $p_V$.

	Furthermore, there is a homeomorphism \[ \mathcal{C}^{\Theta}: P(\Theta, R^{\dagger})/ \mathbb{R} \to P(\Theta, \alpha^{\dagger})/\mathbb{R}\] that identifies pairs of deformations yielding the same geometric circle pattern up to global scalings and rotations. Here $P(\Theta, R^{\dagger})/ \mathbb{R}$ and $P(\Theta, \alpha^{\dagger})/\mathbb{R}$ denote the quotient spaces by constant functions. We call $\mathcal{C}^{\Theta}$ the conjugation map.
\end{theorem}

To establish that the projection of $P(\Theta,\alpha^{\dagger})$ onto $\mathbf{HD}(V)$ is bijective, we employ a similar variational method involving hyperbolic volume. This dual parametrization is more involved since angle parameters are constrained to lie in the interval $(0, \pi)$. For this, we have to show that for every $\epsilon>0$, every discrete harmonic function in $\mathbf{HD}(V)$ admits a perturbation in $\mathbf{D}_0(V)$ such that the resulting function is $\epsilon$-Lipschitz on the primal graph $(V,E)$.

Both Hilbert manifolds $P(\Theta,R^{\dagger})$ and $P(\Theta,\alpha^{\dagger})$ can be parameterized by smooth harmonic functions on the unit disk with finite Dirichlet energy. The uniformized circle pattern on the unit disk determines geodesic cell decompositions for both the primal and dual graphs, where primal vertices are realized as intersection points of circles and dual vertices are circumcenters. We show that these geodesic decompositions possess uniform bounds on angles and side-length ratios. Hence, the geodesic decompositions are \emph{good} as considered by Angel, Barlow, Gurel-Gurevich, and Nachmias\cite{Angel2016} as well as Chelkak \cite{Chelkak2016}. Consequently, by adding diagonals, any discrete function can be extended to a piecewise linear function on the disk whose classical Dirichlet energy is comparable to its combinatorial energy. Using the Royden decomposition for the space of classical Dirichlet-finite functions, we project these piecewise linear extensions to the space of smooth harmonic functions $\mathbf{HD}(\mathbb{D})$. Building on results by Hutchcroft \cite{Hutchcroft2019}, we establish that this construction yields homeomorphisms respectively from $P(\Theta,R^{\dagger})$ and $P(\Theta,\alpha^{\dagger})$ to $\mathbf{HD}(\mathbb{D})$.

\begin{theorem}\label{thm:homeoHD}
	Under the uniformized circle pattern on the unit disk, we denote the position of the intersection points of the circles 
	\[z_V:V \to \mathbb{D}\]
	and the position of the circumcenters \[ z_F:F \to \mathbb{D}.\] Then there are homeomorphisms from the Hilbert manifolds to the Hilbert vector space of classical harmonic Dirichlet functions on the unit disk
	\begin{align*}
		\Phi_V: P(\Theta,\alpha^{\dagger})  \to \mathbf{HD}(\mathbb{D}), \quad \Phi_F: P(\Theta,R^{\dagger}) \to \mathbf{HD}(\mathbb{D})
	\end{align*}
	such that for every $u \in P(\Theta,\alpha^{\dagger})$ and $v \in P(\Theta,R^{\dagger})$, we have
	\[
	\Phi_V(u) \circ z_V - u \in \mathbf{D}_0(V), \quad \Phi_F(v) \circ z_F - v \in \mathbf{D}_0(F).		
	\]
	\end{theorem}

We further discuss the boundary value of our discrete functions on the unit circle, which leads to an analogue of the Hilbert transform. It is a classical result that every Dirichlet-finite harmonic function on the unit disk has a well-defined boundary value as a half-differentiable function on the unit circle. This induces a bounded linear isomorphism to the Sobolev space of half-differentiable functions on the unit circle $\mathcal{B}:\mathbf{HD}(\mathbb{D}) \to \hf$. 

\begin{corollary}\label{cor:boundvalue}
    The boundary value mappings
	\begin{align*}
			\mathcal{B}_V:=& \mathcal{B} \circ \Phi_V: P(\Theta,\alpha^{\dagger}) \to \hf, \\ \mathcal{B}_F:=& \mathcal{B} \circ \Phi_F: P(\Theta,R^{\dagger}) \to \hf
	\end{align*} are homeomorphisms and send constant functions to constant functions. 
\end{corollary}
Particularly, the conjugation map $\mathcal{C}^{\Theta}$ induces a nonlinear homeomorphism $\mathfrak{H}^{\Theta}: \hf/\mathbb{R} \to \hf/\mathbb{R}$ on the quotient spaces defined as the composition
	\[
	\mathfrak{H}^{\Theta} := \overline{\mathcal{B}}_V \circ \mathcal{C}^{\Theta} \circ \overline{\mathcal{B}}_F^{-1}
	\]
which is an analogue of the Hilbert transform. It captures various information of our space of circle patterns.

\begin{figure}
	\centering
	\includegraphics[width=1
	\textwidth]{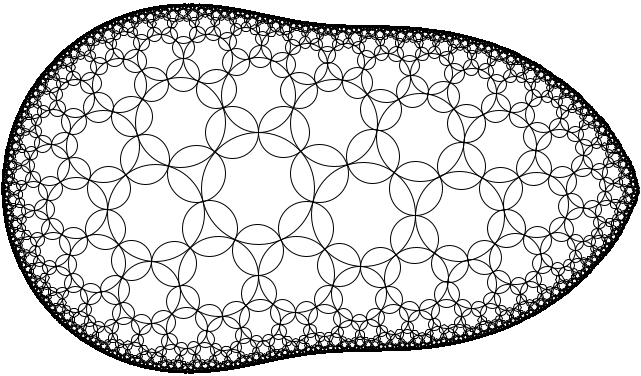}
	\caption{A circle pattern corresponding to $u \in P(\Theta,R^{\dagger})$ whose boundary value $u_{\partial \mathbf{D}} \in H^{\frac{1}{2}}(\partial \mathbb{D})$ is infinite at one point on $z=1$. Specifically, the boundary value is $u_{\partial \mathbb{D}}(z) = \Re \left( \sum_{n=2}^{\infty} \frac{z^n}{n \log n} \right)$.}
	\label{fig:unbounded}
\end{figure}

We equip the Hilbert manifolds $P(\Theta,R^{\dagger})$ and $P(\Theta,\alpha^{\dagger})$ with Riemannian metrics induced from the discrete Dirichlet energy with respect to geometric edge weights or equivalently, the Hessian of the hyperbolic volume functional. Namely, each circle pattern corresponding to $u \in P(\Theta,R^{\dagger})$ induces a graph Laplacian $\Delta_{c^*}$ with respect to geometric edge weight $c^*:E^* \to \mathbb{R}$. The tangent space $T_u P(\Theta,R^{\dagger})$ is identified with the space of discrete harmonic functions $\ker \Delta_{c^*}$. We define the Riemannian metric on $T_{[u]}\left(P(\Theta,R^{\dagger})/\mathbb{R}\right)$ to be the metric induced by the discrete Dirichlet energy 
\[
\mathcal{E}_{c^*}(v) := \sum_{\phi\psi \in E^*} c^*_{\phi\psi} (v_\phi - v_\psi)^2.
\]
Similarly, the space $P(\Theta,\alpha^{\dagger})/\mathbb{R}$ is also equipped with a Riemannian metric induced from the discrete Dirichlet energy with respect to geometric edge weights, which arise from the Hessian of the other functional.
\begin{corollary}
The conjugation map $\mathcal{C}^{\Theta} \colon P(\Theta, R^{\dagger})/ \mathbb{R} \to P(\Theta, \alpha^{\dagger})/\mathbb{R}$ is an isometry relative to the Riemannian metrics induced by the discrete Dirichlet energies with geometric edge weights.
\end{corollary}

The Sobolev space of half-differentiable functions modulo constant functions $H^{\frac{1}{2}}(\partial \mathbb{D})/\mathbb{R}$ is a Hilbert manifold. Its Riemannian metric is determined by the classical Dirichlet energy of the smooth harmonic extension on the unit disk. We explore the relation of the Riemannian metric on $H^{\frac{1}{2}}(\partial \mathbb{D})/\mathbb{R}$ and that on $P(\Theta,R^\dagger)/\mathbb{R}$
under the boundary value mapping $\mathcal{B}_F$.

We consider a bilinear pairing $b: \mathbf{D}(F) \times \mathbf{D}(V) \to \mathbb{R}$.
Given any two functions $u \in \mathbf{D}(F)$ and $v \in \mathbf{D}(V)$, there is a natural pairing
\[
b(u,v):= \frac{1}{2}\sum_{ij \in \vec{E}} (u_{\psi}-u_{\phi})(v_j-v_i) =\sum_{ij \in E} (u_{\psi}-u_{\phi})(v_j-v_i)
\]
where $\phi$ denotes the left face of the oriented edge from $i$ to $j$ and $\psi$ denotes the right face. The first summation is over all the oriented edges while the second summation is over all the unoriented edges.

\begin{theorem}\label{thm:discretesym}
	Given the uniformized circle pattern on the unit disk. Suppose $u \in \mathbf{D}(F)$ and $v \in \mathbf{D}(V)$ have boundary values $u_{\partial \mathbb{D}}, v_{\partial \mathbb{D}} \in \hf $ respectively. Then
	\begin{equation*} \label{eq:dssym}
		b(u,v) = 2 \pi \cdot \omega(u_{\partial \mathbb{D}},v_{\partial \mathbb{D}})
	\end{equation*}
	where $\omega$ is the symplectic form on $\hf$.
\end{theorem}

If $u_{\partial \mathbb{D}},v_{\partial \mathbb{D}}$ are smooth functions on the unit circle, the symplectic form $\omega$ can be written as
\[
\omega(u_{\partial \mathbb{D}},v_{\partial \mathbb{D}}) = \frac{1}{2\pi} \int_{\partial \mathbb{D}} u_{\partial \mathbb{D}} \, d(v_{\partial \mathbb{D}}).
\]
Generally for functions in $\hf$, the symplectic form is defined via the Fourier coefficients of the functions.

We express the Riemannian metric on $P(\Theta,R^{\dagger})/\mathbb{R}$ in terms of the symplectic form $\omega$ on $\hf$ and the nonlinear homeomorphism $\mathfrak{H}^{\Theta}$. 

\begin{corollary}\label{cor:HilbertWP}
Let $u \in P(\Theta,R^{\dagger})$ with geometric edge weights $c^*:E^* \to \mathbb{R}$. Then for any $v,w \in T_{u}(P(\Theta,R^{\dagger})) \cong \ker \Delta_{c^*}$, we have
\[
\sum c^*_{\phi\psi} (v_\phi - v_\psi)(w_\phi - w_\psi) = 2\pi \cdot \omega \left(v_{\partial \mathbb{D}}, d\mathfrak{H}^{\Theta}_{u_{\partial \mathbb{D}}}(w_{\partial \mathbb{D}})\right).
\]
\end{corollary}

The circle patterns in $P(\Theta,R^{\dagger})$ are closely related to the Weil-Petersson class within the universal Teichmüller space. Indeed, every circle pattern in $P(\Theta,R^{\dagger})$ determines a natural quasiconformal mapping of the unit disk as follows. 

Let $u: F \to \mathbb{R}$ be a function such that $e^u R^{\dagger} \in \tilde{P}(\Theta)$. There exists a piecewise linear map 
\[
f_u: (\Omega,\sigma(\Theta,R^{\dagger})) \to (\Omega,\sigma(\Theta,e^u R^{\dagger}))
\]
from the uniformized circle pattern to the deformed circle pattern, which sends circumcenters to circumcenters and intersection points to intersection points. We regard $f_u$ as a discrete Riemann mapping. When $f_{u}$ is quasiconformal, there is a classical Riemann mapping from the unit disk $\mathbb{D}\cong(\Omega,\sigma(\Theta,R^{\dagger}))$ to the deformed circle pattern
\[
g_u: \mathbb{D} \to (\Omega,\sigma(\Theta,e^u R^{\dagger})).
\]
We consider the difference between the discrete Riemann mapping and the classical Riemann mapping via a self-homeomorphism of the unit disk
\[
g_u^{-1} \circ f_u: \mathbb{D} \to \mathbb{D}.
\]
We relate the Beltrami differential of this map to the logarithmic change in radii $u$. Since $g_u$ is conformal, we have the Beltrami differential $\mu(g_u^{-1} \circ f_u) = \mu(f_u)$.

\begin{theorem}\label{thm:projectionWP}
    Let $u: F \to \mathbb{R}$ be such that the Euclidean radii $e^u R^{\dagger}$ define a circle pattern in $\tilde{P}(\Theta)$. Then
    \begin{enumerate}
        \item The map $f_u$ is quasiconformal if and only if the logarithmic scaling factor $u$ has bounded discrete gradient, that is,
        \[
        \sup_{\phi \psi \in E^*} |u_{\phi}-u_{\psi}| < \infty.
        \] 
        \item If $u$ has finite combinatorial Dirichlet energy (i.e., $u \in P(\Theta,R^{\dagger})$), then the Beltrami differential of $f_u$ is $L^2$-integrable with respect to the hyperbolic metric on $\mathbb{D}$. Conversely, if the Beltrami differential of $f_u$ is $L^2$-integrable in the hyperbolic metric and the uniformized circle pattern admits a uniform lower bound on its hyperbolic radii, then $u$ has finite combinatorial Dirichlet energy.
    \end{enumerate}
\end{theorem}


The uniformized circle pattern in Figure \ref{fig:circle_patterns_disk} admits a uniform lower bound on the hyperbolic radii because of symmetry. Generally, to guarantee a uniform lower bound on the hyperbolic radii in the uniformized circle pattern, we believe that the graph $G$ has to satisfy the \emph{strong isoperimetric inequality} (See \cite{Soardi1994}). 

The previous theorem implies that when $u$ has bounded gradient, the map $g_u^{-1} \circ f_u$ is a quasiconformal homeomorphism of the unit disk. Such a mapping can be extended to a quasisymmetric homeomorphism of the unit circle, which defines an element in the universal Teichmüller space $\UT$.

\begin{corollary}
 We consider a broader class of circle patterns
    \[
    P_{\mathrm{QC}}(\Theta,R^{\dagger}) := \left\{ u: F \to \mathbb{R} \;\middle|\; e^{u} R^{\dagger} \in \tilde{P}(\Theta) \text{ and } \sup_{\phi \psi \in E^*} |u_{\phi}-u_{\psi}| < \infty \right\}
    \]
which contains $P(\Theta,R^{\dagger})$. There is a well-defined mapping
 \begin{align*}
        \pi: P_{\mathrm{QC}}(\Theta,R^{\dagger}) & \to \UT \\
        u &\mapsto g_u^{-1} \circ f_{u} |_{\partial \mathbb{D}}
    \end{align*}
and the image of $P(\Theta,R^{\dagger})$ is contained in the Weil-Petersson class $\WT \subset \UT$. 
\end{corollary}

We conjecture that the mapping $\pi$ yields a global homeomorphism between $P(\Theta,R^{\dagger})$ and the Weil-Petersson class $\WT$. Support for this conjecture lies in the shared topological structure of the two spaces: Shen \cite{Shen2018} established that $\WT$ is parameterized by the Sobolev space $H^{1/2}(\partial \mathbb{D})$, and Corollary \ref{cor:boundvalue} demonstrates that our space $P(\Theta,R^{\dagger})$ possesses this exact same structure. Moreover, our conjecture serves as an infinite-dimensional analogue of the Kojima-Mizushima-Tan conjecture, which postulates that every closed Riemann surface admits a unique complex projective structure supporting a circle pattern with prescribed intersection angles \cite{Kojima2003,Lam2021}. Resolving our conjecture will likely require two key developments. First, the discrete Hilbert transform $\mathfrak{H}^{\Theta}$ is expected to play a crucial role, mimicking the classical Hilbert transform in conformal welding. Second, one must establish the regularity of the boundary; specifically, we expect that for any circle pattern in $P(\Theta,R^{\dagger})$, the circles accumulate on a curve with the regularity of a Weil-Petersson quasicircle. As these steps demand a different set of analytical tools, we defer their investigation to future work.

\subsection{Further related work}

Our results involve the crossroads of the potential theory on infinite graphs, discrete geometry, hyperbolic geometry and Teichmüller theory. They provide a novel geometric correspondence for discrete harmonic functions on infinite planar graphs and establish a deep connection to the universal Teichmüller space.

\noindent \textbf{Random walks via geometric representations of graphs.}
Discrete harmonic functions with finite Dirichlet energy are fundamental objects that govern the behavior of simple random walks and characterize graph boundaries. A powerful approach relies on geometric representations of planar graphs, most notably through uniformized circle packings and square tilings \cite{BenjaminiSchramm1996circle,BenjaminiSchramm1996square,Georgakopoulos2016,AngelHut2016}. In these frameworks, a specific geometric realization is constructed to capture the probabilistic properties of the underlying graph.

Our work significantly expands this perspective. Instead of studying a single uniformized circle packing, we establish that the entire Hilbert space of Dirichlet-finite discrete harmonic functions parametrizes an infinite-dimensional deformation space of circle patterns. This constructs a new geometric correspondence and connection to the universal Teichmüller space. Particularly, our framework offers a novel discrete model for random geometry, providing new geometric tools for studying random conformal structures and Liouville quantum gravity (See the book by Berestycki and Powell \cite{berestycki_powell_2025}).

\noindent\textbf{Boundary response matrices of planar graphs.}
The nonlinear homeomorphism $\mathfrak{H}^{\Theta}$ serves as an analogue to the Dirichlet-to-Neumann operator, which is also known as the boundary response matrix of finite planar graphs. In the finite setting, this matrix exhibits profound connections to both cluster algebras and combinatorics \cite{Kenyon2012}. Notably, both the mapping $\mathfrak{H}^{\Theta}$ and the finite response matrix are invariant under cluster mutations of the underlying circle patterns via Miquel's six-circle theorem \cite{KLRR2022}. On the combinatorial side, Kenyon and Wilson proved an extension of Kirchhoff's matrix-tree theorem in which the response matrix enumerates specific subgraphs \cite{Kenyon2011}. Similar combinatorial formulations exist for closed surfaces, where they correspond to the period matrix of a Riemann surface \cite{Lam2025quasi}. Motivated by these phenomena, we anticipate that the mapping $\mathfrak{H}^{\Theta}$ admits a similar combinatorial formulation enumerating specific subgraphs in the infinite graph.

\section{Background}

\subsection{Classical Dirichlet-finite functions and half-differentiable functions}
We follow the framework established in \cite{NagSullivan1995,Sergeev2010,Hutchcroft2019}.

Every locally $L^2$-integrable, weakly differentiable function $u:\mathbb{D} \to \mathbb{R}$ has Dirichlet energy defined by
\[
\mathcal{E}(u) := \iint_{\mathbb{D}}  |\nabla u|^2 \,dx\,dy \in [0, \infty].
\]
We denote by $\mathbf{D}(\mathbb{D})$ the space of such functions with finite Dirichlet energy. This space is equipped with the inner product
\[
\langle u, v \rangle_{\mathbf{D}(\mathbb{D})} := \iint_{O} u v \,dx\,dy + \iint_{\mathbb{D}}  \nabla u \cdot \nabla v \,dx\,dy,
\] 
where $O$ is a fixed precompact open subset of $\mathbb{D}$. With this inner product, $\mathbf{D}(\mathbb{D})$ is a Hilbert space. Different choices of the precompact open set $O$ yield equivalent norms. 

We denote by $\mathbf{HD}(\mathbb{D}) \subset \mathbf{D}(\mathbb{D})$ the closed subspace of harmonic functions. Given a harmonic function $u \in  \mathbf{HD}(\mathbb{D})$ and a point $\xi \in \partial \mathbb{D}$, we define the radial limit $u^*(\xi):= \lim_{r \to 1^{-}} u(r\xi)$ whenever it exists. It is known that $u^*(\xi)$ exists almost everywhere on $\partial \mathbb{D}$ and is a half-differentiable function.

\begin{definition}
    The Sobolev space of half-differentiable functions $\hf$ is the Hilbert space consisting of functions $u \in L^2(\partial \mathbb{D}, \mathbb{R})$ that possess a generalized derivative of order $\frac{1}{2}$. In terms of Fourier series, this is given by
    \[
    \hf := \left\{ u \in L^2(\partial \mathbb{D}, \mathbb{R}) \;\middle|\; u(z) = \sum_{n \in \mathbb{Z}} u_n z^n, \quad \sum_{n \in \mathbb{Z}} |n| |u_n|^2  < \infty \right\}.
    \]
    On the quotient space $\hf/\mathbb{R}$, we have the inner product
    \[ 
    \langle u,v \rangle_{\hf} = \sum_{n \in \mathbb{Z}} |n| u_n \bar{v}_n, 
    \]
    and a symplectic form 
    \[
    \omega(u, v) = \frac{1}{\mathbf{i}}\sum_{n \in \mathbb{Z}} n u_n \bar{v}_n.
    \]
    These are related by the Hilbert transform $\mathfrak{H}: \hf \to \hf$, defined by 
    \[
    \mathfrak{H}(u)= -\mathbf{i} \sum_{n > 0} u_n z^n + \mathbf{i} \sum_{n < 0} u_n z^n,
    \]
    which satisfies the identity $\langle u, v \rangle_{\hf} = \omega(u, \mathfrak{H}(v))$.
\end{definition}

Since $u$ is real-valued, $u_{-n} = \bar{u}_n$ for all integers $n$. Via the Poisson integral, every element $u \in \hf$ admits a unique harmonic extension over the unit disk with Dirichlet energy equal to $2 \pi \langle u, u \rangle_{\hf}$. 

\begin{proposition}
    The spaces $\hf$ and $\mathbf{HD}(\mathbb{D})$ are isomorphic modulo constants via the restriction map and the Poisson integral.
\end{proposition}

The space of smooth functions on the unit circle, $C^{\infty}(\partial \mathbb{D})$, is dense in $\hf$ (see \cite[Section 9.1]{Sergeev2010}). For $u,v \in C^{\infty}(\partial \mathbb{D})$, we can express the symplectic form and inner product as
\[
\omega(u,v) = \frac{1}{2\pi} \int_{\partial \mathbb{D}} u \frac{dv}{d\theta} \,d\theta = \frac{1}{2\pi} \int_{\partial \mathbb{D}} u \,dv
\]
and
\[
\langle u ,v \rangle_{\hf} = \frac{1}{2\pi} \int_{\partial \mathbb{D}} u \,d\mathfrak{H}(v).
\]

\subsubsection{Weil-Petersson class in the universal Teichm\"{u}ller space}

The universal Teichm\"{u}ller space, denoted by $\UT$, plays a central role in Teichm\"{u}ller theory as it contains the Teichm\"{u}ller spaces of all hyperbolic Riemann surfaces \cite{Lehto1987,Sergeev2014}. It is the collection of all \emph{sense-preserving quasisymmetric homeomorphisms} of the unit circle modulo M\"{o}bius transformations. A sense-preserving homeomorphism of the unit circle can be defined using quasiconformal mappings on the unit disk $\mathbb{D}$. 

\begin{definition}
    A homeomorphism $h: \partial \mathbb{D} \to \partial \mathbb{D}$ is sense-preserving and quasisymmetric if it can be extended to a quasiconformal homeomorphism $\tilde{h}:\mathbb{D} \to \mathbb{D}$; that is, there exists a constant $K \geq 1$ such that the Beltrami differential $\mu(z) = \partial_{\bar{z}}\tilde{h} / \partial_z\tilde{h}$ satisfies
    $$
    |\mu(z)| \leq \frac{K-1}{K+1} < 1
    $$
    almost everywhere in $\mathbb{D}$.
\end{definition}

Inside $\UT$, there is a distinguished subspace called the Weil-Petersson class, denoted by $\WT$ \cite{Takhtajan2005,Cui2000}. 
\begin{definition}
	The Weil-Petersson class $\WT \subset \UT$ is the collection of quasisymmetric homeomorphisms of the unit circle that can be extended to quasiconformal homeomorphisms of the unit disk whose Beltrami differential $\mu$ is $L^2$-integrable with respect to the hyperbolic metric, i.e.
\[
\iint_{\mathbb{D}} |\mu(z)|^2 \frac{dx dy}{(1-|z|^2)^2} < \infty.
\]
\end{definition}

Though the following result is not used in the paper, it motivates our conjecture that the mapping $\pi$ in Theorem \ref{thm:projectionWP} yields a homeomorphism between $P(\Theta,R^{\dagger})$ and $\WT$. 
\begin{proposition}[Shen \cite{Shen2018}]
	A sense-preserving homeomorphism $h$ on the unit circle belongs
	to the Weil-Petersson class if and only if $h$ is absolutely continuous and $\log h'$ belongs to the Sobolev class $H^{\frac{1}{2}}(\partial \mathbb{D})$.
\end{proposition}

\subsection{Discrete harmonic functions on graphs}

\subsubsection{Dirichlet-finite functions on graphs and equivalent edge weights} \label{sec:Dirichlet}

Let $c :E \to \mathbb{R}_{>0}$ be some positive edge weights. We define the Dirichlet energy of a function $u:V \to \mathbb{R}$ with respect to the edge weight $c$ as
\[\mathcal{E}_c(u) := \sum_{ij} c_{ij} (u_i-u_j)^2 \in [0, \infty].\]
A function has vanishing Dirichlet energy if and only if it is constant. Fixing an arbitrary vertex $o \in V$, we define the inner product for functions $u,v:V \to \mathbb{R}$
\[
\langle u,v \rangle_c = u_o v_o + \sum_{ij} c_{ij} (u_i-u_j)(v_i-v_j)
\]  
and the norm of $u$
\[
||u||_{c} := \sqrt{\langle u,u \rangle_c} \in [0, \infty].
\]
We denote by $\mathbf{D}_c(V)$ the collection of functions $u:V \to \mathbb{R}$ such that $||u||^2_c$ is finite. The space $\mathbf{D}_c(V)$ is known to be a Hilbert space \cite[Theorem 3.15] {Soardi1994}. This space contains a subspace of functions with finite support, whose closure is denoted by $\mathbf{D}_{c,0}(V) \subset \mathbf{D}_c(V)$. 
\begin{definition}
	Two positive edge weights $c,\tilde{c} :E \to \mathbb{R}_{>0}$ are said to be equivalent if there exists a constant $K>0$ such that for all edges $ij \in E$
	\[
	\frac{c_{ij}}{K} < \tilde{c}_{ij} \leq K c_{ij}.
	\]
\end{definition}

The statement below follows directly from the definitions. The proof is omitted.
\begin{proposition}\label{prop:equivalentweight}
	Suppose $c,\tilde{c}$ are equivalent edge weights, then the norms $||\cdot||_c$ and $||\cdot||_{\tilde{c}}$ are equivalent, i.e. there exists $K > 0$ such that for $u:V \to \mathbb{R}$  
	\[
	\frac{||u||_{c}}{K} \leq ||u||_{\tilde{c}} \leq K ||u||_{c}.
	\]
	This implies that $\mathbf{D}_c(V)=	\mathbf{D}_{\tilde{c}}(V)$  and $\mathbf{D}_{c,0}(V)=\mathbf{D}_{\tilde{c},0}(V)$. Furthermore, the graph is transient with respect to $c$ if and only if it is transient with respect to $\tilde{c}$.
\end{proposition}

	The choice of the vertex $o$ defining the inner product does not matter. One can show that different choices yield equivalent norms.

Associated with the edge weight $c$, we consider the graph Laplacian $\Delta_c: \mathbf{D}_c(V) \to \mathbf{D}_c(V)$ defined by
\[
(\Delta_c u)_i = \sum_j c_{ij} (u_j - u_i).
\]
When the graph is transient, we have the Royden decomposition \cite[Theorem 3.69] {Soardi1994}
\begin{equation}\label{eq:Royden}
	\mathbf{D}_c(V) = \mathbf{D}_{c,0}(V) \oplus \mathbf{HD}_{c}(V).
\end{equation}
Here $\mathbf{HD}_{c}(V):= \ker \Delta_c \subset \mathbf{D}_c(V)$ is the space of discrete harmonic functions with respect to the edge weight $c$, i.e. a function $u \in \mathbf{HD}_{c}(V)$ if and only if
\[
\sum_j c_{ij} (u_j - u_i) =0 \quad \forall i \in V
\]
where the sum is over all neighbors $j$ of $i$. The direct sum in \eqref{eq:Royden} is orthogonal with respect to a bilinear form
\[
\mathcal{E}_c(u,v) = \sum_{ij} c_{ij}  (u_i-u_j)(v_i-v_j),
\]
which is a semi-inner product since $\mathcal{E}_c(u,u)=0$ if and only if $u$ is constant. Thus, modulo constant functions, the bilinear form $\mathcal{E}_c(\cdot, \cdot)$ descends to an inner product on the quotient space $\mathbf{D}_{c}(V)/\mathbb{R}$. 

The constant edge weight $c \equiv  1$ is called the combinatorial edge weight. It is often used to study the combinatorics of the graph.

\begin{definition}
	With the combinatorial edge weight $c\equiv  1$, we define the space of Dirichlet-finite functions $\mathbf{D}(V):=	\mathbf{D}_c(V)$ as well as closed subspace $\mathbf{D}_{0}(V):=\mathbf{D}_{c,0}(V)$ and 	$\mathbf{HD}(V):= \mathbf{HD}_{c}(V)$.
\end{definition}

As we show in the following sections, every circle pattern induces a natural edge weight $c$ and a graph Laplacian. Such geometric edge weights will be shown to be equivalent to the combinatorial weight. 
\subsubsection{Conjugate discrete harmonic functions}

Given a topological cell decomposition $(V,E,F)$ of a disk, there is a dual cell decomposition $(V^*,E^*,F^*)$ with a natural bijection between primal faces and dual vertices ($F \cong V^*$), primal vertices and dual faces ($V \cong F^*$), and primal and dual edges ($E \cong E^*$). 
We adopt the following index conventions:
\begin{itemize}
    \item Primal vertices in $V$ are indexed by Latin letters $i, j, k, \dots$.
    \item Dual vertices in $V^*$ (corresponding to faces in $F$) are indexed by Greek letters $\phi, \psi, \varphi, \dots$.
    \item A primal edge connecting $i,j$ is denoted by $ij$, and the corresponding dual edge connecting $\phi, \psi$ is denoted by $\phi\psi$.
\end{itemize}

Suppose the primal graph is equipped with edge weights $c: E \to \mathbb{R}_{>0}$, denoted by $c_{ij}=c_{ji}$. We define the dual edge weights $c^*: E^* \to \mathbb{R}_{>0}$ on the dual graph by the reciprocal of the primal weights
\[
c^*_{\phi\psi} := \frac{1}{c_{ij}},
\]
where $\phi\psi$ is the dual edge crossing $ij$. The graph Laplacian on the dual graph with respect to $c^*$ is defined for functions $u: V^* \to \mathbb{R}$ by
\[
(\Delta_{c^*} u)_\phi = \sum_{\psi} c^*_{\phi\psi} (u_\psi - u_\phi)
\]
where the sum is over all neighbors $\psi$ of $\phi$ in the dual graph.

Analogous to the primal case, we define the space of Dirichlet-finite functions $\mathbf{D}_{c^*}(F)$ as the space of functions on the dual vertices with finite Dirichlet energy with respect to the dual edge weight $c^*$. We also define the closed subspace $\mathbf{D}_{c^*,0}(F)$ as the closure of finitely supported functions, and $\mathbf{HD}_{c^*}(F)$ as the space of discrete harmonic functions with respect to $\Delta_{c^*}$.

There is a duality between harmonic functions on the primal and dual graphs. A function $v: V \to \mathbb{R}$ is harmonic with respect to $c$ if and only if there exists a conjugate function $u: V^* \to \mathbb{R}$ that is harmonic with respect to $c^*$. These functions are related by the discrete Cauchy-Riemann equations
\[
u_\psi - u_\phi = c_{ij} (v_j - v_i),
\]
where the orientation is chosen such that $v_j-v_i$ corresponds to traversing the edge $ij$, and $u_\psi-u_\phi$ corresponds to traversing the dual edge $\phi\psi$ crossing from left to right. Importantly, they further satisfy the identity
\[\mathcal{E}_c(v,v) = \mathcal{E}_{c^*}(u,u).\] 
This establishes an isomorphism between the spaces of primal and dual harmonic functions with finite Dirichlet energy modulo constant functions
\[
\mathbf{HD}_{c}(V)/\mathbb{R} \cong \mathbf{HD}_{c^*}(F)/\mathbb{R}.
\]

\subsection{Circle patterns in the Euclidean plane}

Let $G = (V, E, F)$ be a cell decomposition of a disk.  
A \emph{circle pattern} is a realization of the vertices in the Euclidean plane such that each face has a circumcircle passing through its vertices.  
Up to Euclidean transformations, a circle pattern is uniquely determined by the Euclidean radii of the circles together with the intersection angles of neighbouring circles.  

\subsubsection{Intersection angles of circle patterns}

Suppose a circle pattern is given. For every edge, we measure the intersection angle between the circumcircles of its two adjacent faces, thereby defining a map
\[
\Theta : E^* \to [0, 2\pi).
\]
For an oriented edge $ij$, the angle $\Theta_{\phi\psi}$ is measured from the counterclockwise tangent of $C_{\phi}$ to the counterclockwise tangent of $C_{\psi}$ at the endpoint $j$, where $C_{\phi}$ and $C_{\psi}$ denote the circumcircles of the faces $\phi, \psi$ adjacent to $ij$, respectively.  
When $\Theta_{\phi\psi} = \pi$, the two circles coincide but have opposite orientations, and the corresponding faces are folded up.  
When $\Theta_{\phi\psi} = 0$, the two circles coincide with the same orientation. In this case, we remove the common edge to merge the two faces.  
Without loss of generality, we assume that all intersection angles $\Theta_{\phi\psi}$ lie in $(0, 2\pi)$.

The intersection angles $\Theta$ define a \emph{discrete conformal structure}: two circle patterns are said to have the same discrete conformal structure if their intersection angles agree. 

A circle pattern is called \emph{Delaunay} if all intersection angles lie in $(0, \pi)$.  
This condition is equivalent to the local convexity of the associated pleated surface in hyperbolic space.

\subsubsection{Euclidean radii of circle patterns}

Given the intersection angles $\Theta$, we review the necessary and sufficient conditions for a function $R : F \to \mathbb{R}_{>0}$ to be the Euclidean radii for some circle pattern in the plane. We also explain the construction of the associated circle pattern metric $\sigma(\Theta, R)$.
	\begin{figure}
		\centering
			\begin{minipage}{0.535\textwidth}
		\centering
		\includegraphics[width=0.99\textwidth]{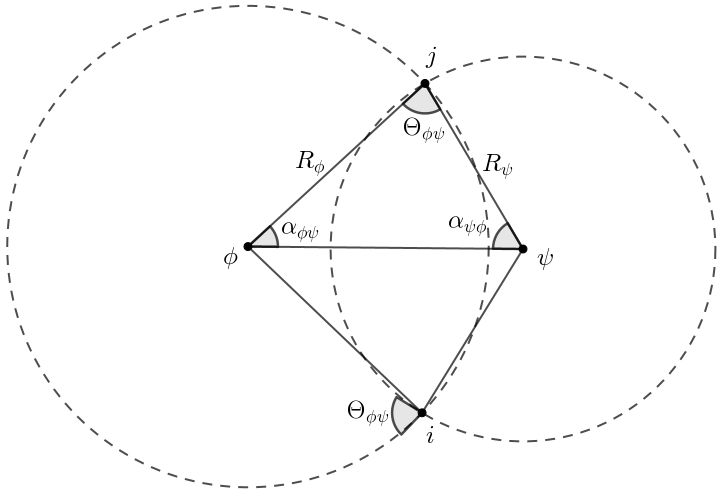}
	\end{minipage}
	\begin{minipage}{0.455\textwidth}
		\centering
		\includegraphics[width=0.99\textwidth]{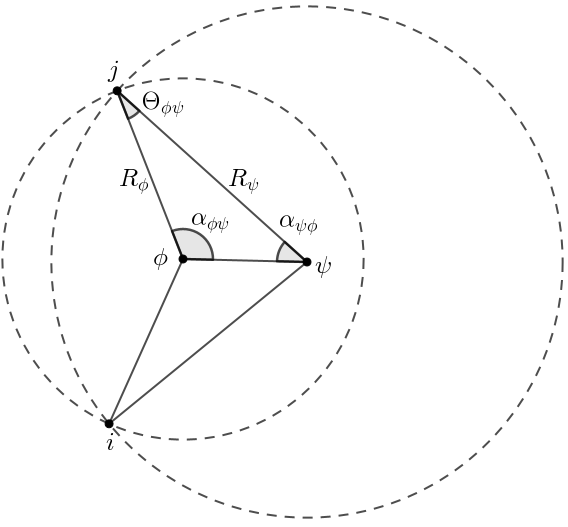}
	\end{minipage}
    \caption{
	Every dual edge $\phi \psi$ is associated with a kite formed by the two circumcenters of the faces $\phi, \psi$ and the two intersection points of their circumcircles. The side lengths of the kite are the radii $R_{\phi}, R_{\psi}$, and the intersection angle between the circles is $\Theta_{\phi\psi}$. We denote the half-angles at the circumcenters by $\alpha_{\phi \psi}$ and $\alpha_{\psi \phi}$.}
\label{fig:angle_at_center}
	\end{figure}

For each unoriented dual edge $\phi \psi$, consider the two circles of radii $R_{\phi}$ and $R_{\psi}$ corresponding to the two faces $\phi, \psi$ and intersect at angle $\Theta_{\phi\psi}=\Theta_{\psi\phi}$.  
The two intersection points together with the circumcenters determine a Euclidean kite, in which two consecutive sides have length $R_{\phi}$ and the other two consecutive sides have length $R_{\psi}$ (See Figure \ref{fig:angle_at_center}).  
Let $\alpha_{\phi \psi}$ and $\alpha_{\psi \phi}$ denote the half-angles at the circumcenters of the faces $\phi$ and $\psi$, respectively.  
They satisfy
\[
\Theta_{\phi\psi} = \pi - \alpha_{\phi \psi} - \alpha_{\psi \phi}.
\]
The half-angles are determined by the radii and intersection angles.

\begin{lemma}\label{lemma:cotweight}
	Suppose the two circles with radii $R_{\phi},R_{\psi}$ intersect with angle $\Theta_{\phi\psi} \in (\epsilon_0,\pi-\epsilon_0)$. They determine the two half-angles $\alpha_{\phi \psi}, \alpha_{\psi \phi} \in (0,\pi)$ at the circumcenters via 
\begin{align*}
	 \alpha_{\phi \psi} &= \cot^{-1}\left(\frac{R_{\phi} - R_{\psi} \cos \Theta_{\phi\psi}}{R_{\psi} \sin \Theta_{\phi\psi}}\right)=\frac{1}{2\ii} \ln \left( \frac{1 - \frac{R_\psi}{R_\phi} e^{-\ii\Theta_{\phi \psi}}}{1 - \frac{R_\psi}{R_\phi} e^{\ii\Theta_{\phi \psi}}} \right), \\
	\alpha_{\psi \phi} &= \cot^{-1}\left(\frac{R_{\psi} - R_{\phi} \cos \Theta_{\phi\psi}}{R_{\phi} \sin \Theta_{\phi\psi}}\right)=\frac{1}{2\ii} \ln \left( \frac{1 - \frac{R_\phi}{R_\psi} e^{-\ii\Theta_{\phi \psi}}}{1 - \frac{R_\phi}{R_\psi} e^{\ii\Theta_{\phi \psi}}} \right).
\end{align*}
\end{lemma}
\begin{proof}
	We have 
	\[
	R_{\phi} \sin \alpha_{\phi \psi} = R_{\psi} \sin \alpha_{\psi \phi} = R_{\psi} \sin(\Theta_{\phi\psi}+\alpha_{\phi \psi})
	\]
	and hence
	\[
	\cot \alpha_{\phi \psi} = \frac{R_{\phi} -R_{\psi} \cos \Theta_{\phi\psi}}{R_{\psi} \sin \Theta_{\phi\psi}}.
	\]
    On the other hand, we make use of complex numbers. The vector sum of the segments $\phi j$ and $j \psi$ is the same as the vector sum of $\phi i$ and $i \psi$, which yields the identity
	\[
R_\phi e^{\ii\alpha_{\phi \psi}} + R_\psi e^{\ii(\alpha_{\phi \psi}+\Theta_{\phi \psi}-\pi)}= R_\phi e^{-\ii\alpha_{\phi \psi}} + R_\psi e^{-\ii(\alpha_{\phi \psi} + \Theta_{\phi \psi}-\pi)}.
	\]
	It leads to another formulation of the angle $\alpha_\phi$.
\end{proof}

Given an arbitrary function $R : F \to \mathbb{R}_{>0}$, one can construct a piecewise Euclidean metric $\sigma$ by gluing these kites along their edges, which might exhibit conical singularities at intersection points of circles and circumcenters. We observe that there is no conical singularity at intersection points of circles if and only if the sum of relevant intersection angles $\Theta$'s is $2\pi$, which has been imposed in the assumptions of $\Theta$. To ensure that there is no conical singularity at a circumcenter, the angle sum at the centers has to be $2\pi$ and it imposes constraints on the radii. 

Explicitly, for a face $\phi$ with circumradius $R_\phi$, let $R_\psi$ be the circumradii of its neighbouring faces $\psi$, and let $\Theta_{\phi\psi}$ be the intersection angles of their circumcircles. The angle sum at the circumcenter of $\phi$ being $2\pi$ is equivalent to the sum of half-angles being $\pi$, i.e.
\begin{align}\label{eq:nonuR}
	\sum_{\psi } \cot^{-1}\!\left( \frac{R_\phi - R_\psi \cos \Theta_{\phi\psi}}{R_\psi \sin \Theta_{\phi\psi}} \right) = \pi
\end{align}
where the sum is over all neighboring faces $\psi$ of $\phi$.

\begin{proposition}
Given the intersection angles $\Theta$ satisfying Definition \ref{def:infintheta}, a function $R: F \to \mathbb{R}_{>0}$ determines a smooth Euclidean metric $\sigma(\Theta, R)$ supporting a circle pattern with radii $R$ and intersection angles $\Theta$ if and only if $R$ satisfies \eqref{eq:nonuR} for every face $\phi \in F$. 
\end{proposition}

The collection of functions $R:F \to \mathbb{R}_{>0}$ that satisfies \eqref{eq:nonuR} for every face $\phi \in F$ is denoted by $\tilde{P}(\Theta)$. Under the developing map, the circle patterns may or may not be embedded in the plane.

\subsubsection{Infinitesimal deformations of circle patterns}\label{sec:infinitesimal_deformations}

We consider infinitesimal deformations of circle patterns preserving the intersection angles $\Theta$ by analyzing the linearization of Equation \eqref{eq:nonuR} with respect to the radii $R$ \cite{Lam2015a}. Such an infinitesimal deformation corresponds to a first order change $\dot{R}$ of the radii $R$ satisfying the linearized equation. It shall be more convenient to consider the logarithmic change of the radii $u: F \to \mathbb{R}$ defined as
\[
u_\phi = \frac{\dot{R}_\phi}{R_\phi}.
\]
If $R(t)$ is a smooth $1$-parameter family of functions in $\tilde{P}(\Theta)$, then the logarithmic change of the radii is given by
\[
u_\phi = \frac{d}{dt}\bigg|_{t=0} \log R_\phi(t).
\]

\begin{proposition}\label{prop:infincot}
Given a circle pattern corresponding to $R \in \tilde{P}(\Theta)$. A function $u: F \to \mathbb{R}$ is the logarithmic change of the radii under an infinitesimal deformation of circle patterns preserving the intersection angles $\Theta$ if and only if it satisfies for every face $\phi \in F$,
\begin{equation}\label{eq:linearu_inf}
    \sum_{\psi \sim \phi} c_{\phi\psi}^* (u_\phi - u_\psi) = 0,
\end{equation}
where $c_{\phi\psi}^*$ are the dual edge weights defined as
\[
c_{\phi\psi}^* = \frac{2}{\cot \alpha_{\phi \psi} + \cot \alpha_{\psi \phi}} = \frac{2 R_{\phi} R_{\psi} \sin \Theta_{\phi\psi}}{R_{\phi}^2 + R_{\psi}^2 - 2 R_{\phi} R_{\psi} \cos \Theta_{\phi\psi}}.
\]
\end{proposition}

The dual edge weights $c_{\phi\psi}^*$ are the reciprocals of the \emph{cotangent weights}
\[
c_{\phi\psi} = \frac{\cot \alpha_{\phi \psi} + \cot \alpha_{\psi \phi}}{2} = \frac{1}{c_{\phi\psi}^*},
\]
which is associated with the graph Laplacian derived from a finite-element discretization \cite{Pinkall1993}. The function $u$ satisfying \eqref{eq:linearu_inf} everywhere is a discrete harmonic function on the dual graph with respect to the dual edge weights $c^*$.

\subsubsection{Volume of hyperbolic polyhedra associated with edges}

An essential tool for obtaining solutions to Equation \eqref{eq:nonuR} is a variational method via the hyperbolic volume of polyhedra in hyperbolic 3-space.

For every edge with adjacent circles of radii $R_{\phi}, R_{\psi}$ intersecting at an angle $\Theta_{\phi\psi}$, we associate a hyperbolic polyhedron (Figure \ref{fig:hypvolume}) as follows. We consider the upper half-space model of hyperbolic 3-space $\mathbb{H}^3$ with boundary $\mathbb{R}^2 \cup \{\infty\}$. We place the two circles in the plane $\mathbb{R}^2$ and denote their intersection points by $z_i, z_j$. These circles in the plane act as the boundaries of two hemispheres in the upper half-space. We denote by $\tilde{z}_\phi, \tilde{z}_\psi \in \mathbb{H}^3$ the vertical lifts of the respective Euclidean circumcenters to these hemispheres. This yields two hyperbolic tetrahedra: the first has three ideal vertices $z_i, z_j, \infty$ and one finite vertex $\tilde{z}_\phi$, while the second has the same ideal vertices but finite vertex $\tilde{z}_\psi$. Gluing these two tetrahedra along the common ideal triangle $z_i z_j \infty$ forms a hyperbolic polyhedron whose signed hyperbolic volume is given by
\[
\mathcal{V}(R_{\phi}, R_{\psi}, \Theta_{\phi\psi}) = \Lambda\left(\alpha_{\phi \psi}\right) + \Lambda\left(\alpha_{\psi \phi}\right),
\]
where $\Lambda(x) = -\int_0^x \log |2\sin t| \, dt$ is Milnor's Lobachevsky function \cite{Milnor1994}, and the half-angles $\alpha_{\phi \psi}, \alpha_{\psi \phi}$ are determined by $R_{\phi}, R_{\psi}$, and $\Theta_{\phi\psi}$ as in Lemma \ref{lemma:cotweight}. 
\begin{figure}
	\centering
	\includegraphics[width=0.5\textwidth]{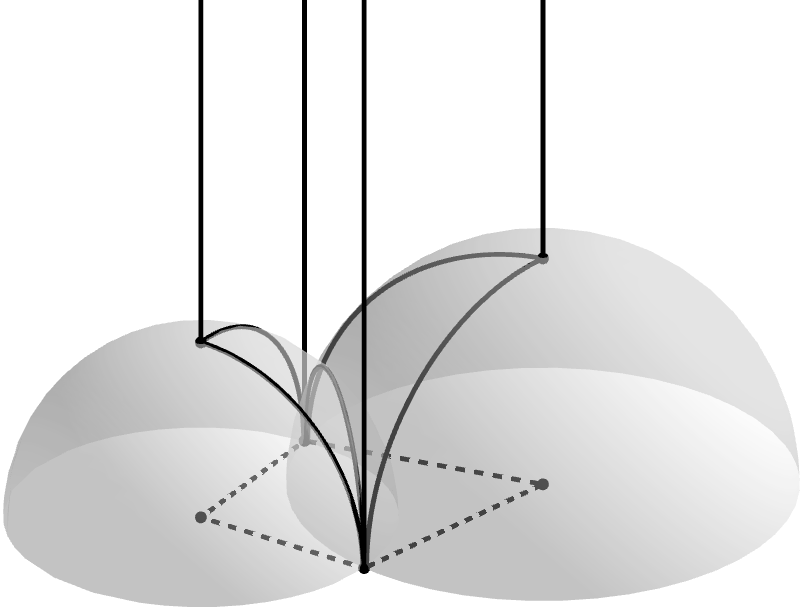}
	\caption{The hyperbolic polyhedron associated with an edge is obtained by gluing two hyperbolic tetrahedra. The signed hyperbolic volume of the polyhedron is given by the sum of the Lobachevsky functions of the half-angles at the circumcenters.}
	\label{fig:hypvolume}
\end{figure}

Although the total hyperbolic volume associated with an infinite circle pattern is infinite, we can still consider the relative volume between two such circle patterns. By incorporating a slight modification to ensure Fréchet differentiability, this relative volume yields a well-defined functional $\mathcal{W}:\mathbf{D}(V) \to \mathbb{R}$ (see Definition \ref{def:Wfunctional}). In Section \ref{sec:Wstar}, we also introduce a functional $\mathcal{W}^*: \mathbf{D}(F) \to \mathbb{R}$ (Definition \ref{def:Wstarfunctional}), which is related to $\mathcal{W}$ via a Legendre transform and connects to the Bloch-Wigner dilogarithm \cite{NeumannZagier1985}. We refer the reader to the work of Bobenko and Springborn \cite{Bobenko2004} for a detailed discussion of this relationship in the finite setting.

\section{Geometric edge weights equivalent to combinatorial edge weights}\label{sec:circle_patterns_wp}
In order to study the deformation space of circle patterns, we consider geometric edge weights induced from circle patterns as motivated by Proposition \ref{prop:infincot}. The goal of this section is to establish the equivalence between the geometric edge weights and the combinatorial edge weights using the Ring lemma, so that the discussion in Section \ref{sec:Dirichlet} is applicable. 
	
	\begin{definition}\label{def:geoedgeweight}
		Given a uniformized circle pattern with Euclidean radii $R^{\dagger}$ and intersection angles $\Theta$ satisfying the assumptions (A1) and (A2). For every $u \in \mathbf{D}(F)$ we define the geometric edge weight $c:E \to \mathbb{R}_{>0}$ on the primal graph by
	    	\begin{align*}
				c_{ij} &= \frac{R_{\phi}^2 + R_{\psi}^2 - 2 R_{\phi} R_{\psi} \cos \Theta_{\phi\psi}}{2 R_{\phi} R_{\psi} \sin \Theta_{\phi\psi}} \\
 &= \tan \frac{\Theta_{\phi\psi}}{2} + \frac{(R_{\phi}-R_{\psi})^2}{2 R_{\phi} R_{\psi} \sin \Theta_{\phi\psi}} \\
&= \frac{\cot \alpha_{\phi \psi} + \cot \alpha_{\psi \phi}}{2}
			\end{align*}
		where $R = e^{u} R^{\dagger}$ while $\alpha_{\phi \psi}$ and $\alpha_{\psi \phi}$ are the half-angles in the associated kites (See Figure \ref{fig:angle_at_center}). We also define the dual edge weight $c^*:E^* \to \mathbb{R}_{>0}$ by
		\[c^*_{\phi\psi} := \frac{1}{c_{ij}},\]
		where $\phi\psi$ is the dual edge crossing $ij$.
	\end{definition}

The Ring lemma is the main tool to establish various estimates throughout the paper. It is an inequality relating the Euclidean radii of neighbouring circles, which is known to hold for circle packings \cite{Stephenson2005,He1999} and has been extended to circle patterns with prescribed intersection angles $\Theta$ satisfying (A1) and (A2) \cite{Ge2025}.

\begin{proposition}[Ring lemma \cite{Stephenson2005,Ge2025}]\label{lem:ring}
		Given the angle function $\Theta$ satisfying the assumptions (A1) and (A2). There exists a constant $C=C(\Theta, \epsilon_0)>1$ such that for any embedded circle pattern with intersection angle $\Theta$, the Euclidean radii $R:F \to \mathbb{R}_{>0}$ satisfies for every dual edge $\phi\psi$
		\[
		\frac{1}{C} \leq \frac{R_{\phi}}{R_{\psi}} \leq C.
		\]
\end{proposition}

We deduce an extension of the Ring lemma to more general circle patterns, including those that may not be embedded or whose associated piecewise Euclidean metric contains conical singularities.
\begin{corollary}\label{cor:generalring}
	There exists a constant $C=C(\Theta, \epsilon_0)>1$ such that for any circle pattern corresponding to $u \in \mathbf{D}(F)$, its Euclidean radii $R= e^{u} R^{\dagger}$ satisfy for every dual edge $\phi\psi$
		\[
		\frac{1}{C e^{\sqrt{\mathcal{E}(u)}}} \leq \frac{1}{C e^{\sup_{\phi\psi \in E^{*}} |u_{\phi}-u_{\psi}|}} \leq \frac{e^{u_{\phi}}R^{\dagger}_{\phi}}{e^{u_{\psi}}R^{\dagger}_{\psi}} \leq C e^{\sup_{\phi\psi \in E^{*}} |u_{\phi}-u_{\psi}|} \leq C e^{\sqrt{\mathcal{E}(u)}}
		\]
	where $\mathcal{E}(u)= \sum_{\phi\psi \in E^{*}} |u_{\phi}-u_{\psi}|^2$ is the Dirichlet energy of $u$.
\end{corollary}
\begin{proof}
Applying the Ring lemma to the uniformized circle pattern $R^{\dagger}$, we have $\frac{1}{C} \leq \frac{R^{\dagger}_{\phi}}{R^{\dagger}_{\psi}} \leq C$. For any dual edge $\phi\psi$, the definition of the energy $\mathcal{E}(u)$ implies \[ |u_{\phi} - u_{\psi}| \leq \sup_{\phi\psi \in E^{*}} |u_{\phi}-u_{\psi}| \leq \sqrt{\mathcal{E}(u)}.\] Combining these, we deduce the claimed inequalities.	
\end{proof}

\begin{proposition}\label{prop:bound}
	For every $u \in \mathbf{D}(F)$, its geometric edge weights satisfy
	\[
 \tan \frac{\epsilon_0}{2} \leq c_{ij} \leq \tan \frac{\pi- \epsilon_0}{2} + \frac{4C^2 e^{2\sup_{\phi\psi \in E^{*}} |u_{\phi}-u_{\psi}|}}{ \sin \epsilon_0} \leq \tan \frac{\pi- \epsilon_0}{2} + \frac{4C^2 e^{2 \sqrt{\mathcal{E}(u)}}}{ \sin \epsilon_0}.
	\]
On the other hand, for any half-angle at centers
	\[
	|\cot \alpha_{ij,\phi}| \leq \frac{e^{\sqrt{\mathcal{E}(u)}}C+1}{\sin \epsilon_0} 
	\]
	and
	\[
	0< \eta < \alpha_{ij,\phi} < \pi - \eta
	\]
	where \[
	\eta := \cot^{-1} (\frac{e^{\sqrt{\mathcal{E}(u)}}C+1}{\sin \epsilon_0})
	\] 
	is a small constant.
\end{proposition}
\begin{proof}
	It follows from the definition that
	\[
	c_{ij} \geq  \tan \frac{\Theta_{\phi\psi}}{2}  \geq \tan \frac{\epsilon_0}{2}.
	\]
	The equality is attained if and only if the two circles have the same radii. On the other hand, Corollary \ref{cor:generalring} implies that 
	\[
	c_{ij} =\tan \frac{\Theta_{\phi\psi}}{2} + \frac{(R_{\phi}-R_{\psi})^2}{2 R_{\phi} R_{\psi} \sin \Theta_{\phi\psi}} \leq \tan \frac{\pi- \epsilon_0}{2} + \frac{4C^2 e^{2\sup_{\phi\psi \in E^{*}} |u_{\phi}-u_{\psi}|}}{ \sin \epsilon_0}.
	\]
	On the other hand, from Lemma \ref{lemma:cotweight}, 
	\[
		|\cot \alpha_{ij,\phi}| = |\frac{\frac{R_\phi}{R_\psi} - \cos \Theta_{\phi\psi}}{ \sin \Theta_{\phi\psi}}| <\frac{e^{\sqrt{\mathcal{E}(u)}}C+1}{\sin \epsilon_0}.
	\]
	The estimate on the angles $\alpha$ follows from the fact that $\cot(\cdot)$ is monotone decreasing on $(0,\pi)$.
\end{proof}
To conclude, we relate the geometric edge weights to the combinatorial edge weights which are constantly equal to $1$.
\begin{corollary}\label{cor:equivalence}
	For every $u \in \mathbf{D}(F)$, its geometric edge weights $c$ and $c^*$ are respectively equivalent to the combinatorial edge weights.
\end{corollary}

\section{Hilbert manifold $P(\Theta,R^{\dagger})$} \label{sec:hilbert}

Given a uniformized circle pattern with intersection angle $\Theta$ and radius $R^{\dagger}$, we verify Theorem \ref{thm:homHDF} that the space $P(\Theta,R^{\dagger})$ is a Hilbert submanifold in $\mathbf{D}(F)$ and homeomorphic to the Hilbert vector space $\mathbf{HD}(F)$. Furthermore, the tangent space at $u \in P(\Theta,R^{\dagger})$ is $T_{u} P(\Theta,R^{\dagger}) = \mathbf{HD}_{c^*}(F)$, the space of discrete harmonic functions with respect to the geometric edge weight
\[
c^*_{\phi \psi} = \frac{2 e^{u_{\phi}+u_{\psi}} R_{\phi} R_{\psi} \sin \Theta_{\phi\psi}}{e^{2u_{\phi}} R_{\phi}^2 + e^{2u_{\psi}} R_{\psi}^2 - 2 e^{u_{\phi}+u_{\psi}}  R_{\phi} R_{\psi} \cos \Theta_{\phi\psi}}.
\]
The main tool is a variational method involving the relative volume $\mathcal{W}^*$ as defined in Definition \ref{def:Wstarfunctional}. Throughout, the superscript $*$ emphasise the map is defined for the dual cell decomposition $(V^*,E^*,F^*)$. In later sections, analogous functions are denoted without superscript $*$ to indicate that they are defined for the cell decomposition $(V,E,F)$.

\subsection{Hilbert submanifold in $\mathbf{D}(F)$}

By Corollary \ref{cor:equivalence}, all the geometric edge weights $c^*$ induced from circle patterns in $P(\Theta,R^\dagger)$ are equivalent to the combinatorial edge weights. Hence, we have 
\[
\mathbf{D}_{c^*}(F) = \mathbf{D}(F), \quad \mathbf{D}_{c^*,0}(F) = \mathbf{D}_{0}(F).
\] 
We denote the inclusion map by
\[
\iota: \mathbf{D}_{0}(F) \to \mathbf{D}(F)
\]
and the induced mapping on the dual spaces
\[
\iota^{\star}: \mathbf{D}^{\star}(F) \to \mathbf{D}^{\star}_{0}(F),
\]
where $\mathbf{D}^{\star}(F)$ and $\mathbf{D}^{\star}_{0}(F)$ are respectively the spaces of bounded linear functionals on $\mathbf{D}(F)$ and $\mathbf{D}_{0}(F)$.

\begin{definition}
	For every oriented dual edge $\phi \psi \in \vec{E}^*$, we define the change of curvature $\mathcal{K}_{\phi \psi}^*: \mathbf{D}(F) \to \mathbb{R}$ 
	\begin{align*}
		\mathcal{K}^{*}_{\phi \psi}(u) =& 2\cot^{-1}\!\left( \frac{\frac{R^{\dagger}_\phi}{R^{\dagger}_\psi} - \cos \Theta_{\phi\psi}}{\sin \Theta_{\phi\psi}} \right) - 2\cot^{-1}\!\left( \frac{e^{u_{\phi}-u_{\psi}} \frac{R^{\dagger}_\phi}{R^{\dagger}_\psi} -   \cos \Theta_{\phi\psi}}{ \sin \Theta_{\phi\psi}} \right)\\
		=& 2\alpha^{\dagger}_{\phi \psi}- 2\alpha_{\phi \psi}
	\end{align*}
	where $\alpha^{\dagger}_{\phi \psi}$ and $\alpha_{\phi \psi}$ are respectively the half-angles at the circumcenter $\phi$ in the kites corresponding to the circle patterns with radii $R^{\dagger}$ and $R = e^{u} R^{\dagger}$.
\end{definition}

\begin{lemma}\label{lem:gedge}
Given $u \in \mathbf{D}(F)$, there is a constant $C_1=C_1(\mathcal{E}(u),\Theta,\epsilon_0)> 0$ such that for every edge $\phi\psi \in E^*$
\[
|\mathcal{K}^{*}_{\phi \psi}(u)| \leq C_1 |u_{\phi}-u_{\psi}|.
\]
\end{lemma}
\begin{proof}
 Let $C$ be the constant in the Ring lemma \ref{lem:ring}. Given $u \in \mathbf{D}(F)$, we consider the analytic function for $t \in [-\sqrt{\mathcal{E}(u)},\sqrt{\mathcal{E}(u)}]$, $r \in [\frac{1}{C},C]$ and $\theta \in [\epsilon_0,\pi-\epsilon_0]$
\[
p(t,r,\theta) = 2\cot^{-1}\left(\frac{r - \cos \theta}{\sin \theta}\right) - 2\cot^{-1}\left(\frac{e^{t} r - \cos \theta}{\sin \theta}\right).
\]
We observe that $p(0,r,\theta) = 0$ for all $r$ and $\theta$. The mean value theorem yields
\[
p(t,r,\theta)= p(t,r,\theta)-p(0,r,\theta) = \partial_t p(\xi,r,\theta) \cdot t
\]
for some $\xi \in [0,t]$. 

Since the domain of $p$ is compact, the partial derivative is bounded
\[
C_1 := \sup_{t \in [-\sqrt{\mathcal{E}(u)},\sqrt{\mathcal{E}(u)}], r \in [\frac{1}{C},C], \theta \in [\epsilon_0,\pi-\epsilon_0]} |\partial_t p(t,r,\theta)|
\]
and hence we have
\[
|p(t,r,\theta)| \leq C_1 |t|.
\]
For any dual edge $\phi \psi \in E^*$, we have $|u_{\phi}-u_{\psi}| \leq \sqrt{\mathcal{E}(u)}$, $\frac{R^\dagger_\phi}{R^\dagger_\psi} \in [\frac{1}{C},C]$ and $\Theta_{\phi\psi} \in [\epsilon_0,\pi-\epsilon_0]$, we deduce the claimed inequality via
\[
|p(u_{\phi}-u_{\psi},\frac{R^\dagger_\phi}{R^\dagger_\psi},\Theta_{\phi\psi})| \leq C_1 |u_{\phi}-u_{\psi}|.
\] 
\end{proof}

\begin{proposition}\label{prop:gradient}
	For every $u \in \mathbf{D}(F)$ and oriented edge $\phi \psi \in \vec{E}^*$, we have
	\[
	\mathcal{K}^{*}_{\phi \psi}(u) = - \mathcal{K}^{*}_{\psi \phi}(u).
	\] 
   There is a well-defined map $\mathcal{K}^*:\mathbf{D}(F) \to \mathbf{D}^*(F)$ such that for every $u,v \in \mathbf{D}(F)$
	\[
	\langle \! \langle\mathcal{K}^*(u), v \rangle \! \rangle := \frac{1}{2} \sum_{\phi \psi \in \vec{E}^*} \mathcal{K}^{*}_{\phi \psi}(u) (v_{\phi}-v_{\psi}) = \sum_{\phi \psi \in E^*} \mathcal{K}^{*}_{\phi \psi}(u) (v_{\phi}-v_{\psi})
	\]
	where $\langle \! \langle \cdot, \cdot \rangle \! \rangle$ is the duality pairing between $\mathbf{D}^*(F)$ and $\mathbf{D}(F)$.

    Furthermore, we have $u \in P(\Theta,R^{\dagger})$ if and only if $\iota^{\star} \circ\mathcal{K}^*(u)=0$ is trivial in $\mathbf{D}^*_{0}(F)$.
\end{proposition}
\begin{proof}
We first prove the antisymmetry $\mathcal{K}^*_{\phi \psi}(u) = - \mathcal{K}^*_{\psi \phi}(u)$ for every dual edge $\phi \psi \in \vec{E}^*$. Observe that for the corresponding kite with adjacent circumcenters $\phi$ and $\psi$, the inner angles of the triangles $\phi \psi i$ for the circles with radii $e^{u_{\phi}}R^{\dagger}_{\phi}$ and $e^{u_{\psi}}R^{\dagger}_{\psi}$ are given by 
\[
\alpha_{\phi \psi}, \quad \alpha_{\psi \phi} \quad \Theta_{\phi\psi}.
\]
Since the sum of the inner angles of a triangle is $\pi$, we have for any $u_{\phi},u_{\psi}$
\begin{align*}
\mathcal{K}^*_{\phi \psi}(u) + \mathcal{K}^*_{\psi \phi}(u) &= 2(\alpha^{\dagger}_{\phi \psi} - \alpha_{\phi \psi}) + 2(\alpha^{\dagger}_{\psi \phi} - \alpha_{\psi \phi}) \\
&= 2(\alpha^{\dagger}_{\phi \psi} + \alpha^{\dagger}_{\psi \phi}) - 2(\alpha_{\phi \psi} + \alpha_{\psi \phi}) \\
&= 2(\pi - \Theta_{\phi\psi}) - 2(\pi - \Theta_{\phi\psi}) =0.    
\end{align*}
Thus, we deduce the claim that $\mathcal{K}^*_{\phi \psi}(u) = - \mathcal{K}^*_{\psi \phi}(u)$.

It is obvious from the definition that $\mathcal{K}^*(u)$ is a linear functional on $\mathbf{D}(F)$ and it remains to show that it is bounded. Lemma \ref{lem:gedge} indicates that there exists a constant $C_1=C_1(\mathcal{E}(u),\Theta,\epsilon_0)>0$ such that for every dual edge $\phi \psi \in E^*$
\begin{equation*}
|\mathcal{K}^*_{\phi \psi}(u)| \leq C_1 |u_{\phi}-u_{\psi}|.
\end{equation*} 
Hence, by the Cauchy-Schwarz inequality, for any $v \in \mathbf{D}(F)$, we have
\begin{align*}
|\langle \! \langle\mathcal{K}^*(u), v \rangle \! \rangle| &= |\sum_{\phi \psi \in E^*} \mathcal{K}^{*}_{\phi \psi}(u) (v_{\phi}-v_{\psi})| \\
&\leq C_1 \sum_{\phi \psi \in E^*} |u_{\phi}-u_{\psi}| |v_{\phi}-v_{\psi}| \\
&\leq C_1 \sqrt{\mathcal{E}(u)} \sqrt{\mathcal{E}(v)}
\end{align*}
which implies that $\mathcal{K}^*(u) \in \mathbf{D}^{\star}(F)$.

Finally, we observe that $u \in P(\Theta,R^{\dagger})$ if and only if for every face $\phi \in F$
\[\sum_{\psi} \mathcal{K}^*_{\phi \psi}(u) =0.\]
This is equivalent to saying that for every indicator function associated with a face $\phi \in F$ 
\[
\chi^{\phi}_{\psi} := \begin{cases}  1 & \text{ if } \psi = \phi \\ 0 & \text{ otherwise}
\end{cases}
\]
we have
\[\langle \! \langle\mathcal{K}^*(u), \chi^{\phi} \rangle \! \rangle = \sum_{\psi} \mathcal{K}^*_{\phi \psi}(u) =0.\]
Since the span of the indicator functions is dense in $\mathbf{D}_{0}(F)$, we deduce that $u \in P(\Theta,R^{\dagger})$ if and only if $\iota^{\star} \circ\mathcal{K}^*(u)$ is trivial in $\mathbf{D}^*_{0}(F)$.
\end{proof}

Consequently, the space of circle patterns is the zero set of $\iota^{\star} \circ \mathcal{K}^*$
\[
P(\Theta,R^{\dagger}) = (\iota^{\star} \circ \mathcal{K}^*)^{-1}(0).
\]
We further investigate the differential of the map $\mathcal{K}^*$.
\begin{proposition}\label{prop:selfadjoint}
    For every $u \in \mathbf{D}(F)$, the differential of $\mathcal{K}^*$ at $u$ is given by the linear map $d\mathcal{K}^*_u: \mathbf{D}(F) \to \mathbf{D}^*(F)$ such that for every $v,\tilde{v} \in \mathbf{D}(F)$
    \[
    \langle \! \langle d\mathcal{K}^*_u(v), \tilde{v} \rangle \! \rangle = \sum_{\phi \psi \in E^*} c^*_{\phi\psi} (v_{\phi}- v_{\psi})(\tilde{v}_{\phi}- \tilde{v}_{\psi}) =  \langle \! \langle v, d\mathcal{K}^*_u(\tilde{v}) \rangle \! \rangle 
    \]
    where $c^*_{\phi\psi}$ are the dual edge weights associated with the circle pattern with radii $R = e^{u} R^{\dagger}$ and intersection angles $\Theta$.
\end{proposition}
\begin{proof}
    For every oriented dual edge $\phi \psi \in \vec{E}^*$, we compute the partial derivative
    \begin{align*}
        \frac{\partial \mathcal{K}^*_{\phi \psi}}{\partial u_{\phi}} &= -\frac{\partial }{\partial u_{\phi}} \cot^{-1}\!\left( \frac{e^{u_{\phi}-u_{\psi}} \frac{R^{\dagger}_\phi}{R^{\dagger}_\psi} -   \cos \Theta_{\phi\psi}}{ \sin \Theta_{\phi\psi}} \right) \\
        &= \frac{2}{1+\left(\frac{e^{u_{\phi}-u_{\psi}} \frac{R^{\dagger}_\phi}{R^{\dagger}_\psi} -   \cos \Theta_{\phi\psi}}{ \sin \Theta_{\phi\psi}}\right)^2} \cdot \frac{e^{u_{\phi}-u_{\psi}} \frac{R^{\dagger}_\phi}{R^{\dagger}_\psi}}{\sin \Theta_{\phi\psi}}\\
        &= c^*_{\phi\psi}
    \end{align*}
    Similarly, we have
    \[
    \frac{\partial \mathcal{K}^*_{\phi \psi}}{\partial u_{\psi}} = -c^*_{\phi\psi}.
    \]
    Thus, for every $v,\tilde{v} \in \mathbf{D}(F)$, we have
    \begin{align*}
        \langle \! \langle d\mathcal{K}^*_u(v), \tilde{v} \rangle \! \rangle &=  \sum_{\phi \psi \in E^*}  \left(  \frac{\partial \mathcal{K}^*_{\phi \psi}}{\partial u_{\phi}} v_{\phi} +  \frac{\partial \mathcal{K}^*_{\phi \psi}}{\partial u_{\psi}} v_{\psi}  \right) (\tilde{v}_{\phi}-\tilde{v}_{\psi})\\
        &=  \sum_{\phi \psi \in E^*} c^*_{\phi\psi} (v_{\phi}- v_{\psi})(\tilde{v}_{\phi}- \tilde{v}_{\psi}).
    \end{align*}
    The self-adjointness of $d\mathcal{K}^*_u$ follows from the symmetry $c^*_{\phi\psi}= c^*_{\psi\phi}$.
\end{proof}

We are now in position to prove that $P(\Theta,R^{\dagger})$ is a Hilbert submanifold of $\mathbf{D}(F)$.

\begin{proposition} \label{prop:hilbert}
	The space $P(\Theta,R^{\dagger})$ is a Hilbert submanifold in the Hilbert space $\mathbf{D}(F)$.
\end{proposition}

\begin{proof}
    We investigate the differential of $\iota^{\star}\circ \mathcal{K}^*:\mathbf{D}(F) \to \mathbf{D}^*_{0}(F)$. Since $\iota^{\star}$ is linear, we have
    \[d(\iota^{\star} \circ \mathcal{K}^*)_u = \iota^{\star} \circ d\mathcal{K}^*_u.\]
    For any $u \in \mathbf{D}(F)$, the differential is given by for $v \in \mathbf{D}(F), \tilde{v} \in \mathbf{D}_{0}(F)$
    \begin{align*}
    	\langle \! \langle \iota^{\star} \circ d\mathcal{K}^*_u(v), \tilde{v} \rangle \! \rangle &= \langle \! \langle d\mathcal{K}^*_u(v), \tilde{v} \rangle \! \rangle = \sum_{\psi \sim \phi} c^*_{\phi\psi} (v_{\phi}- v_{\psi})(\tilde{v}_{\phi}- \tilde{v}_{\psi}).
    \end{align*}
    Taking $\tilde{v}$ to be the indicator function of a face $\phi \in F$, we deduce that
    \[\left(\iota^{\star} \circ d\mathcal{K}^*_u(v)\right)_{\phi} = \sum_{\psi} c^*_{\phi\psi} (v_{\phi}- v_{\psi}).\]
    Thus, the differential $\iota^{\star} \circ d\mathcal{K}^*_u$ coincides with the graph Laplacian associated with the geometric edge weights $c^*_{\phi\psi}$ and
    \[
    \ker \left(\iota^{\star} \circ d\mathcal{K}^*_u\right) = \mathbf{HD}_{c^*}(F).
    \]

    Next, we show that the image of $\iota^{\star} \circ d\mathcal{K}^*_u$ to $\mathbf{D}^*_{0}(F)$ is surjective, following the Riesz representation theorem. Since the geometric edge weight  $c^*_{\phi\psi}$ is equivalent to the combinatorial edge weight and $\mathbf{D}_0(F)$ does not contain non-trivial constant functions, the Dirichlet energy $\mathcal{E}_{c^*}$ defines an inner product on $\mathbf{D}_0(F)$ whose norm is complete. Thus for any $h \in \mathbf{D}^*_{0}(F)$, the Riesz representation theorem with respect to the inner product $\mathcal{E}_{c^*}$ on $\mathbf{D}_0(F)$ yields a unique function $v \in \mathbf{D}_0(F)$ such that for any $\tilde{v}\in \mathbf{D}_0(F)$,
	\begin{align*}
		\langle \! \langle h, \tilde{v} \rangle \! \rangle  = \mathcal{E}_{c^*}\left(v,\tilde{v} \right)= \sum_{\psi \sim \phi} c^*_{\phi\psi} (v_{\phi}- v_{\psi})(\tilde{v}_{\phi}- \tilde{v}_{\psi}) = \langle \! \langle \iota^{\star} \circ d\mathcal{K}^*_u(v), \tilde{v} \rangle \! \rangle. 
	\end{align*}
    Since this holds for all $\tilde{v} \in \mathbf{D}_0(F)$, we deduce that
    \[\iota^{\star} \circ d\mathcal{K}^*_u(v) = h\]    
    and we obtain the surjectivity of $\iota^{\star} \circ d\mathcal{K}^*_u$ to $\mathbf{D}^*_{0}(F)$.
    
	Since for any $u \in \mathbf{D}(F)$, the differential of $\iota^{\star} \circ \mathcal{K}^*$ at $u$ is a bounded linear map from $\mathbf{D}(F)$ to $\mathbf{D}^*_{0}(F)$ satisfying
	\[
	\mathbf{D}(F) = \mathbf{D}_0(F) \oplus \ker \left( d(\iota^{\star} \circ \mathcal{K}^*)_u \right), \quad \mathbf{D}^*_{0}(F) = d(\iota^{\star} \circ \mathcal{K}^*)_u(\mathbf{D}(F))
	\] 
    the rank theorem \cite[Theorem 2.5.15]{Abraham1988} yields that $P(\Theta,R^{\dagger}) = (\iota^{\star} \circ \mathcal{K}^*)^{-1}(0)$ is a Hilbert manifold and the tangent space \[
	T_u P(\Theta,R^{\dagger})= \ker \left( d(\iota^{\star} \circ \mathcal{K}^*)_u \right)= \mathbf{HD}_{c^*}(F)
	\]
	consists of discrete harmonic functions with respect to the geometric edge weight $c^*_{\phi\psi}$.  
\end{proof}

\subsection{A relative volume functional $\mathcal{W}^*$}\label{sec:Wstar}
We observe that the map $\mathcal{K}^*$ admits a potential functional. Indeed, Proposition \ref{prop:selfadjoint} indicates that the differential $d\mathcal{K}^*_u$ is self-adjoint for every $u \in \mathbf{D}(F)$. Since $\mathbf{D}(F)$ is a vector space and in particular simply connected, by the Poincar\'e lemma, we can integrate the 1-form defined by $\mathcal{K}^*$ to obtain a functional $\mathcal{W}^*: \mathbf{D}(F) \to \mathbb{R}$ whose gradient is $\mathcal{K}^*$. Namely,
$\mathcal{W}^*$ can be defined by integrating along the line segment from the origin:
    \begin{align*}
        \mathcal{W}^*(u) &= \int_{0}^{1} \langle \! \langle \mathcal{K}^*(tu), u \rangle \! \rangle \, dt \\
        &= \sum_{\phi \psi \in E^*} \int_{0}^{1} \mathcal{K}^{*}_{\phi \psi}(tu) (u_{\phi}-u_{\psi}) \, dt.
    \end{align*}

Below, we establish an explicit formula for $\mathcal{W}^*$, which is required in the next section to prove coercivity and apply variational methods to the deformation space $P(\Theta,R^{\dagger})$. We also summarize the primary properties of $\mathcal{W}^*$.

\begin{definition}\label{def:Wstarfunctional}
	 For every dual edge $\phi \psi \in E^*$,  we define an analytic function $\mathcal{W}^*_{\phi\psi}: \mathbf{D}(F) \to \mathbb{R}$ as
	\begin{align*}
		\mathcal{W}^*_{\phi\psi}(u) :=& \Im \Li\left(e^{u_\phi - u_\psi} \frac{R^{\dagger}_\phi}{R^{\dagger}_\psi} e^{\mathbf{i}\Theta_{\phi\psi}}\right) + \Im \Li\left(e^{u_\psi - u_\phi} \frac{R^{\dagger}_\psi}{R^{\dagger}_\phi} e^{\mathbf{i}\Theta_{\psi\phi}}\right)\\ &-\Im \Li\left(\frac{R^{\dagger}_\phi}{R^{\dagger}_\psi} e^{\mathbf{i}\Theta_{\phi\psi}}\right) - \Im \Li\left(\frac{R^{\dagger}_\psi}{R^{\dagger}_\phi} e^{\mathbf{i}\Theta_{\psi\phi}}\right) \\ & + (\alpha^{\dagger}_{\phi\psi} - \alpha^{\dagger}_{\psi\phi})(u_\phi - u_\psi)  \\
        =& \mathcal{W}^*_{\psi\phi}(u)
	\end{align*}
    where $\Im \Li(z)$ is the imaginary part of the dilogarithm function 
	\[
	\Li(z) = -\int_{0}^{z} \frac{\ln(1-w)}{w} \, dw \quad \text{for } z \in \mathbb{C} \setminus [1, +\infty).
	\]
    Then the relative volume functional $\mathcal{W}^*:\mathbf{D}(F) \to \mathbb{R}$ is defined by summing over all dual edges
    \[
    \mathcal{W}^*(u) = \sum_{\phi \psi \in E^*} \mathcal{W}^*_{\phi\psi}(u).
    \]
\end{definition}

\begin{proposition}\label{prop:potential}
   The functional $\mathcal{W}^*$ is well-defined and satisfies the following properties:
     \begin{enumerate}
        \item $\mathcal{W}^*_{\phi\psi}(0) = 0$ for every dual edge $\phi \psi \in E^*$, and hence $\mathcal{W}^*(0) = 0$.
        \item For every $u \in \mathbf{D}(F)$, the differential is $d\mathcal{W}^*(u) = \mathcal{K}^*(u)$.
        \item The origin $u=0$ is the unique critical point of $\mathcal{W}^*$ in $\mathbf{D}(F)$.
        \item We have $u \in P(\Theta,R^{\dagger})$ if and only if $d\mathcal{W}^*(u)|_{\mathbf{D}_0(F)} = 0$. 
        \item For every $u \in \mathbf{D}(F)$, the restriction of the functional $\mathcal{W}^*$ to the affine subspace $u + \mathbf{D}_0(F)$ is strictly convex.
     \end{enumerate}
\end{proposition}
\begin{proof}
    We first compute the partial derivative of $\mathcal{W}^*_{\phi\psi}$ with respect to $u_{\phi}$. By Lemma \ref{lemma:cotweight}, we have
    \begin{align*}
        &\frac{\partial \mathcal{W}^*_{\phi\psi}}{\partial u_{\phi}}\\ =& \frac{\partial }{\partial u_{\phi}} \Im \Li\left(e^{u_\phi - u_\psi} \frac{R^{\dagger}_\phi}{R^{\dagger}_\psi} e^{\mathbf{i}\Theta_{\phi\psi}}\right) + \frac{\partial }{\partial u_{\phi}} \Im \Li\left(e^{u_\psi - u_\phi} \frac{R^{\dagger}_\psi}{R^{\dagger}_\phi} e^{\mathbf{i}\Theta_{\psi\phi}}\right) + (\alpha^{\dagger}_{\phi\psi} - \alpha^{\dagger}_{\psi\phi}) \\
        =& - \Im \left( \ln\left(1- e^{u_\phi - u_\psi} \frac{R^{\dagger}_\phi}{R^{\dagger}_\psi} e^{\mathbf{i}\Theta_{\phi\psi}}\right) \right) + \Im \left( \ln\left(1- e^{u_\psi - u_\phi} \frac{R^{\dagger}_\psi}{R^{\dagger}_\phi} e^{\mathbf{i}\Theta_{\psi\phi}}\right) \right) + (\alpha^{\dagger}_{\phi\psi} - \alpha^{\dagger}_{\psi\phi})\\
        =& - \alpha_{\phi\psi} + \alpha_{\psi\phi} + (\alpha^{\dagger}_{\phi\psi} - \alpha^{\dagger}_{\psi\phi}) = \mathcal{K}^*_{\phi\psi}(u).
    \end{align*}
    Similarly, we have
    \[\frac{\partial \mathcal{W}^*_{\phi\psi}}{\partial u_{\psi}} = -\mathcal{K}^*_{\phi\psi}(u).\]
    Thus, for every $u,v \in \mathbf{D}(F)$, we have
    \begin{align*}
        \langle \! \langle d\mathcal{W}^*(u), v \rangle \! \rangle &= \sum_{\phi \psi \in E          ^*} \left( \frac{\partial \mathcal{W}^*_{\phi\psi}}{\partial u_{\phi}} v_{\phi} + \frac{\partial \mathcal{W}^*_{\phi\psi}}{\partial u_{\psi}} v_{\psi} \right) \\
        &= \sum_{\phi \psi \in E^*} \mathcal{K}^*_{\phi\psi}(u) (v_{\phi} - v_{\psi}) = \langle \! \langle \mathcal{K}^*(u), v \rangle \! \rangle.
    \end{align*}
    The remaining properties are restatements of Proposition \ref{prop:gradient} and Proposition \ref{prop:selfadjoint}.
\end{proof}

\subsection{Projection to the space of discrete harmonic Dirichlet functions $\mathbf{HD}(F)$}

For any edge weight $c^*$ equivalent to the combinatorial edge weight,  the Royden decomposition yields a decomposition of the space of Dirichlet functions $\mathbf{D}(F)$ as
\[
\mathbf{D}(F) = \mathbf{D}_0(F) \oplus \mathbf{HD}(F) = \mathbf{D}_0(F) \oplus \mathbf{HD}_{c^*}(F).
\]
where the direct sum is orthogonal with respect to the pairing $\mathcal{E}_{c^*}(u,v)$. We let $p_{c^*}:\mathbf{D}(F) \to \mathbf{HD}_{c^*}(F)$ the projection onto the harmonic part, which is a bounded linear operator. In particular, we denote by $p_F:\mathbf{D}(F) \to \mathbf{HD}(F)$ the orthogonal projection with respect to the combinatorial edge weight $c^*\equiv 1$. 

Our goal is to show that the restriction of $p_F$ to the Hilbert manifold $P(\Theta,R^\dagger)\subset \mathbf{D}(F)$ is a homeomorphism to $\mathbf{HD}(F)$. We first show that $p_F|_{P(\Theta,R^\dagger)}$ is a local homeomorphism. Subsequently, using a variational argument involving the functional $\mathcal{W}^*$, we show that $p_F|_{P(\Theta,R^\dagger)}$ is bijective.

\begin{proposition}\label{prop:localhomHDF}
The projection $p_F|_{P(\Theta,R^\dagger)}: P(\Theta,R^\dagger) \to \mathbf{HD}(F)$ is a local homeomorphism.
\end{proposition}
\begin{proof}
For $u \in P(\Theta,R^\dagger)$, we have $T_u P(\Theta,R^\dagger) = \mathbf{HD}_{c^*}(F)$ where $c^*$ is the geometric edge weight associated with the circle pattern with radii $e^u R^{\dagger}$ and is equivalent to the combinatorial edge weight by Corollary \ref{cor:equivalence}. Thus, $p_F|_{T_u P(\Theta,R^\dagger)}: T_u P(\Theta,R^\dagger) \to \mathbf{HD}(F)$ is an isomorphism whose inverse is $p_{c^*}|_{\mathbf{HD}(F)}$.
\end{proof}

To show that the projection $p_F|_{P(\Theta,R^{\dagger})}$ is a bijection onto $\mathbf{HD}(F)$, we employ a variational method involving the functional $\mathcal{W}^*: \mathbf{D}(F) \to \mathbb{R}$ introduced in the previous section. Specifically, we consider the restriction of $\mathcal{W}^*$ to the affine subspace $h + \mathbf{D}_0(F)$ for an arbitrary fixed $h \in \mathbf{HD}(F)$. To simplify the notation, we consider the following functional on the subspace $\mathbf{D}_0(F)$.

\begin{definition}
	Given a uniformized circle pattern with radii $R^{\dagger}:F \to \mathbb{R}_{>0}$ and intersection angles $\Theta$. For every fixed $h \in \mathbf{HD}(F)$, we define the functional $\mathcal{W}^*_h: \mathbf{D}_0(F) \to \mathbb{R}$ as 
	\begin{align*}
		\mathcal{W}^*_h(u) :=& \sum_{\phi \psi} \bigg[\Im \Li\left(e^{u_\phi - u_\psi} \frac{e^{h_{\phi}}R^{\dagger}_\phi}{e^{h_{\psi}}R^{\dagger}_\psi} e^{\mathbf{i}\Theta_{\phi\psi}}\right) + \Im \Li\left(e^{u_\psi - u_\phi} \frac{e^{h_{\psi}}R^{\dagger}_\psi}{e^{h_{\phi}}R^{\dagger}_\phi} e^{\mathbf{i}\Theta_{\psi\phi}}\right) \\
		& -  \Im \Li\left(\frac{e^{h_{\phi}}R^{\dagger}_\phi}{e^{h_{\psi}}R^{\dagger}_\psi} e^{\mathbf{i}\Theta_{\phi\psi}}\right) - \Im \Li\left(\frac{e^{h_{\psi}}R^{\dagger}_\psi}{e^{h_{\phi}}R^{\dagger}_\phi} e^{\mathbf{i}\Theta_{\psi\phi}}\right) + (\alpha^{\dagger}_{\phi\psi} - \alpha^{\dagger}_{\psi\phi})(u_\phi - u_\psi)]  \\
		=& \sum_{\phi\psi} (\mathcal{W}_{\phi \psi}^*(h+u) - \mathcal{W}_{\phi \psi}^*(h)) \\=& \mathcal{W}^*(h+u) - \mathcal{W}^*(h).
	\end{align*}
\end{definition}

From Proposition \ref{prop:potential}, we know that $u \in \mathbf{D}_0(F)$ is a critical point of $\mathcal{W}_h^*$ if and only if $u + h \in P(\Theta,R^\dagger)$. The strong convexity of $\mathcal{W}_h^*$ on $\mathbf{D}_0(F)$ implies that there exists at most one critical point of $\mathcal{W}_h^*$ on $\mathbf{D}_0(F)$. To show the existence of such a critical point, it suffices to show that $\mathcal{W}_h^*$ is coercive, i.e., $\mathcal{W}_h^*(u) \to +\infty$ as $\|u\| \to +\infty$. Indeed, this implies that $\mathcal{W}_h^*$ attains its minimum at some element in $\mathbf{D}_0(F)$, which is the desired critical point.

In order to show coercivity, we have to analyse carefully the asymptotic growth of $\mathcal{W}_h^*$ when $\|u\|\to +\infty$. Proposition \ref{prop:bound} indicates that the geometric edge weight $c^*_{\phi\psi}$ might go to zero as $|u_{\phi}-u_{\psi}| \to +\infty$ and hence the Hessian of $\mathcal{W}_h^*$ might degenerate. Along such directions, we derive linear lower bounds for the growth of $\mathcal{W}_h^*$. In contrast, if $\|u\|\to +\infty$ while $|u_{\phi}-u_{\psi}|$ remains bounded for every dual edge $\phi \psi \in E^*$, the geometric edge weight $c^*_{\phi\psi}$ remains uniformly bounded. In this case, the Hessian of $\mathcal{W}_h^*$ is uniformly positive definite and we derive quadratic lower bounds for the growth of $\mathcal{W}_h^*$. In the next two lemmas, we make these arguments precise. 

\begin{lemma}\label{lem:liformula}
Let $x \in \mathbb{R}$ and $\theta \in (0, \pi)$. Then
\begin{equation*}
    \Im \operatorname{Li}_2(e^{x+\mathbf{i}\theta}) + \Im \operatorname{Li}_2(e^{-x+\mathbf{i}\theta}) > (\pi - \theta)|x|.
\end{equation*}
\end{lemma}

\begin{proof}
Let $f(x) = \Im \operatorname{Li}_2(e^{x+\mathbf{i}\theta}) + \Im \operatorname{Li}_2(e^{-x+\mathbf{i}\theta})$. Since $f(x)$ is an even function of $x$, it suffices to consider $x > 0$. We utilize the inversion formula for the dilogarithm
\begin{equation}
    \operatorname{Li}_2(z) + \operatorname{Li}_2(1/z) = -\frac{\pi^2}{6} - \frac{1}{2}(\ln(-z))^2.
\end{equation}
Let $z = e^{x+\mathbf{i}\theta}$. Taking the imaginary part of the identity and observing that $\Im \operatorname{Li}_2(1/z) = \Im \operatorname{Li}_2(e^{-x-\mathbf{i}\theta}) = -\Im \operatorname{Li}_2(e^{-x+\mathbf{i}\theta})$, we obtain
\begin{equation}
    \Im \operatorname{Li}_2(e^{x+\mathbf{i}\theta}) - \Im \operatorname{Li}_2(e^{-x+\mathbf{i}\theta}) = -\Im\left(\frac{1}{2}(x + \mathbf{i}(\theta - \pi))^2\right).
\end{equation}
The right-hand side simplifies to $(\pi - \theta)x$. Substituting this back into the expression for $f(x)$ yields
\begin{equation}
    f(x) = (\pi - \theta)x + 2\Im \operatorname{Li}_2(e^{-x+\mathbf{i}\theta}).
\end{equation}
For $x > 0$, we have $|e^{-x+\mathbf{i}\theta}| < 1$. Since $\Im \operatorname{Li}_2(w) > 0$ for $w$ in the upper half-plane within the unit disk, the term $2\Im \operatorname{Li}_2(e^{-x+\mathbf{i}\theta})$ is strictly positive. Therefore, $f(x) > (\pi - \theta)x$.
\end{proof}

\begin{lemma}\label{lem:plus}
	Given a uniformized circle pattern with radii $R^{\dagger}:F \to \mathbb{R}_{>0}$, intersection angles $\Theta:\ E^* \to (\epsilon_0, \pi - \epsilon_0)$, and fixed $h \in \mathbf{HD}(F)$. Then there exist constants $\eta, M>0$ depending on $\epsilon_0$, $R^{\dagger}$ and $\mathcal{E}(h)$ such that for every $u \in \mathbf{D}_0(F)$ and every dual edge $\phi \psi \in E^*$, we have
	\begin{align*}
		\mathcal{W}^*_{\phi\psi}(u+h) - \mathcal{W}^*_{\phi\psi}(h) \geq 2\eta |u_{\phi}-u_{\psi}| - M.
	\end{align*}
	In particular, if $|u_{\phi}-u_{\psi}| \geq \frac{M}{\eta}$, then
	\begin{align*}
		\mathcal{W}^*_{\phi\psi}(u+h) - \mathcal{W}^*_{\phi\psi}(h) \geq \eta |u_{\phi}-u_{\psi}|.
	\end{align*}
\end{lemma}
\begin{proof}
	Lemma \ref{lem:liformula} implies that
	\begin{align*}
		& \mathcal{W}^*_{\phi\psi}(u+h) - \mathcal{W}^*_{\phi\psi}(h) + \Im \Li\left(e^{h_{\phi}-h_{\psi}} \frac{R^{\dagger}_\phi}{R^{\dagger}_\psi} e^{\mathbf{i}\Theta_{\phi\psi}}\right) + \Im \Li\left(e^{h_{\psi}-h_{\phi}} \frac{R^{\dagger}_\psi}{R^{\dagger}_\phi} e^{\mathbf{i}\Theta_{\psi\phi}}\right)  \\
		\geq & (\pi - \Theta_{\phi\psi}) \left|u_{\phi}-u_{\psi}+h_{\phi}-h_{\psi}+ \log \frac{R^{\dagger}_\phi}{R^{\dagger}_\psi}\right|  
		 + (\alpha^{\dagger}_{\phi\psi} - \alpha^{\dagger}_{\psi\phi})(u_{\phi}-u_{\psi}) \\
		\geq &  (\alpha^{\dagger}_{\phi\psi} + \alpha^{\dagger}_{\psi\phi}) |u_{\phi}-u_{\psi}| - (\pi - \Theta_{\phi\psi}) \left|h_{\phi}-h_{\psi}+ \log \frac{R^{\dagger}_\phi}{R^{\dagger}_\psi}\right| + (\alpha^{\dagger}_{\phi\psi} - \alpha^{\dagger}_{\psi\phi})(u_{\phi}-u_{\psi}) \\
		\geq & 2 (\min\{\alpha^{\dagger}_{\phi\psi}, \alpha^{\dagger}_{\psi\phi}\} )|u_{\phi}-u_{\psi}| - (\pi - \Theta_{\phi\psi}) \left|h_{\phi}-h_{\psi}+ \log \frac{R^{\dagger}_\phi}{R^{\dagger}_\psi}\right|.
	\end{align*}
    Let $C>0$ be the constant in the Ring lemma (Lemma \ref{lem:ring}). Proposition \ref{prop:bound} indicates that for every dual edge $\phi \psi \in E^*$,
	\begin{align*}
	\alpha^{\dagger}_{\phi\psi} > \eta, \quad \alpha^{\dagger}_{\psi\phi} > \eta
	\end{align*}
	where $\eta = \cot^{-1} \left(\frac{C+1}{\sin \epsilon_0}\right)$. Furthermore, we define
	\[
	M= \sup\left( (\pi - \theta)\left|t+\log r \right| + \left|\Im \Li(e^{t}r e^{\mathbf{i}\theta})\right| + \left|\Im \Li(\frac{1}{e^t r} e^{\mathbf{i}\theta})\right| \right)
	\]
    where the supremum is taken over all $\theta \in [\epsilon_0, \pi - \epsilon_0]$, $r \in [\frac{1}{C}, C ]$ and $t \in [-\sqrt{\mathcal{E}(h)}, \sqrt{\mathcal{E}(h)}]$. The constant $M$ is finite since the domain is compact. Thus, we have
	\begin{align*}
		\mathcal{W}^*_{\phi\psi}(u+h) - \mathcal{W}^*_{\phi\psi}(h) \geq 2 \eta |u_{\phi}-u_{\psi}| - M.
	\end{align*}

\end{proof}

\begin{lemma}\label{lem:minus}
	Let $\eta, M>0$ be the constants in the previous lemma. There exists a constant $\lambda=\lambda(\frac{M}{\eta})>0$ such that for any dual edge $\phi \psi \in E^*$ with $|u_{\phi}-u_{\psi}| \leq \frac{M}{\eta}$, we have
	\begin{align*}
		\mathcal{W}^*_{\phi\psi}(u+h) - \mathcal{W}^*_{\phi\psi}(h) &>  \mathcal{K}^*_{\phi\psi}(h) \cdot (u_{\phi} - u_{\psi}) + \frac{\lambda}{2} |u_{\phi}-u_{\psi}|^2.
	\end{align*}
\end{lemma}
\begin{proof}
	Let $u \in \mathbf{D}_0(F)$ and $\phi \psi \in E^*$. We write $x = u_{\phi} - u_{\psi}$. It suffices to consider the function
	\begin{align*}
		f(x) =& \Im \Li\left(e^{x} \frac{e^{h_{\phi}}R^{\dagger}_\phi}{e^{h_{\psi}}R^{\dagger}_\psi} e^{\mathbf{i}\Theta_{\phi\psi}}\right) + \Im \Li\left(e^{-x} \frac{e^{h_{\psi}}R^{\dagger}_\psi}{e^{h_{\phi}}R^{\dagger}_\phi} e^{\mathbf{i}\Theta_{\psi\phi}}\right) \\
		& -  \Im \Li\left(\frac{e^{h_{\phi}}R^{\dagger}_\phi}{e^{h_{\psi}}R^{\dagger}_\psi} e^{\mathbf{i}\Theta_{\phi\psi}}\right) - \Im \Li\left(\frac{e^{h_{\psi}}R^{\dagger}_\psi}{e^{h_{\phi}}R^{\dagger}_\phi} e^{\mathbf{i}\Theta_{\psi\phi}}\right) + (\alpha^{\dagger}_{\phi\psi} - \alpha^{\dagger}_{\psi\phi}) x
	\end{align*}
	Note that $f(0)=0$. The first derivative and the second derivative of $f$ satisfy
	\begin{equation*}
    \left. \frac{df}{dx} \right|_{x = u_{\phi} - u_{\psi}} = \mathcal{K}^*_{\phi\psi}(u+h)
\end{equation*}
 and 
\begin{align*}
    \left. \frac{d^2f}{dx^2} \right|_{x = u_{\phi} - u_{\psi}} &= c_{\phi\psi}^*
\end{align*}
where $c_{\phi\psi}^*$ is the geometric edge weight assigned to the dual edge $\phi\psi$ with radii $e^{u+h} R^{\dagger}$ and intersection angles $\Theta$.

By Proposition \ref{prop:bound}, there exists a constant $\lambda=\lambda(\frac{M}{\eta})>0$ such that for every dual edge $\phi \psi \in E^*$ with $|u_{\phi}-u_{\psi}| \leq \frac{M}{\eta}$, we have
\begin{align*}
	c_{\phi\psi}^* &> \lambda.
\end{align*}
Thus, by Taylor's formula, we have for some $\xi$ between $0$ and $u_{\phi} - u_{\psi}$
\begin{align*}
	\mathcal{W}^*_{\phi\psi}(u+h) - \mathcal{W}^*_{\phi\psi}(h) &= f(u_{\phi} - u_{\psi}) \\&= f(0) + \left. \frac{df}{dx} \right|_{x=0} (u_{\phi} - u_{\psi}) + \frac{1}{2} \left. \frac{d^2f}{dx^2} \right|_{x=\xi} |u_{\phi} - u_{\psi}|^2 \\
	&\geq \mathcal{K}^*_{\phi\psi}(h) (u_{\phi} - u_{\psi}) + \frac{\lambda}{2} |u_{\phi}-u_{\psi}|^2.
\end{align*}
\end{proof}

Now we combine the above two lemmas to show the coercivity of the functional $\mathcal{W}^*_h$ on $\mathbf{D}_0(F)$.

\begin{proposition}
	Let $\eta, M>0$ be the constants in the previous lemmas. For any $u \in D_0(V)$, we decompose the dual edge set $E^*$ into two subsets 
	\begin{align*}
	E^*_{+}&=\{\phi \psi \in E^* \mid |u_\phi - u_\psi| \geq \frac{M}{\eta}\}, \\
	E^*_{-}&=\{\phi \psi \in E^* \mid |u_\phi - u_\psi| < \frac{M}{\eta}\} = E^* \setminus E^*_{+}.
	\end{align*}
We decompose the Dirichlet energy $\mathcal{E}(u)= \mathcal{E}_{+}(u)+\mathcal{E}_{-}(u)$, where
	\[
	\mathcal{E}_{+}(u)=\sum_{\phi \psi \in E^*_{+}} |u_{\phi}-u_{\psi}|^2, \quad \mathcal{E}_{-}(u)=\sum_{\phi \psi \in E^*_{-}} |u_{\phi}-u_{\psi}|^2.
	\]
Then there exist constants $\lambda=\lambda(\frac{M}{\eta})>0$ and $\tilde{C}_1=\tilde{C}_1(\mathcal{E}(h),\Theta,\epsilon_0)>0$ such that
\begin{align*}
	\mathcal{W}^*_h(u) &> \eta \sqrt{\mathcal{E}_{+}(u)} - \tilde{C}_1 \sqrt{\mathcal{E}_{-}(u)} + \frac{\lambda}{2} \mathcal{E}_{-}(u).
\end{align*}
In particular, the functional $\mathcal{W}^*_h$ is coercive on $\mathbf{D}_0(F)$.
\end{proposition}
\begin{proof}
Lemma \ref{lem:plus} yields that
\begin{align*}
	\sum_{\phi \psi \in E^*_{+}} \left( \mathcal{W}^*_{\phi\psi}(u+h) - \mathcal{W}^*_{\phi\psi}(h) \right) &\geq \eta \sum_{\phi \psi \in E^*_{+}}  |u_{\phi}-u_{\psi}| \\
	&\geq \eta \left(\sum_{\phi \psi \in E^*_{+}} |u_{\phi}-u_{\psi}|^2\right)^{1/2} = \eta \sqrt{\mathcal{E}_{+}(u)}.
\end{align*}
On the other hand, Lemma \ref{lem:minus} and Lemma \ref{lem:gedge} imply that
\begin{align*}
	\sum_{\phi \psi \in E^*_{-}} \left( \mathcal{W}^*_{\phi\psi}(u+h) - \mathcal{W}^*_{\phi\psi}(h) \right) &> \sum_{\phi \psi \in E^*_{-}} \left[ \mathcal{K}^*_{\phi\psi}(h) \cdot (u_{\phi} - u_{\psi}) + \frac{\lambda}{2} |u_{\phi}-u_{\psi}|^2 \right] \\
	&\geq -\sqrt{\sum_{\phi \psi \in E^*_{-}} |\mathcal{K}^*_{\phi\psi}(h)|^2} \sqrt{\mathcal{E}_{-}(u)} + \frac{\lambda}{2} \mathcal{E}_{-}(u) \\
	&\geq -C_1 \sqrt{\mathcal{E}(h)} \sqrt{\mathcal{E}_{-}(u)} + \frac{\lambda}{2} \mathcal{E}_{-}(u) 
\end{align*}
where $C_1$ is the constant in Lemma \ref{lem:gedge} and we take $\tilde{C}_1= C_1 \sqrt{\mathcal{E}(h)}$. Combining the two estimates, we obtain the desired inequality via
\[
\mathcal{W}^*_h(u) = \sum_{\phi \psi \in E^*_{+}} \left( \mathcal{W}^*_{\phi\psi}(u+h) - \mathcal{W}^*_{\phi\psi}(h) \right) + \sum_{\phi \psi \in E^*_{-}} \left( \mathcal{W}^*_{\phi\psi}(u+h) - \mathcal{W}^*_{\phi\psi}(h) \right).
\]

We claim that the functional $\mathcal{W}^*_h$ is coercive on $\mathbf{D}_0(F)$. Indeed, let $u_n \in \mathbf{D}_0(F)$ be a sequence with $\|u_n\| \to +\infty$. If $\mathcal{E}_{+}(u_n) \to +\infty$, then the inequality implies that $\mathcal{W}^*_h(u_n) \to +\infty$. Otherwise, we can assume that $\mathcal{E}_{+}(u_n)$ is bounded. Thus, $\mathcal{E}_{-}(u_n) \to +\infty$ and the inequality implies that $\mathcal{W}^*_h(u_n) \to +\infty$ as well. This proves the claim.
\end{proof}

We are now ready to show that the projection $p_F|_{P(\Theta,R^{\dagger})}$ yields a homeomorphism between $P(\Theta,R^{\dagger})$ and $\mathbf{HD}(F)$.

\begin{proof}[Proof of Theorem \ref{thm:homHDF}]
Proposition \ref{prop:localhomHDF} indicates that the projection $p_F|_{P(\Theta,R^{\dagger})}$ is a local homeomorphism. It remains to show that $p_F|_{P(\Theta,R^{\dagger})}$ is bijective. 

We have established that for any $h \in \mathbf{HD}(F)$, the functional $\mathcal{W}^*_h: \mathbf{D}_0(F) \to \mathbb{R}$ is strictly convex and coercive. Standard variational arguments imply that $\mathcal{W}^*_h$ attains a unique global minimum at some $v \in \mathbf{D}_0(F)$ (See \cite[Corollary 3.23]{Brezis2011}).

Let $u = h + v$. By Proposition \ref{prop:potential}, $v$ being a critical point of $\mathcal{W}^*_h$ is equivalent to
\[
d\mathcal{W}^*_h(v) = 0 \iff  d\mathcal{W}^*(h+v)|_{\mathbf{D}_0(F)} = 0.
\]
This condition precisely means that $u \in P(\Theta, R^{\dagger})$ with $p_F(u) = h$. 

The uniqueness of the minimum $v$ implies the uniqueness of $u$. Suppose there are two points $u_1, u_2 \in P(\Theta, R^{\dagger})$ such that $p_F(u_1) = p_F(u_2) = h$. Then we can write $u_1 = h + v_1$ and $u_2 = h + v_2$ for some $v_1, v_2 \in \mathbf{D}_0(F)$. Both $v_1$ and $v_2$ must be critical points of the strictly convex functional $\mathcal{W}^*_h$, hence $v_1 = v_2$ and $u_1 = u_2$. Therefore, $p_F|_{P(\Theta,R^{\dagger})}$ is a bijection.

Finally, since $p|_{P(\Theta,R^{\dagger})}$ is a local homeomorphism and bijective, it is a global homeomorphism between $P(\Theta,R^{\dagger})$ and $\mathbf{HD}(F)$.
\end{proof}

\section{Conjugate angle parametrization}

Elements in $P(\Theta, R^{\dagger})\subset \mathbf{HD}(F)$ parametrize circle patterns by the change in the logarithmic radii from the uniformized circle pattern. In this section, we introduce an alternative parametrization of circle patterns through the change in the half-angles at the circumcenters.

Recall that given a function $u: F \to \mathbb{R}$ and for every dual edge $\phi \psi \in E^*$, we deform the corresponding kite in the uniformized circle pattern to obtain a new kite with side lengths $e^{u_{\phi}} R^{\dagger}_{\phi}$ and $e^{u_{\psi}} R^{\dagger}_{\psi}$ meeting at an angle $\Theta_{\phi \psi}$ (See Figure \ref{fig:angle_at_center}). The deformed kite is uniquely determined in the plane up to translation and rotation. A flat Euclidean metric is obtained by gluing these deformed kites along their common edges if and only if the angle sum around every circumcenter is $2 \pi$. Particularly in such a case, the rotation holonomy around every circumcenter is trivial.

Analogously, given a function $v: V \to \mathbb{R}$ on the vertices with $|v_j-v_i|$ sufficiently small for $ij \in E$, we deform the kites as follows. For every oriented edge $ij \in E$ with left face $\phi$ and right face $\psi$, we modify the kite in the uniformized circle pattern by specifying the half-angles at the circumcenters
\[
\alpha^{\dagger}_{\phi \psi} \mapsto \alpha_{\phi \psi} = \alpha^\dagger_{\phi \psi} + \frac{v_j - v_i}{2}, \quad \alpha^{\dagger}_{\psi \phi} \mapsto \alpha_{\psi \phi} = \alpha^\dagger_{\psi \phi} - \frac{v_j - v_i}{2}
\]
and further demand that the diagonal $ij$ in the deformed kite gets rotated counterclockwise by an angle of $\frac{v_i+v_j}{2}$ compared to the diagonal in the uniformized kite. In this way, the deformed kite is uniquely determined in the plane up to translation and scaling. With some trigonometry, one can verify that the sides $\phi i$ and $\psi i$ in the deformed kite are rotated by an angle of $v_i$ compared to the corresponding sides in the uniformized circle pattern. Similarly, the sides $\phi j$ and $\psi j$ are rotated by an angle of $v_j$. Hence, with suitable translations and scalings, neighboring deformed kites can be glued together along their common edges. Under this deformation, the intersection angle of the corresponding circles remains unchanged since
\[\alpha_{\phi \psi} + \alpha_{\psi \phi} = \alpha^{\dagger}_{\phi \psi} + \alpha^{\dagger}_{\psi \phi} = \pi-\Theta_{\phi \psi}.\]
When gluing the neighbouring kites, we observe that the angle sum around every circumcenter $\phi$ and intersection point $i$ remains $2\pi$. However, the scaling holonomy around every vertex $i$ might be non-trivial. In order to obtain a flat Euclidean metric by gluing the deformed kites, we must require the scaling holonomy around every vertex to be trivial. Using the sine law on the kites, the ratio of the side lengths of the deformed kite is 
\[
\frac{R_{\psi}}{R_{\phi}} = \frac{\sin \alpha_{\phi \psi}}{\sin \alpha_{\psi \phi}}.
\]
The scaling holonomy around a vertex $i$ is trivial if and only if the product of these ratios of adjacent kites around $i$ equals one. This leads to a system of nonlinear equations for the function $v$ on vertices.

We consider the space of Dirichlet-finite functions $\mathbf{D}(V)$ consisting of functions $v:V \to \mathbb{R}$ with finite combinatorial Dirichlet energy
\[
\sum_{ij \in E} |v_i - v_j|^2 < \infty.
\]
We define the space $P(\Theta, \alpha^{\dagger})$ as the set of functions $v \in \mathbf{D}(V)$ that describe circle patterns with the same intersection angles $\Theta$ but realized through the changes in the half-angles at the circumcenters.

\begin{definition}
   Let $\mathcal{A}_{\alpha^{\dagger}} \subset \mathbf{D}(V)$ be the set of admissible functions 
    \[
   \mathcal{A}_{\alpha^{\dagger}} = \{ v \in \mathbf{D}(V) \mid \alpha^\dagger_{\phi \psi} + \frac{v_j - v_i}{2} > 0, \; \alpha^\dagger_{\psi \phi} - \frac{v_j - v_i}{2} > 0 \; \forall ij \in \vec{E} \}
   \]
    where for each oriented edge $ij \in \vec{E}$, $\phi$ and $\psi$ denote the adjacent left and right faces respectively. We let $\overline{\mathcal{A}}_{\alpha^{\dagger}}$ denote the closure of $\mathcal{A}_{\alpha^{\dagger}}$ in $\mathbf{D}(V)$, wherein at least one of the inequalities become non-strict.
	
	The space $P(\Theta, \alpha^{\dagger}) \subset \mathcal{A}_{\alpha^{\dagger}} $ is defined as the set of functions $v \in \mathbf{D}(V)$ satisfying for every vertex $i \in V$
\begin{align}\label{eq:nonv_vertex}
	\sum_{j} \log \frac{\sin (\alpha^\dagger_{\phi \psi} + \frac{v_j - v_i}{2})}{\sin (\alpha^\dagger_{\psi \phi} - \frac{v_j - v_i}{2}) } = 0,
\end{align}
where the sum is taken over all vertices $j$ adjacent to $i$ and for each edge $ij$, $\phi$ and $\psi$ denote the adjacent left and right faces respectively.
\end{definition}

In the following subsections, we establish results for the space $P(\Theta, \alpha^{\dagger})$ that parallel our earlier findings for $P(\Theta, R^{\dagger})$. Specifically, we demonstrate that $P(\Theta, \alpha^{\dagger})$ is a Hilbert submanifold of $\mathbf{D}(V)$. Furthermore, we show that it is homeomorphic to the space of discrete harmonic functions on vertices, $\mathbf{HD}(V)$, via the projection $p_V: \mathbf{D}(V) \to \mathbf{HD}(V)$ associated with the orthogonal decomposition
\[\mathbf{D}(V) = \mathbf{D}_0(V) \oplus \mathbf{HD}(V).\]

\subsection{Hilbert submanifold in $\mathbf{D}(V)$}

We first show that $P(\Theta, \alpha^{\dagger})$ is a Hilbert submanifold of $\mathbf{D}(V)$. 

\begin{definition}
	For every oriented edge $ij \in \vec{E}$ with left face $\phi$ and right face $\psi$, we define the function $\mathcal{K}_{ij}: \mathcal{A}_{\alpha^{\dagger}} \to \mathbb{R}$ as
	\begin{align*}
		\mathcal{K}_{ij}(v) = \log \frac{\sin (\alpha^\dagger_{\phi \psi} + \frac{v_j - v_i}{2})}{\sin (\alpha^\dagger_{\psi \phi} - \frac{v_j - v_i}{2}) } - \log \frac{\sin \alpha^\dagger_{\phi \psi} }{\sin \alpha^\dagger_{\psi \phi}  }.
	\end{align*}
\end{definition}

\begin{lemma}\label{lem:Kbounded}
Given $v \in \mathcal{A}_{\alpha^{\dagger}}$, there is a constant $C_2=C_2(v,\Theta,\epsilon_0)> 0$ such that for every edge $ij \in E$
\[
|\mathcal{K}_{ij}(v)| \leq C_2 |v_i - v_j|.
\]
\end{lemma}
\begin{proof}
 The proof is similar to that of Lemma \ref{lem:gedge}. We consider the function
 \[p(x) = \log \frac{\sin (\alpha^\dagger_{\phi \psi} + \frac{x}{2})}{\sin (\alpha^\dagger_{\psi \phi} - \frac{x}{2}) } - \log \frac{\sin \alpha^\dagger_{\phi \psi} }{\sin \alpha^\dagger_{\psi \phi}  }.\]
 We observe that $p(0) = 0$ and $\mathcal{K}_{ij}(v) = p(v_j - v_i)$.

 By the mean value theorem, there exists some $\xi$ between $0$ and $v_j-v_i$ such that
 \[|p(x)| = |p(x) - p(0)| = |p'(\xi)| \cdot |x|.\]
 A direct computation shows that
 \[p'(x) = \frac{1}{2} \left( \cot(\alpha^\dagger_{\phi \psi} + \frac{x}{2}) + \cot(\alpha^\dagger_{\psi \phi} - \frac{x}{2}) \right).\]

We claim that there exists $\tilde{\eta}>0$ such that for every edge $ij \in E$, we have 
\[
\alpha^\dagger_{\phi \psi} + \frac{v_j - v_i}{2} \geq \tilde{\eta}, \quad \alpha^\dagger_{\psi \phi} - \frac{v_j - v_i}{2} \geq \tilde{\eta}.
\]
Indeed, Proposition \ref{prop:bound} implies that there exists a constant $\eta$ such that for the angles in the uniformized circle pattern
\[
\alpha^\dagger_{\phi \psi} > \eta, \quad \alpha^\dagger_{\psi \phi} > \eta.
\]
Because $v$ has finite Dirichlet energy, there are only finitely many edges $ij \in E$ with $\frac{|v_j - v_i|}{2} > \frac{\eta}{2}$. Thus, we can take
\[
\eta_1= \inf_{ij \in E} \{\alpha^\dagger_{\phi \psi} + \frac{v_j - v_i}{2}, \quad \alpha^\dagger_{\psi \phi} - \frac{v_j - v_i}{2} \} > 0.
\]
We take $\tilde{\eta} = \min\{\eta, \eta_1\}$.

Then for every $x$ between $0$ and $v_j - v_i$, we have
 \[
 |p'(x)| = |\frac{1}{2} \frac{\sin \Theta_{\phi \psi}}{\sin(\alpha^\dagger_{\phi \psi} + \frac{x}{2}) \sin(\alpha^\dagger_{\psi \phi} - \frac{x}{2})}| \leq \frac{1}{2} \frac{1}{\sin^2 \tilde{\eta}}
 \]
 where the upper bound is independent of $ij \in E$. This completes the proof.	
\end{proof}

We denote the inclusion map
\[
\iota: \mathbf{D}_{0}(V) \to \mathbf{D}(V)
\]
and the induced mapping on the dual spaces
\[
\iota^{\star}: \mathbf{D}^{\star}(V) \to \mathbf{D}^{\star}_{0}(V),
\]
where $\mathbf{D}^{\star}(V)$ and $\mathbf{D}^{\star}_{0}(V)$ are respectively the spaces of bounded linear functionals on $\mathbf{D}(V)$ and $\mathbf{D}_{0}(V)$.

\begin{proposition}\label{prop:Kedge}
	For every $v\in \mathcal{A}_{\alpha^{\dagger}}$ and oriented edge $ij \in \vec{E}$, we have
	\begin{align*}
		\mathcal{K}_{ij}(v) = -\mathcal{K}_{ji}(v).
	\end{align*}
	There is a well-defined map $\mathcal{K}: \mathcal{A}_{\alpha^{\dagger}} \to \mathbf{D}(V)^*$ such that for every $v \in \mathcal{A}_{\alpha^{\dagger}}$ and $ u \in \mathbf{D}(V)$,
	\begin{align*}
		\langle \! \langle \mathcal{K}(v), u \rangle \! \rangle := \frac{1}{2}\sum_{ij \in \vec{E}} \mathcal{K}_{ij}(v) \cdot (u_i - u_j) =  \sum_{ij \in E} \mathcal{K}_{ij}(v) \cdot (u_i - u_j)
	\end{align*}
	where $\langle \! \langle \cdot, \cdot \rangle \! \rangle$ denotes the dual pairing between $\mathbf{D}(V)^*$ and $\mathbf{D}(V)$.

	Furthermore, we have $v \in P(\Theta, \alpha^{\dagger})$ if and only if $\iota^{\star}\circ \mathcal{K}(v) = 0$.
\end{proposition}
\begin{proof}
	Suppose $ij \in \vec{E}$ with left face $\phi$ and right face $\psi$. By definition, we have
	\begin{align*}
		\mathcal{K}_{ji}(v) &= \log \frac{\sin (\alpha^\dagger_{\psi \phi} + \frac{v_i - v_j}{2})}{\sin (\alpha^\dagger_{\phi \psi} - \frac{v_i - v_j}{2}) } - \log \frac{\sin (\alpha^\dagger_{\psi \phi} )}{\sin (\alpha^\dagger_{\phi \psi} ) } \\
		&= - \log \frac{\sin (\alpha^\dagger_{\phi \psi} + \frac{v_j - v_i}{2})}{\sin (\alpha^\dagger_{\psi \phi} - \frac{v_j - v_i}{2}) } + \log \frac{\sin (\alpha^\dagger_{\phi \psi} )}{\sin (\alpha^\dagger_{\psi \phi} ) } \\
		&= -\mathcal{K}_{ij}(v).
	\end{align*}
Similar to the proof of Proposition \ref{prop:gradient}, the map $\mathcal{K}: \mathcal{A}_{\alpha^{\dagger}} \to \mathbf{D}(V)^*$ is well-defined and bounded by Lemma \ref{lem:Kbounded}.

Finally, fix a vertex $i \in V$ and let $\chi_i \in \mathbf{D}_0(V)$ be its indicator function, i.e. $\chi_i$ equals to 1 at vertex $i$ and 0 elsewhere. We have
\begin{align*}
	\langle \! \langle \iota^{\star}\circ \mathcal{K}(v), \chi_i \rangle \! \rangle &= \langle \! \langle \mathcal{K}(v), \chi_i \rangle \! \rangle  = \ \sum_{j} \mathcal{K}_{ij}(v)	
\end{align*}
where the sum is over all vertices $j$ adjacent to $i$. Since the span of the indicator functions is dense in $\mathbf{D}_0(V)^*$, we have $\iota^{\star}\circ \mathcal{K}(v) = 0$ if and only if $v \in P(\Theta, \alpha^{\dagger})$.
\end{proof}

\begin{remark}
    The constant term $-\log \frac{\sin \alpha^\dagger_{\phi \psi}}{\sin \alpha^\dagger_{\psi \phi}}$ in the definition of $\mathcal{K}_{ij}(v)$ serves as a crucial normalization. Without it, the map $\mathcal{K}(v)$ would fail to define a bounded linear functional on the Hilbert space $\mathbf{D}(V)$. 
\end{remark}

We further investigate the differential of the map $\mathcal{K}$.

\begin{proposition}\label{prop:selfadjoint_vertex}
	For any $v \in \mathcal{A}_{\alpha^{\dagger}}$, the differential of the map $\mathcal{K}$ at $v$, denoted by $d\mathcal{K}_v: \mathbf{D}(V) \to \mathbf{D}(V)^*$ satisfies for every $u, \tilde{u} \in \mathbf{D}(V)$,
	\[
	\langle \! \langle d\mathcal{K}_v(u), \tilde{u} \rangle \! \rangle = -\sum_{ij \in E} c_{ij} (u_i - u_j)(\tilde{u}_i - \tilde{u}_j) = \langle \! \langle u, d\mathcal{K}_v(\tilde{u}) \rangle \! \rangle.
	\]
	Here, $c_{ij}$ are the geometric edge weights on the primal graph associated with $v$
	\[
	c_{ij} = \frac{1}{2}\left( \cot\left(\alpha^\dagger_{\phi \psi} + \frac{v_j - v_i}{2}\right) + \cot\left(\alpha^\dagger_{\psi \phi} - \frac{v_j - v_i}{2}\right) \right) > 0.
	\]
\end{proposition}
\begin{proof}
	We compute the partial derivative of $\mathcal{K}_{ij}(v)$ with respect to $v_j$. Let $\alpha_{ij} = \alpha^\dagger_{\phi \psi} + \frac{v_j - v_i}{2}$ and $\alpha_{ji} = \alpha^\dagger_{\psi \phi} - \frac{v_j - v_i}{2}$.
	\begin{align*}
		\frac{\partial \mathcal{K}_{ij}}{\partial v_j} &= \frac{\partial}{\partial v_j} \left( \log \sin \left(\alpha^\dagger_{\phi \psi} + \frac{v_j - v_i}{2}\right) - \log \sin \left(\alpha^\dagger_{\psi \phi} - \frac{v_j - v_i}{2}\right) \right) = c_{ij}. 
	\end{align*}
	Similarly, $\frac{\partial \mathcal{K}_{ij}}{\partial v_i} = -c_{ij}$.
	Thus, for any $\dot{v}, \tilde{v} \in \mathbf{D}(V)$,
	\begin{align*}
		\langle \! \langle d\mathcal{K}_v(u), \tilde{u} \rangle \! \rangle &= \sum_{ij \in E} \left( \frac{\partial \mathcal{K}_{ij}}{\partial v_j} u_j + \frac{\partial \mathcal{K}_{ij}}{\partial v_i} u_i \right) (\tilde{u}_i - \tilde{u}_j) \\
		&= \sum_{ij \in E} c_{ij} (u_j - u_i) (\tilde{u}_i - \tilde{u}_j).	\end{align*}
\end{proof}

\begin{proposition}\label{prop:hilbertV}
	The space $P(\Theta, \alpha^{\dagger})$ is a Hilbert submanifold of $\mathbf{D}(V)$.
\end{proposition}
\begin{proof}
	Proposition \ref{prop:Kedge} establishes that $P(\Theta,\alpha^{\dagger})= \left( \iota^{\star}\circ \mathcal{K} \right)^{-1}(0)$. The rest of the proof is analogous to that of Proposition \ref{prop:hilbert} and so we omit the details.
\end{proof}

\subsection{A relative volume functional $\mathcal{W}$}

Proposition \ref{prop:selfadjoint_vertex} demonstrates that the differential $d\mathcal{K}_v$ is self-adjoint for every admissible function $v \in \mathcal{A}_{\alpha^{\dagger}}$. Consequently, we can integrate $\mathcal{K}$ as a 1-form to obtain a potential functional defined on the closure $\overline{\mathcal{A}}_{\alpha^{\dagger}}$.

\begin{definition}\label{def:Wfunctional}
	For every oriented primal edge $ij \in \vec{E}$ with circumcenters $\phi$ and $\psi$ on the left and right respectively, we define a functional $\mathcal{W}_{ij}: \overline{\mathcal{A}}_{\alpha^{\dagger}} \to \mathbb{R}$ by
	\begin{align*}
		\mathcal{W}_{ij}(v) =& 2 \left( \Lambda\left(\alpha^\dagger_{\phi \psi} + \frac{v_j - v_i}{2}\right) + \Lambda\left(\alpha^\dagger_{\psi \phi} - \frac{v_j - v_i}{2}\right) 
		 - \Lambda\left(\alpha^\dagger_{\phi \psi}\right) - \Lambda\left(\alpha^\dagger_{\psi \phi}\right) \right)\\+& (v_i - v_j)  \log \frac{\sin \alpha^\dagger_{\psi \phi}}{\sin \alpha^\dagger_{\phi \psi}} \\=& \mathcal{W}_{ji}(v)
	\end{align*}
	where $\Lambda(x) = -\int_0^x \log |2\sin t| dt$ is Milnor's Lobachevsky function.
	We define the functional $\mathcal{W}: \overline{\mathcal{A}}_{\alpha^{\dagger}} \to \mathbb{R}$ as
	\[
	\mathcal{W}(v) = \sum_{ij \in E} \mathcal{W}_{ij}(v).
	\]
\end{definition}

The Lobachevsky function $\Lambda:\mathbb{R} \to \mathbb{R}$ is continuous, odd and periodic with period $\pi$. It is smooth and analytic except at integer multiples of $\pi$ (See \cite{Milnor1994}).

\begin{proposition}\label{prop:potential_vertex}
    The functional $\mathcal{W}$ is well-defined and smooth on the admissible set $\mathcal{A}_{\alpha^{\dagger}} \subset \mathbf{D}(V)$. Furthermore, it satisfies the following properties:
     \begin{enumerate}
        \item $\mathcal{W}_{ij}(0) = 0$ for every edge $ij \in E$, and hence $\mathcal{W}(0)=0$.
        \item For every $v \in \mathcal{A}_{\alpha^{\dagger}}$, the differential is $d\mathcal{W}(v) = \mathcal{K}(v)$.
        \item We have $v \in P(\Theta,\alpha^{\dagger})$ if and only if $d\mathcal{W}(v)|_{\mathbf{D}_0(V)} = 0$.
        \item The origin $v=0$ is the unique critical point of $\mathcal{W}$ in $\mathcal{A}_{\alpha^{\dagger}}$.
        \item For every $v \in \mathbf{D}(V)$, the restriction of the functional $\mathcal{W}$ to the convex set 
        \[ \mathcal{A}_{\alpha^{\dagger}} \cap \left( v + \mathbf{D}_0(V) \right) \] 
        is strictly concave if the intersection is non-empty.
     \end{enumerate}
\end{proposition}
\begin{proof}
We compute the partial derivatives of $\mathcal{W}_{ij}$ with respect to $v_i$ and $v_j$. Let $\alpha_{ij} = \alpha^\dagger_{\phi \psi} + \frac{v_j - v_i}{2}$ and $\alpha_{ji} = \alpha^\dagger_{\psi \phi} - \frac{v_j - v_i}{2}$. Using the property $\Lambda'(x) = -\log |2\sin x|$, we have:
\begin{align*}
	\frac{\partial \mathcal{W}_{ij}}{\partial v_j} &= 2 \left( -\log(2\sin \alpha_{ij}) \cdot \frac{1}{2} - \log(2\sin \alpha_{ji}) \cdot \left(-\frac{1}{2}\right) \right) - \log \frac{\sin \alpha^\dagger_{\psi \phi}}{\sin \alpha^\dagger_{\phi \psi}} \\
	&= \log \frac{\sin \alpha_{ji}}{\sin \alpha_{ij}} - \log \frac{\sin \alpha^\dagger_{\psi \phi}}{\sin \alpha^\dagger_{\phi \psi}} = -\mathcal{K}_{ij}(v).
\end{align*}
Similarly, $\frac{\partial \mathcal{W}_{ij}}{\partial v_i} = \mathcal{K}_{ij}(v)$.
Thus, for any $\tilde{v} \in \mathbf{D}(V)$,
\begin{align*}
	d\mathcal{W}(v)(\tilde{v}) &= \sum_{ij \in E} \left( \frac{\partial \mathcal{W}_{ij}}{\partial v_i} \tilde{v}_i + \frac{\partial \mathcal{W}_{ij}}{\partial v_j} \tilde{v}_j \right) \\
	&= \sum_{ij \in E} \mathcal{K}_{ij}(v) (\tilde{v}_i - \tilde{v}_j) = \langle \! \langle \mathcal{K}(v), \tilde{v} \rangle \! \rangle.
\end{align*}
This proves (1). Property (2) follows immediately from the definition of $P(\Theta, \alpha^{\dagger})$. For (3), observe that the Hessian of $\mathcal{W}$ restricted to $v + \mathbf{D}_0(V)$ is given by the operator $d\mathcal{K}_v$, which by Proposition \ref{prop:selfadjoint_vertex} is negative definite due to the positive edge weights $c_{ij}$. This ensures strict concavity.
\end{proof}

\subsection{Homeomorphism to $\mathbf{HD}(V)$}

The Royden decomposition yields an orthogonal decomposition of the space of Dirichlet-finite functions:
\[
\mathbf{D}(V) = \mathbf{D}_0(V) \oplus \mathbf{HD}(V).
\]
We denote by $p_V: \mathbf{D}(V) \to \mathbf{HD}(V)$ the projection onto the harmonic component, which is a bounded linear operator. Similar to the previous section, our goal is to show that the restriction of $p_V$ to the Hilbert manifold $P(\Theta,\alpha^\dagger)\subset \mathbf{D}(V)$ is a homeomorphism onto $\mathbf{HD}(V)$. The proof that $p_V|_{P(\Theta,\alpha^\dagger)}$ is a local homeomorphism proceeds exactly as in the previous section. To establish global bijectivity, we employ a variational method based on the functional $\mathcal{W}$ defined above, which requires different techniques compared to the analysis of the functional $\mathcal{W}^*$ used previously.

\begin{proposition}
The projection $p_V|_{P(\Theta,\alpha^\dagger)}: P(\Theta,\alpha^\dagger) \to \mathbf{HD}(V)$ is a local homeomorphism.
\end{proposition}
\begin{proof}
	The proof is analogous to that of Proposition \ref{prop:localhomHDF}. 
\end{proof}

We observe that for $h \in \mathbf{HD}(V)$ and $v \in \mathbf{D}_0(V)$, we have $h+v \in \mathcal{A}_{\alpha^{\dagger}}$ if and only if $v \in (-h + \mathcal{A}_{\alpha^{\dagger}})$. We thus consider the functional $\mathcal{W}$ restricted to the affine set $\mathbf{D}_0(V) \cap \left( -h + \mathcal{A}_{\alpha^{\dagger}}\right)$.

\begin{definition}
	Given a uniformized circle pattern with half-angles $\alpha^{\dagger}$ and intersection angles $\Theta$. For every fixed $h \in \mathbf{HD}(V)$, we define the functional $\mathcal{W}_h: \mathbf{D}_0(V) \cap \left( -h + \mathcal{A}_{\alpha^{\dagger}}\right) \to \mathbb{R}$ as 
	\begin{align*}
		\mathcal{W}_h(v):=& \mathcal{W}(h+v) \\
		=& \sum_{ij \in E} 2 \left( \Lambda\left(\alpha^\dagger_{\phi \psi} + \frac{h_j + v_j - h_i - v_i}{2}\right) + \Lambda\left(\alpha^\dagger_{\psi \phi} - \frac{h_j + v_j - h_i - v_i}{2}\right) \right. \\
		&\left. - \Lambda\left(\alpha^\dagger_{\phi \psi}\right) - \Lambda\left(\alpha^\dagger_{\psi \phi}\right)  
		 + \log \frac{\sin \alpha^\dagger_{\psi \phi}}{\sin \alpha^\dagger_{\phi \psi}} (h_i+v_i - v_j-h_j)\right).
	\end{align*}
\end{definition}

We first show that for every $h \in \mathbf{HD}(V)$, the domain of the functional $\mathcal{W}_h$ is non-empty.

\begin{lemma}\label{lem:lipschitz}
For any $\epsilon >0$ and $ h \in \mathbf{D}(V)$, there exists a function $v:V \to \mathbb{R}$ with finite support such that the function $\tilde{h}:= h +v$ satisfies for every $x,y \in V$
\[
|\tilde{h}_x - \tilde{h}_y| < \epsilon \cdot d_G(x,y)
\]  
where the graph metric $d_G$ is defined as the minimal number of edges connecting $x$ and $y$.
\end{lemma}
\begin{proof}
Fix $\epsilon >0$, $h \in \mathbf{D}(V)$ and a root vertex $o \in V$. For $R>0$, we denote by $B_R$ the ball of radius $R$ centered at $o$ in the graph metric
\[
B_R := \{ x \in V \mid d_G(o,x) \leq R\}.
\]

Since $h$ has finite Dirichlet energy, there are finitely many edges $ij$ such that $|h_i-h_j| \geq \epsilon$. Therefore, there exists $R_0>0$ such that for every $ij \in E$ with $i,j \notin B_{R_0}$,
\[|h_i - h_j| < \epsilon.
\]

On the other hand, for any $x,y \in V$ and any geodesic $\gamma$ connecting them. By the Cauchy-Schwarz inequality, we have
\[
|h_x - h_y| \leq \sum_{ij \in \gamma} |h_i - h_j| \leq \sqrt{d_G(x,y)} \left( \sum_{ij \in \gamma} |h_i - h_j|^2 \right)^{\frac{1}{2}} \leq \sqrt{d_G(x,y)} \sqrt{\mathcal{E}(h)}.
\] 
Let $R_1 = \frac{\mathcal{E}(h)}{\epsilon^2}$. Then for any $x,y \in V$ with $d_G(x,y) \geq R_1$, we have
\[
\frac{|h_x - h_y|}{d_G(x,y)} \leq \frac{\sqrt{\mathcal{E}(h)}}{R_1^{1/2}} = \epsilon.
\]
Take $R_2 = \max\{R_0, R_1\}$. Let $x,y \in V \setminus B_{2R_2}$ and let $\gamma$ be a geodesic connecting them. There are two cases:
\begin{enumerate}
	\item If $\gamma \cap B_{R_2} = \emptyset$, then all edges in $\gamma$ are outside $B_{R_0}$ and hence \[ |h_x- h_y| \leq \sum_{ij \in \gamma} |h_i - h_j| < \epsilon d_G(x,y).\] 
	\item If $\gamma \cap B_{R_2} \neq \emptyset$, then $d_G(x,y) \geq R_2 \geq R_1$ and we thus have
	\[|h_x - h_y| \leq \epsilon d_G(x,y).\]
\end{enumerate}
Combining the two cases, we conclude that for any $x,y \in V \setminus B_{2R_2}$,
\[|h_x - h_y| \leq \epsilon d_G(x,y).
\]

With the previous estimates, the construction of the function $\tilde{h}$ relies on the McShane extension theorem \cite{McShane1934}. We define the function $\tilde{h}$ as:
\[
\tilde{h}_x := \begin{cases} 
\inf_{y \in V \setminus B_{2R_2}} \big( h_y + \epsilon \cdot d_G(x, y) \big) & \text{if } x \in B_{2R_2} \\
h_x & \text{if } x \in V \setminus B_{2R_2}
\end{cases}
\]
 which coincides with $h$ outside $B_{2R_2}$. 
 
 Let $x \in B_{2R_2}$, we claim that $\tilde{h}_x$ is finite. Fix $y, y_0 \in V \setminus B_{2R_2}$, the previous estimates yield that
 \[
 h_y- h_{y_0} \geq -\epsilon d_G(y,y_0)
 \]
 and hence
\begin{align*}
h_y + \epsilon \cdot d_G(x, y) & \geq h_{y_0} - \epsilon \cdot d_G(y, y_0) + \epsilon \cdot d_G(x, y) \\
& \geq h_{y_0} - \epsilon \cdot d_G(x, y_0)
\end{align*}
Since $y$ is arbitrary, we conclude that $\tilde{h}_x$ is finite.

Next, we verify that $\tilde{h}$ is $\epsilon$-Lipschitz inside $B_{2R_2}$. Let $x_1, x_2 \in B_{2R_2}$ and fix an arbitrary $\delta > 0$. By the definition of the infimum for $\tilde{h}_{x_2}$, there exists $y \in V \setminus B_{2R_2}$ such that
\[
h_y + \epsilon d_G(x_2, y) < \tilde{h}_{x_2} + \delta.
\]
With this same $y$ and the triangle inequality, the definition of $\tilde{h}_{x_1}$ yields
\begin{align*}
\tilde{h}_{x_1} &\leq h_y + \epsilon d_G(x_1, y) \\
&\leq \big( h_y + \epsilon d_G(x_2, y) \big) + \epsilon d_G(x_1, x_2)\\
& < \tilde{h}_{x_2} + \delta + \epsilon d_G(x_1, x_2).
\end{align*}

Taking the limit as $\delta \to 0$ results in $\tilde{h}_{x_1} - \tilde{h}_{x_2} \leq \epsilon d_G(x_1, x_2)$. By symmetry, exchanging $x_1$ and $x_2$ provides $\tilde{h}_{x_2} - \tilde{h}_{x_1} \leq \epsilon d_G(x_1, x_2)$, which proves that $|\tilde{h}_{x_1} - \tilde{h}_{x_2}| \leq \epsilon d_G(x_1, x_2)$.
 
 Finally, we verify the Lipschitz condition across the boundary. Let $x_1 \in B_{2R_2}$ and $x_2 \in V \setminus B_{2R_2}$. By definition, $\tilde{h}_{x_2} = h_{x_2}$.

To establish the upper bound, we use the definition of $\tilde{h}_{x_1}$ as an infimum and evaluate it at the specific point $y = x_2$:
\[
\tilde{h}_{x_1} \leq h_{x_2} + \epsilon d_G(x_1, x_2) = \tilde{h}_{x_2} + \epsilon d_G(x_1, x_2).
\]
This directly yields $\tilde{h}_{x_1} - \tilde{h}_{x_2} \leq \epsilon d_G(x_1, x_2)$.

For the lower bound, let $y \in V \setminus B_{2R_2}$ be arbitrary. Because $h$ is $\epsilon$-Lipschitz on $V \setminus B_{2R_2}$ and by the triangle inequality, we have
\begin{align*}
h_{x_2} &\leq h_y + \epsilon d_G(x_2, y) \\
&\leq h_y + \epsilon \big( d_G(x_2, x_1) + d_G(x_1, y) \big) \\
&= \big( h_y + \epsilon d_G(x_1, y) \big) + \epsilon d_G(x_1, x_2).
\end{align*}
Rearranging gives $h_{x_2} - \epsilon d_G(x_1, x_2) \leq h_y + \epsilon d_G(x_1, y)$. Since this inequality holds for all $y \in V \setminus B_{2R_2}$, it must also hold for the infimum over $y$ on the right side, giving
\[
h_{x_2} - \epsilon d_G(x_1, x_2) \leq \tilde{h}_{x_1}.
\]
It yields $\tilde{h}_{x_2} - \tilde{h}_{x_1} \leq \epsilon d_G(x_1, x_2)$. Together with the upper bound, we conclude that $|\tilde{h}_{x_1} - \tilde{h}_{x_2}| \leq \epsilon d_G(x_1, x_2)$.
 
Finally, set $v := \tilde{h} - h$. Since $\tilde{h}$ and $h$ are identical outside $B_{2R_2}$, $v$ has finite support.
\end{proof}

\begin{proposition}
	For every $h \in \mathbf{HD}(V)$, the domain of the functional $\mathcal{W}_h$ has non-empty interior, i.e.
	\[
	\mathbf{D}_0(V) \cap \left( -h + \mathcal{A}_{\alpha^{\dagger}}\right) \neq \emptyset.
	\]
\end{proposition}
\begin{proof}
	As a consequence of the Ring lemma, proposition \ref{prop:bound} yields that there exists some $\eta >0$ such that for every edge $ij \in E^*$ with left face $\phi$ and right face $\psi$, the half-angles are uniformly bounded away from $0$ and $\pi$:
	\[\eta < \alpha^\dagger_{\phi \psi} < \pi - \eta, \quad \text{and} \quad \eta < \alpha^\dagger_{\psi \phi} < \pi - \eta.\]
	By Lemma \ref{lem:lipschitz}, there exists some $v \in \mathbf{D}_0(V)$ such that the function $\tilde{h} := h + v$ satisfies for every edge $ij \in E$
	\[|\tilde{h}_j - \tilde{h}_i| < \frac{\eta}{2}.\]
	Hence, we have $h+v \in \mathcal{A}_{\alpha^{\dagger}}$ and thus $v \in \mathbf{D}_0(V) \cap \left( -h + \mathcal{A}_{\alpha^{\dagger}}\right) \neq \emptyset$.
\end{proof}

\begin{proposition}
	The functional $\mathcal{W}_h$ is coercive on $\mathbf{D}_0(V) \cap \left( -h + \mathcal{A}_{\alpha^{\dagger}}\right)$.
\end{proposition}
\begin{proof}
  Proposition \ref{prop:bound} ensures that for circle patterns with intersection angles in $(\epsilon_0, \pi - \epsilon_0)$, the geometric edge weights $c_{ij}$ are uniformly bounded below by a constant $\tan \frac{\epsilon_0}{2} >0$. Consequently, the Hessian of $\mathcal{W}_h$ satisfies
  \[
  \langle \! \langle d^2\mathcal{W}_h(v)(u), u \rangle \! \rangle = - \sum_{ij \in E} c_{ij} (u_i - u_j)^2 \leq -\tan \frac{\epsilon_0}{2} \sum_{ij \in E} (u_i - u_j)^2.
  \]
  This inequality implies that $\mathcal{W}_h$ decays at least quadratically with the Dirichlet norm of $u$. Therefore, as $\|u\|_{\mathbf{D}(V)} \to \infty$, we have $\mathcal{W}_h(u) \to -\infty$, establishing the coercivity of $\mathcal{W}_h$.
\end{proof}

Since the functional $\mathcal{W}_h$ is coercive and strictly concave on the non-empty closed convex set $\mathbf{D}_0(V) \cap \left( -h + \overline{\mathcal{A}}_{\alpha^{\dagger}}\right)$, it attains a unique maximizer $v^*$ in this set. It remains to show that $v^*$ lies in the interior $\mathbf{D}_0(V) \cap \left( -h + \mathcal{A}_{\alpha^{\dagger}}\right)$ rather than on the boundary. Once this is established, Proposition \ref{prop:potential_vertex} implies that $h + v^* \in P(\Theta, \alpha^{\dagger})$ with the projection $p_V(h + v^*) = h$. This argument is similar to Rivin's proof in the finite-dimensional setting \cite{Rivin}.

\begin{proposition}\label{prop:maxinterior}
	The maximizer $v^*$ of the functional $\mathcal{W}_h$ lies in the interior $\mathbf{D}_0(V) \cap \left( -h + \mathcal{A}_{\alpha^{\dagger}}\right)$.
\end{proposition}
\begin{proof}
	We take an element $v$ in the interior of the convex set $\mathbf{D}_0(V) \cap \left( -h + \overline{\mathcal{A}}_{\alpha^{\dagger}}\right)$ and for $t\in [0,1]$, we consider the linear interpolation $v^t = v^* + t (v-v^*)$ as well as the restriction of the function
	\[f(t) := \mathcal{W}_h(v^t).
	\]
    Our goal is to show that $f(t)> f(0)$ for sufficiently small $t>0$, which contradicts the maximality of $v^*$ on the set.

    We denote by $E_0 \subset E$ the set of edges where the angle constraints are saturated: For each $ij \in E_0$ with left face $\phi$ and right face $\psi$, we have either
\[\alpha^\dagger_{\phi \psi} + \frac{h_j + v^*_j - h_i - v^*_i}{2} = 0 \quad \text{or} \quad \alpha^\dagger_{\psi \phi} - \frac{h_j + v^*_j - h_i - v^*_i}{2} = 0. \]
And we denote by $E_1 = E \setminus E_0$ the set of edges where the angle constraints are strict. 

We observe that $E_0$ is a finite set. Since the half-angles $\alpha^\dagger$ are uniformly bounded between $\eta$ and $\pi - \eta$, the saturation of an angle constraint on an edge $ij$ implies that
\[|h_j + v^*_j - h_i - v^*_i| \geq \eta.\]
Given that $h+ v^*$ has finite Dirichlet energy, there can only be finitely many edges satisfying this inequality.

On the other hand, we claim that there exists $\tilde{\eta} >0$ such that for $t\in [0,1]$ and for every edge $ij \in E_1$,
\[\tilde{\eta} < \alpha^\dagger_{\phi \psi} + \frac{h_j + v^t_j - h_i - v^t_i}{2} < \pi - \tilde{\eta},\]
and
\[\tilde{\eta} < \alpha^\dagger_{\psi \phi} - \frac{h_j + v^t_j - h_i - v^t_i}{2} < \pi - \tilde{\eta}.\]
It suffices to establish the bounds for $t=0$ and $t=1$, since the angle expressions depend linearly on $t$. 

We focus on the half-angles corresponding to $v$. Since $v$ has finite Dirichlet energy, there are only finitely many edges $ij \in E_1$ such that the half-angles are not bounded away from $\eta/2$ and $\pi - \eta/2$. For the remaining finite edges, the half-angles can be uniformly bounded away from $0$ and $\pi$. A similar argument applies to the half-angles corresponding to $v^*$ over the edge set $E_1$. Thus, we conclude the existence of such $\tilde{\eta} >0$.

We compute the derivative:
\begin{align*}
	f'(t) 
	=& \sum_{ij \in E} \mathcal{K}_{ij}(h + v^t) \cdot ( (v_i - v^*_i) - (v_j - v^*_j) ) \\
	=& \sum_{ij \in E_1} \mathcal{K}_{ij}(h + v^t) \cdot ( (v_i - v^*_i) - (v_j - v^*_j) ) \\&+ \sum_{ij \in E_0} \mathcal{K}_{ij}(h + v^t) \cdot ( (v_i - v^*_i) - (v_j - v^*_j) ).
\end{align*}
Lemma \ref{lem:Kbounded} ensures that there exists some constant $C=C(\tilde{\eta},h,v,v^*)>0$ such that for every edge $ij \in E_1$ and $t \in [0,1]$,
\[|\mathcal{K}_{ij}(h + v^t)| \leq C| (v^t_i - v^t_j) + (h_i - h_j)| \leq C( |v_i - v_j| + |v^*_i - v^*_j| + |h_i - h_j|).\]
By Cauchy-Schwarz inequality, we have
\begin{align*}
	&\sum_{ij \in E_1} \mathcal{K}_{ij}(h + v^t) \cdot ( (v_i - v^*_i) - (v_j - v^*_j) ) \\ \geq & -C \left( \sum_{ij \in E_1} ( |v_i - v_j| + |v^*_i - v^*_j| + |h_i - h_j| )^2 \right)^{1/2}  \left( \sum_{ij \in E_1} ( (v_i - v^*_i) - (v_j - v^*_j) )^2 \right)^{1/2} \\
	\geq & - \sqrt{3} C \left( \mathcal{E}(v) + \mathcal{E}(v^*) + \mathcal{E}(h) \right)^{1/2} \cdot \left( \mathcal{E}(v - v^*) \right)^{1/2}.
\end{align*}

Observe that for every edge $ij \in E_0$, each term \[ \mathcal{K}_{ij}(h + v^t) \cdot ( (v_i - v^*_i) - (v_j - v^*_j) ) \to +\infty \quad \text{as } t \to 0^+ \]  since 
\[
\log \sin \alpha \to -\infty \quad \text{as } \alpha \to 0^+.
\]

Thus there exists some sufficiently small $t_0>0$ such that $f'(t) > 0$ for every $t \in (0,t_0)$. Hence by the mean value theorem, we have $f(t_0) - f(0) = f'(t_1) t_0 >0$ for some $t_1 \in (0,t_0)$, contradicting the maximality of $v^*$.
\end{proof}

Finally, we arrive at the main theorem of this section.

\begin{theorem}\label{thm:pvhom}
	The projection $p|_{P(\Theta,\alpha^\dagger)}: P(\Theta,\alpha^\dagger) \to \mathbf{HD}(V)$ is a homeomorphism.
\end{theorem}
\begin{proof}
	We have seen that $p|_{P(\Theta,\alpha^\dagger)}$ is a local homeomorphism. It remains to prove that it is a bijection.
	
	Fix an arbitrary $h \in \mathbf{HD}(V)$. We search for $v$ in $\mathbf{D}_0(V)$ such that $h+v \in P(\Theta, \alpha^{\dagger})$. This is equivalent to finding a critical point $v$ of the functional $\mathcal{W}_h$ in the open set $\mathbf{D}_0(V) \cap \left( -h + \mathcal{A}_{\alpha^{\dagger}}\right)$.
	
	We have established that the functional $\mathcal{W}_h$ is strictly concave on the convex set $\mathbf{D}_0(V) \cap \left( -h + \overline{\mathcal{A}}_{\alpha^{\dagger}}\right)$, is coercive (approaches $-\infty$ as $\|v\| \to \infty$), and that its domain has a non-empty interior. Therefore, there exists a unique maximizer $v^* \in \mathbf{D}_0(V) \cap \left( -h + \overline{\mathcal{A}}_{\alpha^{\dagger}}\right)$.
	
	By Proposition \ref{prop:maxinterior}, the maximizer $v^*$ lies in the interior $\mathbf{D}_0(V) \cap \left( -h + \mathcal{A}_{\alpha^{\dagger}}\right)$. Consequently, $v^*$ is a critical point of $\mathcal{W}_h$, which implies by Proposition \ref{prop:potential_vertex} that $h + v^* \in P(\Theta, \alpha^{\dagger})$. Since $p_V(h+v^*) = h$, the map is surjective.
	
	For injectivity, suppose $h + v_1, h + v_2 \in P(\Theta, \alpha^{\dagger})$ for some $v_1, v_2 \in \mathbf{D}_0(V)$. Both $v_1$ and $v_2$ are critical points of the strictly concave functional $\mathcal{W}_h$, which implies $v_1 = v_2$.
	
	Since $p|_{P(\Theta,\alpha^\dagger)}$ is a bijective local homeomorphism between Hilbert manifolds, it is a global homeomorphism.
\end{proof}

\subsection{Conjugation map and isometry}

Elements $u \in P(\Theta, R^{\dagger})$ and $v \in P(\Theta, \alpha^{\dagger})$ are said to be \emph{conjugate} if they describe the same circle pattern realization in the plane (up to global rigid motions and scaling). The relation is derived from the geometry of the Euclidean kites associated with edges. For an oriented edge $ij$ shared by faces $\phi$ and $\psi$, let $R = e^u R^\dagger$ be the new radii. The length of the common chord shared by the circles is given by \[ 2 R_\phi \sin \alpha_{\phi \psi} = 2 R_\psi \sin \alpha_{\psi \phi}.\] Taking the ratio $\frac{R_\psi}{R_\phi} = \frac{\sin \alpha_{\phi \psi}}{\sin \alpha_{\psi \phi}}$ and normalizing by the ratio from the uniformized circle pattern $\frac{R^\dagger_\psi}{R^\dagger_\phi} = \frac{\sin \alpha^\dagger_{\phi \psi}}{\sin \alpha^\dagger_{\psi \phi}}$, we obtain the logarithmic relation:
\begin{align}\label{eq:conjugation}
	u_{\psi} - u_{\phi}  = \log \frac{ \sin (\alpha^\dagger_{\phi \psi} + \frac{v_j - v_i}{2}) \sin \alpha^\dagger_{\psi \phi}  }{ \sin (\alpha^\dagger_{\psi \phi} + \frac{v_i - v_j}{2}) \sin \alpha^\dagger_{\phi \psi} }.
\end{align}
Its linearization leads to the following relation.
\begin{proposition} \label{prop:conjugation_isometry}
	The conjugation relation \eqref{eq:conjugation} defines a homeomorphism $\mathcal{C}: P(\Theta, R^{\dagger})/\mathbb{R} \to P(\Theta, \alpha^{\dagger})/\mathbb{R}$. The differential of $\mathcal{C}$ at $u$ maps $\mathfrak{u} \in \mathbf{HD}_{c^*}(F)$ to its harmonic conjugate $\mathfrak{v} \in \mathbf{HD}_{c}(V)$ satisfying the linear equations:
	\[
	\mathfrak{u}_{\psi} - \mathfrak{u}_{\phi} = c_{ij} (\mathfrak{v}_j - \mathfrak{v}_i).
	\]
	Furthermore, we have $\mathcal{E}_{c^*}(\mathfrak{u}) = \mathcal{E}_c(\mathfrak{v})$.
\end{proposition}
\begin{proof}	
We observe that the geometric edge weights on the primal and dual graphs are reciprocal, $c_{ij} = 1/c^*_{\phi\psi}$. The equality of the Dirichlet energies follows from a direct computation
\[
\sum_{\phi\psi} c^*_{\phi\psi} (\mathfrak{u}_{\psi} - \mathfrak{u}_{\phi})^2 = \sum_{ij} \frac{1}{c_{ij}} c_{ij}^2 (\mathfrak{v}_j - \mathfrak{v}_i)^2 = \sum_{ij} c_{ij} (\mathfrak{v}_j - \mathfrak{v}_i)^2.
\]
\end{proof}

\begin{proof}[Proof of Theorem \ref{thm:hilbertcentral}]
	It is a direct consequence of Theorem \ref{thm:pvhom} and Proposition \ref{prop:conjugation_isometry}.
\end{proof}

\section{Boundary-value mapping and the Hilbert transform}

In this section, we define an analogue of the Hilbert transform as a mapping between the boundary values of conjugate functions. This construction mirrors the classical Hilbert transform on the unit disk, which maps the boundary values of a harmonic function to the boundary values of its conjugate harmonic function.

\subsection{Homeomorphisms to $\mathbf{HD}(\mathbb{D})$ and $\hf$}\label{subsec:restriction}

We show that the spaces of circle patterns $P(\Theta,R^{\dagger})$ and $P(\Theta,\alpha^{\dagger})$ are each homeomorphic to the space of classical harmonic Dirichlet functions $\mathbf{HD}(\mathbb{D})$ on the unit disk, and hence parameterized by the boundary values on the unit circle as half-differentiable functions. Particularly, we prove Theorem \ref{thm:homeoHD} stated in the introduction.

The homeomorphisms are based on Hutchcroft's results \cite{Hutchcroft2019}.

\begin{definition}
	Let $G = (V, E, F)$ be a cell decomposition satisfying the conditions of Definition \ref{def:infintheta}. A straight-line embedding of $G$ into the Euclidean plane is called a $(D,\eta)$-good embedding if there exist global constants $D \ge 1$ and $\eta > 0$ such that the following geometric constraints hold
	\begin{enumerate}
		\item Angle Bound: For every face, all interior angles are within $[\eta,\pi - \eta]$.
		\item Comparable Edges: For any two adjacent edges sharing a common vertex, the ratio of their Euclidean lengths is within $[D^{-1}, D]$.
	\end{enumerate}
	An embedding is \emph{good} if it is a $(D,\eta)$-good embedding for some constants $D$ and $\eta$.
\end{definition}
It follows that every face in a good embedding is convex and has a uniformly bounded number of edges. 

The following proposition follows from Hutchcroft's work \cite[Theorem 1.5]{Hutchcroft2019}. The original statement in \cite{Hutchcroft2019} is specific to the case of circle packings, but in the footnote afterwards, Hutchcroft pointed out that the same results and proofs hold for good embeddings.

\begin{theorem}\cite{Hutchcroft2019} \label{thm:Hutchcroft}
    Let $G = (V, E, F)$ be a cell decomposition satisfying the conditions of Definition \ref{def:infintheta}. Suppose it is realized as a geodesic decomposition of a bounded simply-connected open domain $\Omega \subset \mathbb{C}$ with vertices $z:V \to \Omega$, and the embedding is good. Then the following hold:
	\begin{enumerate}
\item For every discrete harmonic Dirichlet function $h\in \mathbf{HD}(V)$, there exists a unique harmonic Dirichlet function $H\in \mathbf{HD}(\Omega)$ such that $h-H \circ z_{V}\in \mathbf{D}_{0}(V)$. We denote this function $H$ by $\Psi_{V}(h)$.
\item For every classical harmonic Dirichlet function $H\in \mathbf{HD}(\Omega)$, there exists a unique discrete harmonic Dirichlet function $h\in \mathbf{HD}(V)$ such that $h-H \circ z_{V}\in \mathbf{D}_{0}(V)$. We denote this function $h$ by $\Psi_{V}^{-1}(H)$.
\end{enumerate}
Moreover, the mappings $\Psi_{V}:\mathbf{HD}(V)\rightarrow \mathbf{HD}(\Omega)$ and $\Psi_{V}^{-1}:\mathbf{HD}(\Omega)\rightarrow \mathbf{HD}(V)$ are bounded linear isomorphisms.
\end{theorem}

In order to apply Theorem \ref{thm:Hutchcroft}, we first show that our uniformized circle patterns in the unit disk yield good embeddings for the primal graph and the dual graph.

\begin{proposition}\label{prop:primalgood}
	Under the uniformized circle pattern, the primal geodesic decomposition $(V,E,F)$ of the disk  with vertices realized as intersection points of circles is a good embedding. Similarly, the dual geodesic decomposition $(V^*,E^*,F^*)$ of the disk with vertices realized as circumcenters of circles is also a good embedding.
\end{proposition}
\begin{proof}
The uniformized circle pattern determines straight-line embeddings of the primal and dual graphs by realizing vertices as intersection points and circumcenters of circles, respectively. We first verify that there are uniform bounds on the interior angles of faces.

\begin{figure}[htbp]
	\centering
	\includegraphics[width=0.7\textwidth]{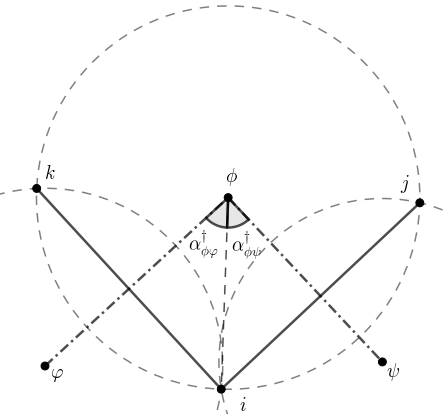}
	\caption{The primal and dual geodesic decompositions of the uniformized circle pattern.}
	\label{fig:primal_dual}
\end{figure}

We focus on two adjacent edges $ij$ and $ik$ sharing a common vertex $i$ in the primal graph (See Figure \ref{fig:primal_dual}). The three vertices lie on the same circle corresponding to the face $\phi$. We write $\phi \psi$ and $\phi \varphi$ for the edges dual to $ij$ and $ik$, respectively. The interior angle $\angle \varphi \phi \psi$ is the sum of half-angles $\alpha^\dagger_{\phi \psi}$ and $\alpha^\dagger_{\phi \varphi}$, which is bounded below by $2\eta$ because of the Ring lemma (Proposition \ref{prop:bound}). On the other hand, the sum of all half-angles around $\phi$ is $\pi$ and there are at least three half-angles around $\phi$. Hence, the interior angle $\angle \varphi \phi \psi$ is also bounded above by $\pi - \eta$. We observe that the dual edges are always perpendicular to the primal edges, and hence the interior angle $\angle k i j$ is equal to $\pi - \angle \varphi \phi \psi$, which is also uniformly bounded within $[\eta, \pi - 2\eta]$.

It remains to verify that adjacent edges have comparable lengths. We observe that the length ratio of adjacent edges $ij$ and $ik$ 
\[
\frac{|z_i-z_j|}{|z_i-z_k|} = \frac{\sin \alpha^\dagger_{\phi \psi}}{\sin \alpha^\dagger_{\phi \varphi}}
\]
which is uniformly bounded within $[D^{-1}, D]$ for $D = 1/\sin \eta$. On the other hand, the length ratio of adjacent dual edges $\phi \psi$ and $\phi \varphi$ is
\[
\frac{|z_\phi - z_\psi|}{|z_\phi - z_\varphi|} = \frac{c_{ij} |z_i-z_j|}{c_{ik}|z_i-z_k|} 
\]
where $c_{ij}$ and $c_{ik}$ are the geometric edge weights defined in Definition \ref{def:geoedgeweight}. By Proposition \ref{prop:bound}, there exists some constant $\tilde{C}>0$ such that for every edge $ij \in E$, the edge weight $c_{ij}$ is uniformly bounded within $[\tilde{C}^{-1}, \tilde{C}]$. Hence, the length ratio of adjacent dual edges $\phi \psi$ and $\phi \varphi$ is also uniformly bounded within $[D^{-1}\tilde{C}^{-2}, D \tilde{C}^2]$.

Thus, the primal and dual geodesic decompositions under the uniformized circle pattern yield good embeddings.
\end{proof}

We are ready to prove that the space of circle patterns $P(\Theta,R^{\dagger})$ and $P(\Theta,\alpha^{\dagger})$ are respectively homeomorphic to the space of classical harmonic Dirichlet functions $\mathbf{HD}(\mathbb{D})$ on the unit disk.
\begin{proof}[Proof of Theorem~\ref{thm:homeoHD}]
	We construct the homeomorphisms $\Phi_V$ and $\Phi_F$ by composing the projections to the discrete harmonic spaces with the isomorphisms provided by Theorem \ref{thm:Hutchcroft}.

	For the primal graph, Proposition \ref{prop:primalgood} establishes that the uniformized circle pattern yields a good embedding. By Theorem \ref{thm:Hutchcroft}, there exists a bounded linear isomorphism $\Psi_V: \mathbf{HD}(V) \to \mathbf{HD}(\mathbb{D})$ such that for any $h \in \mathbf{HD}(V)$, we have $\Psi_V(h) \circ z_V - h \in \mathbf{D}_0(V)$. We define the map $\Phi_V$ as the composition
	\[
	\Phi_V := \Psi_V \circ p_V|_{P(\Theta,\alpha^{\dagger})}.
	\]
	Since $p_V|_{P(\Theta,\alpha^{\dagger})}: P(\Theta,\alpha^{\dagger}) \to \mathbf{HD}(V)$ is a homeomorphism, $\Phi_V$ is a homeomorphism. For any $u \in P(\Theta,\alpha^{\dagger})$, let $h = p_V(u) \in \mathbf{HD}(V)$. By the definition of the projection $p_V$, we have $u - h \in \mathbf{D}_0(V)$. Furthermore, $\Phi_V(u) \circ z_V - h = \Psi_V(h) \circ z_V - h \in \mathbf{D}_0(V)$. Taking the difference yields
	\[
	\Phi_V(u) \circ z_V - u = (\Phi_V(u) \circ z_V - h) - (u - h) \in \mathbf{D}_0(V).
	\]

	For the dual graph, Proposition \ref{prop:primalgood} also establishes that the dual geodesic decomposition yields a good embedding. Applying Theorem \ref{thm:Hutchcroft} to this good embedding, we obtain a bounded linear isomorphism $\Psi_F: \mathbf{HD}(F) \to \mathbf{HD}(\mathbb{D})$ such that $\Psi_F(h) \circ z_F - h \in \mathbf{D}_0(F)$ for any $h \in \mathbf{HD}(F)$. We define $\Phi_F$ as the composition
	\[
	\Phi_F := \Psi_F \circ p_F|_{P(\Theta,R^{\dagger})}.
	\]
	This composition is a homeomorphism since each constituent map is a homeomorphism. For any $v \in P(\Theta,R^{\dagger})$, let $h = p_F(v) \in \mathbf{HD}(F)$. We have $v - h \in \mathbf{D}_0(F)$. Since $\Phi_F(v) \circ z_F - h = \Psi_F(h) \circ z_F - h \in \mathbf{D}_0(F)$, we conclude that
	\[
	\Phi_F(v) \circ z_F - v = (\Phi_F(v) \circ z_F - h) - (v - h) \in \mathbf{D}_0(F).
	\]
\end{proof}

Finally, it leads to an analogue of the Hilbert transform via circle patterns.

\begin{definition}
    We define the boundary value mappings
    \begin{align*}
        \mathcal{B}_V &:= \mathcal{B} \circ \Phi_V \colon P(\Theta,\alpha^{\dagger}) \to \hf, \\
        \mathcal{B}_F &:= \mathcal{B} \circ \Phi_F \colon P(\Theta,R^{\dagger}) \to \hf,
    \end{align*}
    where $\mathcal{B} \colon \mathbf{HD}(\mathbb{D}) \to \hf$ denotes the classical boundary value isomorphism. Because $\mathcal{B}_V$ and $\mathcal{B}_F$ map constant functions to constant functions, they naturally induce homeomorphisms $\overline{\mathcal{B}}_V$ and $\overline{\mathcal{B}}_F$ on the respective quotient spaces modulo constants. 
    
    We define $\mathfrak{H}^{\Theta} \colon \hf/\mathbb{R} \to \hf/\mathbb{R}$, an analogue of the Hilbert transform, as the composition
    \[
    \mathfrak{H}^{\Theta} := \overline{\mathcal{B}}_V \circ \mathcal{C}^{\Theta} \circ \overline{\mathcal{B}}_F^{-1},
    \]
    where $\mathcal{C}^{\Theta} \colon P(\Theta, R^{\dagger})/\mathbb{R} \to P(\Theta, \alpha^{\dagger})/\mathbb{R}$ is the conjugation map.
\end{definition}
\section{Pairing between discrete functions}

We equip the deformation space $P(\Theta, R^{\dagger})/\mathbb{R}$ with a Riemannian metric induced from the Hessian of the functional $\mathcal{W}^*$, which is the Dirichlet energy with respect to the geometric edge weights. With the mapping $\mathfrak{H}^{\Theta}$, we relate this Riemannian metric to the symplectic form on $\hf$, thereby proving Theorem \ref{thm:discretesym} and Corollary \ref{cor:HilbertWP}. 

\begin{definition}
    We define a bilinear form $b \colon \mathbf{D}(F) \times \mathbf{D}(V) \to \mathbb{R}$ by setting, for any $u \in \mathbf{D}(F)$ and $v \in \mathbf{D}(V)$,
    \[
    b(u,v) := \frac{1}{2}\sum_{ij \in \vec{E}} (u_{\psi}-u_{\phi})(v_j-v_i) =  \sum_{ij \in E} (u_{\psi}-u_{\phi})(v_j-v_i),
    \]
    where $\phi$ and $\psi$ are the left and right faces of the oriented edge $ij$, respectively.
\end{definition}

 The pairing is well-defined and acts as a bounded bilinear form due to the Cauchy-Schwarz inequality
\[
b(u,v) \leq \sqrt{\left( \sum_{\phi \psi \in E^*}   (u_{\psi}-u_{\phi})^2 \right) \left(\sum_{ij \in E} (v_j-v_i)^2 \right) }.
\]

To prove the identity in Theorem \ref{thm:discretesym}, we first focus on functions whose boundary values are smooth. We observe that if $\Omega$ is a finite subcomplex of the cell decomposition $G=(V,E,F)$ with $\partial\Omega$ as a polygonal curve oriented in such a way that $\Omega$ is on the left, then
\[
\sum_{ij \in \Omega} (u_{\psi}-u_{\phi})(v_j-v_i) = \sum_{ij \in \partial \Omega} u_{\psi}(v_j-v_i).
\]
The right-hand side is expected to converge to the line integral $\int u dv$ on the unit circle as $\Omega$ exhausts the entire $G$.

In the following two propositions, we construct the necessary sequence of finite subcomplexes $\Omega_k$ that exhausts $G$. We recall that the uniformized circle pattern determines a geodesic cell decomposition of the open unit disk with vertices $z:V \to \mathbb{D}$ such that every face is a cyclic polygon and the central angles are uniformly bounded between $2\eta$ and $\pi-2\eta$ for some $\eta>0$ (Proposition \ref{prop:bound}). The following is a special case of a more general estimate for good embeddings in \cite[Proposition 2.5]{Angel2016}. We sketch an elementary proof for the special case in order to be self-contained.

\begin{proposition}\label{prop:Angel}
   Fix integer $K>0$ and a small constant $\eta>0$. There exists a constant $C_1=C_1(K,\eta)$ such that for any cyclic $k$-polygon with $k \leq K$ and central angles uniformly bounded between $2\eta$ and $\pi-2\eta$, we have for any points $p,q$ on the polygon,
	\[
	d_0(p,q) \leq C_1 |p-q|
	\]
	where $|p-q|$ is the Euclidean distance and $d_0(p,q)$ is the Euclidean length of the shorter polygonal path connecting $p$ and $q$.
\end{proposition}
\begin{proof}
	We outline the proof using a standard compactness argument. Up to uniform scaling, we may assume the vertices of the polygons lie on the unit circle. For a fixed integer $K$ and a constant $\eta > 0$, the moduli space of cyclic $k$-gons with $k \leq K$ and central angles in the closed interval $[2\eta, \pi-2\eta]$ is compact. The ratio $\frac{d_0(p,q)}{|p-q|}$ is continuous and well-defined whenever $p \neq q$. As $p$ approaches $q$, the limit of this ratio is $1$ if the points converge along the same edge. If they converge to a vertex with interior angle $\alpha$ from adjacent edges, the limit is bounded above by $\frac{1}{\sin(\alpha/2)}$. Because the central angles are bounded above by $\pi-2\eta$, the interior angles are uniformly bounded below by $2\eta$, which ensures $\frac{1}{\sin(\alpha/2)} \leq \frac{1}{\sin(\eta)}$. By the compactness of the polygon boundaries and the finite union of the moduli spaces for all $k \leq K$, the ratio attains a finite global maximum $C_1(K, \eta)$.
\end{proof}

We are now ready to construct a sequence of finite subcomplexes $\Omega_n$.

\begin{proposition}\label{prop:seqcomplex}
    For the uniformized circle pattern, there exists a sequence of finite subcomplexes $\Omega_1 \subset \Omega_2 \subset \dots \subset \Omega_n \subset \dots $ that exhausts the cell decomposition $G=(V,E,F)$ and satisfies the following properties:
    \begin{enumerate}
        \item Each $\Omega_n$ is a finite cell decomposition of a topologically closed disk. 
        \item There exists a constant $C_1>1$ such that for all $n$, the boundary $\partial \Omega_n$ has a total Euclidean length bounded above by $2\pi C_1$. 
    \end{enumerate}
\end{proposition}
\begin{proof}
	We construct the sequence by induction. Take $\Omega_1$ to be any face that contains the origin. Suppose a finite subcomplex $\Omega_n$ is defined. There exists $\rho_n \in (1-\frac{1}{n},1)$ such that the circle centered at the origin with radius $\rho_n$ does not intersect any vertex or edge of ${\Omega}_n$. The circle intersects the 1-skeleton graph of $G$ at finitely many points, labeled consecutively as $q_0,q_1,q_2,\dots,q_k=q_0$. We perturb the circle to a polygonal curve as follows. For every pair of consecutive points $q_i, q_{i+1}$, they lie on some common face and can be joined by the shortest path running along the face, whose total length is $d_0(q_i, q_{i+1})\leq C_1 |q_{i+1}-q_i|$ by Proposition \ref{prop:Angel}. Concatenating all these paths and erasing repetitions, we get a polygonal curve consisting of edges in the 1-skeleton graph. We take $\Omega_{n+1}$ to be the bounded component containing the origin. Notice that
	\[
	\text{total length of } \partial \Omega_{n+1} \leq C_1 \sum_i |q_{i+1}-q_i| <C_1 \cdot 2\pi \rho_n <2\pi C_1.
	\]
	By construction, $\Omega_n \subset \Omega_{n+1}$ and the sequence $\{ \Omega_n\}_{n\in \mathbb{N}}$ exhausts the cell decomposition.
\end{proof}

\begin{proposition}\label{prop:riemannappro}
    Suppose $u,v \in \mathbf{HD}(\mathbb{D})$ have boundary values in $C^{\infty}(\partial  \mathbb{D})$ and let $L>0$. For any $\epsilon>0$, there exists $\delta>0$ such that if $\Gamma$ is a loop in the 1-skeleton graph $(V,E)$ where every edge $ij \in \Gamma$ has length $|z_i-z_j| < \delta$ and the total length satisfies $\sum_{ij \in \Gamma} |z_j-z_i| < L$, then
    \[
    \left| \int_{\Gamma} u \, dv - \sum_{ij \in \Gamma} u_{\psi}(v_j-v_i) \right| < \epsilon,
    \]
    where $\psi$ is the face on the left of the oriented edge from $i$ to $j$.
\end{proposition}
\begin{proof}
	Since $u,v \in C^{\infty}(\overline{\mathbb{D}})$, all their partial derivatives are bounded.
	
	Since $u$ is uniformly continuous and $\partial_x v, \partial_y v$ are bounded on $\mathbb{D}$, there exists $\delta_1>0$ such that for all $p,q \in \mathbb{D}$ with $|p-q|< \delta_1$
	\[
	|u(p)-u(q)| < \frac{\epsilon}{4L \cdot \max\{ |\partial_x v|, |\partial_y v|,1\}}.
	\]
	Observe that for any edge $ij$, the neighbouring circle has radius $R^{\dagger}_{\psi} = \frac{|z_j-z_i|}{2 \sin \alpha^{\dagger}_{\psi \phi}} <  \frac{|z_j-z_i|}{2 \sin \eta}$. Then for any edge segment $ij$ with length $|z_j-z_i| < 2 \delta_1 \sin  \eta$, the distance of every point on the line segment $[z_i,z_j]$ from the neighboring circumcenter $z_{\psi}$ is less than $R^{\dagger}_{\psi}<\delta_1$. Thus
	\begin{align*}
		|\int_{[z_i,z_j]} u \frac{dv}{ds} ds - 	\int_{[z_i,z_j]} u_{\psi} \frac{dv}{ds} ds| &\leq 	2\max\{ |\partial_x v|, |\partial_y v|\} (\int_{[z_i,z_j]} |u- u_{ij,r} | ds)  \\ &<  \frac{|z_i-z_j|}{2L} \epsilon
	\end{align*}
	where $s$ is the arc-length parametrization of the line segment.
	
	On the other hand, $\partial_{xx} v, \partial_{xy} v$ and $\partial_{yy}v$ are bounded over $\mathbb{D}$. Take 
	\[
	\delta_2 = \frac{\epsilon}{2L \max\{|\partial_{xx} v|, |\partial_{xy} v|, |\partial_{yy}v|,1\} \cdot \max |u|}.
	\]
	Then along the line segment $[z_i,z_j]$ with $|z_i-z_j| <\delta_2$, we have by Taylor's theorem, for $p \in [z_i,z_j]$,
	\[
	|\frac{v_j-v_i}{|z_j-z_i|} - \frac{dv}{ds}(p)| < |z_j-z_i| \max\{|\partial_{xx} v|, |\partial_{xy} v|, |\partial_{yy}v|\} < \frac{\epsilon}{2L \max |u|}.
	\]
	Thus
	\begin{align*}
		|	\int_{[z_i,z_j]} u_{\psi} \frac{dv}{ds} ds - u_{\psi}(v_j-v_i)| &< |u_{\psi}| \int_{[z_i,z_j]} |\frac{dv}{ds}- \frac{v_j-v_i}{|z_j-z_i|}| ds \\
		&< \frac{|z_i-z_j|}{2L} \epsilon.
	\end{align*}
	Hence, for a polygon curve $\Gamma$ with every edge length $|z_j-z_i| < \min \{2 \delta_1 \sin  \eta,  \delta_2\}$, one has
	\[
	|\int_{[z_i,z_j]} u dv    - u_{\psi}(v_j-v_i)| < \frac{|z_j-z_i|}{L} \epsilon
	\]
	and
	\[
	|\int_{\Gamma} u dv -  \sum_{ij \in \Gamma} u_{\psi}(v_j-v_i)| < \epsilon.
	\]
\end{proof}

\begin{lemma}\label{cor:smalllengths}
	Fix $\delta>0$. Under the uniformized circle pattern, the set of edges with Euclidean length  $|z_j-z_i| > \delta$ is a finite set.
\end{lemma}
\begin{proof}
    We observe that there exists a constant $C_2>0$ such that for any edge $ij$, the area of the cyclic face containing $ij$ is at least $C_2 |z_j-z_i|^2$. Indeed, since the ratio of the area of the face to $|z_j-z_i|^2$ is invariant under scaling, we can apply the compactness argument as in the proof of Proposition \ref{prop:Angel} to show that there is a uniform lower bound $C_2$ for this ratio for all cyclic $k$-polygon with $k \leq K$ and central angles in the closed interval $[2\eta, \pi-2\eta]$. Since the total area of the unit disk is $\pi$, there can only be finitely many edges with length greater than a fixed $\delta$.
\end{proof}

\begin{proposition}
	Let $u,v \in  \mathbf{HD}(\mathbb{D})$ with boundary values in $C^{\infty}(\partial \mathbb{D})$. Under the uniformized circle pattern, we denote by $z_F: F \to \mathbb{D}$ the circumcenter of the circles and by $z_V: V \to \mathbb{D}$ the intersection points of circles. Then
	\[
	b(u\circ z_F,v\circ z_V) = \int_{\partial \mathbb{D}} u dv = 2 \pi  \omega(u,v).
	\]
\end{proposition}
\begin{proof}
   In the following, we write $u_\phi = u(z_\phi)$ and $v_i = v(z_i)$ for simplicity. We take the sequence of finite subcomplexes $\Omega_k$ as in Proposition \ref{prop:seqcomplex}.
	
	Let $\epsilon >0$. Then there exists an integer $N_1$ such that for $n>N_1$,
	\[
	|\int_{\partial \mathbb{D}} u dv - \int_{\partial \Omega_n} u dv| = | \iint_{\mathbb{D}-\Omega_n} du \wedge dv| < \epsilon.
	\]
	We take $L= 2 \pi C_1$. Applying Proposition \ref{prop:riemannappro} and Lemma \ref{cor:smalllengths} implies that there exists another integer $N_2$ such that for $n>N_2$, 
	\[
	|\int_{\partial \Omega_n} u dv - \sum_{ij \in \partial \Omega_n} u_{\psi}(v_j-v_i)| < \epsilon.
	\]
	Thus for $n> \max\{N_1,N_2\}$, we have
	\[
	|2\pi \omega(u,v) -  \sum_{ij \in \Omega_n} (u_{\psi}-u_{\phi})(v_j-v_i) | = |\int_{\partial \mathbb{D}} u dv - \sum_{ij \in \partial \Omega_n} u_{\psi}(v_j-v_i)| < 2 \epsilon.
	\] 
	Hence
	\[
	b(u\circ z_F,v\circ z_V) = \lim_{n \to \infty} \sum_{ij \in \Omega_n} (u_{\psi}-u_{\phi})(v_j-v_i) = 2\pi \omega(u,v).
	\]
\end{proof}

\begin{lemma}\label{lem:bD0}
	If $u \in \mathbf{D}_0(F)$ and $v\in \mathbf{D}(V)$, then
	\[
	b(u,v)=0.
	\]
	Similarly, if $u \in \mathbf{D}(F)$ and $v\in \mathbf{D}_0(V)$, then 
	\[
	b(u,v)=0.
	\]
\end{lemma}
\begin{proof}
	Suppose $u \in \mathbf{D}_0(F)$ has finite support. Then
   \[
 b(u,v)= - \sum_{\phi \in F} u_{\phi} \sum_{ij \in \partial \phi} (v_j-v_i) =0.  
   \]
   By continuity, the same holds for any $u \in \mathbf{D}_0(F)$. The other case is similar.
\end{proof}

\begin{lemma}\label{lem:bHDD}
Let $u,v \in \hf$ and $\hat{u},\hat{v} \in \mathbf{HD}(\mathbb{D})$ be their harmonic extensions. Then
	\[
	b(\hat{u}\circ z_F,\hat{v}\circ z_V) = 2 \pi \omega(u,v).
	\]
\end{lemma}
\begin{proof}
    The space of smooth functions on the unit circle is dense in $\hf$. We can find sequences of smooth functions $\{u_m\}_{m \in \mathbb{N}}$ and $\{v_n\}_{n \in \mathbb{N}}$ that converge to $u$ and $v$ in $\hf$. The corresponding harmonic extensions $\hat{u}_m, \hat{v}_n$ and $\hat{u}, \hat{v}$ converge in the classical Dirichlet energy norm. By Hutchcroft's results (Theorem \ref{thm:Hutchcroft}), the discrete harmonic functions $p_F(\hat{u}_m \circ z_F)$ and $p_V(\hat{v}_n \circ z_V)$ converge to $p_F(\hat{u} \circ z_F)$ and $p_V(\hat{v} \circ z_V)$ in the combinatorial Dirichlet energy norm. Hence
	\begin{align*}
		b(\hat{u}\circ z_F,\hat{v}\circ z_V) =& b(p_F(\hat{u} \circ z_F), p_V(\hat{v} \circ z_V))\\=& \lim_{m,n \to \infty} b(p_F(\hat{u}_m \circ z_F), p_V(\hat{v}_n \circ z_V)) \\
		=&\lim_{m,n \to \infty} b((\hat{u}_m \circ z_F), (\hat{v}_n \circ z_V)) \\=& 2\pi \lim_{m,n \to \infty} \omega(u_m,v_n) \\=& 2\pi \omega(u,v).
	\end{align*}
\end{proof}

\begin{proof}[Proof of Theorem \ref{thm:discretesym}]
Suppose $u \in \mathbf{D}(F)$ and $v \in \mathbf{D}(V)$ have boundary values $u_{\partial \mathbb{D}}, v_{\partial \mathbb{D}} \in \hf$ respectively. By the construction of the boundary value mapping, there exist harmonic extensions $\hat{u}, \hat{v} \in \mathbf{HD}(\mathbb{D})$ of $u_{\partial \mathbb{D}}$ and $v_{\partial \mathbb{D}}$ respectively, such that $u - \hat{u} \circ z_F \in \mathbf{D}_0(F)$ and $v - \hat{v} \circ z_V \in \mathbf{D}_0(V)$. 
Using the bilinearity of $b$, Lemma \ref{lem:bD0} and Lemma \ref{lem:bHDD} implies that
\begin{align*}
	b(u,v) &= b(\hat{u} \circ z_F + (u - \hat{u} \circ z_F), \hat{v} \circ z_V + (v - \hat{v} \circ z_V)) \\
	&= b(\hat{u} \circ z_F, \hat{v} \circ z_V) \\
	&= 2\pi \omega(u_{\partial \mathbb{D}}, v_{\partial \mathbb{D}}).
\end{align*}
\end{proof}

\begin{proof}[Proof of Corollary \ref{cor:HilbertWP}]
By Proposition \ref{prop:conjugation_isometry}, the differential of the conjugation map at $u$ maps $w \in T_u P(\Theta, R^\dagger) = \mathbf{HD}_{c^*}(F)$ to its harmonic conjugate $\tilde{w} \in T_{[v]} (P(\Theta, \alpha^\dagger)/\mathbb{R}) = \mathbf{HD}_c(V)$, which satisfies
\[
w_\psi - w_\phi = c_{ij} (\tilde{w}_j - \tilde{w}_i).
\]
Using this relation and the fact that $c^*_{\phi\psi} = 1/c_{ij}$, we can rewrite the geometric Dirichlet inner product as
\begin{align*}
	\sum_{\phi\psi \in E^*} c^*_{\phi\psi} (v_\phi - v_\psi)(w_\phi - w_\psi) &= \sum_{ij \in E} (v_\psi - v_\phi) c^*_{\phi\psi} (w_\psi - w_\phi) \\
	&= \sum_{ij \in E} (v_\psi - v_\phi) (\tilde{w}_j - \tilde{w}_i) \\
	&= b(v, \tilde{w}).
\end{align*}
By Theorem \ref{thm:discretesym}, we have
\[
b(v, \tilde{w}) = 2\pi \cdot \omega(v_{\partial \mathbb{D}}, \tilde{w}_{\partial \mathbb{D}})
\]
and $\tilde{w}_{\partial \mathbb{D}} = d\mathfrak{H}^{\Theta}_{u_{\partial \mathbb{D}}}(w_{\partial \mathbb{D}})$, which completes the proof.
\end{proof}

\section{Projection to the universal Teichm\"{u}ller space}

We show that every circle pattern in $P(\Theta,R^{\dagger})$ yields a quasiconformal homeomorphism of the unit disk. This induces a natural projection from $P(\Theta,R^{\dagger})$ to the universal Teichm\"{u}ller space $\UT$. We further prove that the image of $P(\Theta,R^{\dagger})$ lies within the Weil-Petersson subclass $\WT$. Our proof is motivated by the arguments in \cite{Wang2024}.

For any circle pattern $u \in P(\Theta,R^{\dagger})$ (or, more generally, $u \in \tilde{P}(\Theta)$), the Euclidean structure $(\Omega,\sigma(\Theta,e^u R^{\dagger}))$ naturally admits a geodesic cell decomposition. In particular, the uniformized circle pattern induces a geodesic cell decomposition of the unit disk $\mathbb{D} \cong (\Omega,\sigma(\Theta,R^{\dagger}))$. 

These decompositions allow us to construct a piecewise-linear map $f_{u}:\mathbb{D} \to (\Omega,\sigma(\Theta,e^u R^{\dagger}))$. Specifically, we define $f_{u}$ so that it maps the intersection points and circumcenters of the uniformized pattern to the corresponding points in the deformed circle pattern. For every edge $ij \in E$ with adjacent faces $\phi$ and $\psi$, we uniquely extend $f_{u}$ to an affine map over the triangles $\phi \psi i$ and $\psi \phi j$. In the following, we demonstrate that $f_{u}$ is quasiconformal and analyze its Beltrami differential.

We let $\alpha^{\dagger}_{\phi \psi}, \alpha^{\dagger}_{\psi \phi}, \Theta_{\phi\psi}$ denote the interior angles of the triangles $\phi \psi i$ and $\psi \phi j$ under the uniformized circle pattern, while let $\alpha_{\phi \psi}, \alpha_{\psi \phi}, \Theta_{\phi\psi}$ denote the corresponding interior angles under the deformed circle pattern with Euclidean radii $e^u R^{\dagger}$.

\begin{lemma}\label{lem:beltri}
	Suppose $L$ is a linear map that maps a Euclidean triangle with inner angles $\alpha,\beta,\gamma= \pi - \alpha -\beta$ to another Euclidean triangle with inner angles $\tilde{\alpha},\tilde{\beta},\tilde{\gamma}=\pi - \tilde{\alpha} - \tilde{\beta}$. Then its Beltrami differential $\mu(L) := \frac{\partial_{\bar{z}} L}{\partial_z L}$ is constant on the interior of the triangle and 
    \begin{align*}
    		\mu(L)=  \frac{ \frac{\sin \tilde{\gamma}}{\sin \gamma} - \frac{\sin \tilde{\alpha} \sin \tilde{\beta}}{\sin \alpha \sin \beta} + \mathbf{i} (\frac{\cos \tilde{\alpha} \sin \tilde{\beta}}{\sin \alpha \sin \beta}- \frac{\cos \alpha  \sin \tilde{\gamma}}{\sin \alpha \sin \gamma})}{\frac{\sin \tilde{\gamma}}{\sin \gamma} + \frac{\sin \tilde{\alpha} \sin \tilde{\beta}}{\sin \alpha \sin \beta} - \mathbf{i} (\frac{\cos \tilde{\alpha} \sin \tilde{\beta}}{\sin \alpha \sin \beta}- \frac{\cos \alpha  \sin \tilde{\gamma}}{\sin \alpha \sin \gamma})}.
    \end{align*}	
  Furthermore,
  \begin{align*}
  	|\mu(L)|^2= 1-  \frac{4 \cdot \frac{\sin \tilde{\alpha} \sin \tilde{\beta} \sin \tilde{\gamma}}{\sin \alpha \sin \beta \sin \gamma} }{|\frac{\sin \tilde{\gamma}}{\sin \gamma} + \frac{\sin \tilde{\alpha} \sin \tilde{\beta}}{\sin \alpha \sin \beta}|^2 + |\frac{\cos \tilde{\alpha} \sin \tilde{\beta}}{\sin \alpha \sin \beta}- \frac{\cos \alpha  \sin \tilde{\gamma}}{\sin \alpha \sin \gamma}|^2}.
  \end{align*}
\end{lemma}
\begin{proof}
	The Beltrami differential is invariant under scaling and translation. Given the interior angles, the lengths of the triangle are $\sin \alpha, \sin \beta, \sin \gamma$ up to scaling. We can position the vertices as $(0,0)$, $(\sin \gamma,0)$, $(\sin \beta \cos \alpha, \sin \beta \sin \alpha)$. With these coordinates, the linear map $L$ mapping $(0,0)$, $(\sin \gamma,0)$, $(\sin \beta \cos \alpha, \sin \beta \sin \alpha)$ to $(0,0)$, $(\sin \tilde{\gamma},0)$, $(\sin \tilde{\beta} \cos \tilde{\alpha}, \sin \tilde{\beta} \sin \tilde{\alpha})$ is given by
	\[
  L=\left(\begin{array}{cc}
  	\sin \tilde{\gamma} & \sin \tilde{\beta} \cos \tilde{\alpha} \\ 0 & \sin \tilde{\beta} \sin \tilde{\alpha}
  \end{array}\right)  \left(\begin{array}{cc}
  \sin \gamma & \sin \beta \cos \alpha \\ 0 & \sin \beta \sin \alpha
\end{array}\right) ^{-1} 
	\]
	and the Beltrami differential can be computed explicitly.
\end{proof}

\begin{proposition}\label{prop:quasiLinfin}
	Suppose $u \in \tilde{P}(\Theta)$. Then we have $\sup_{\phi \psi \in E^*} |u_{\phi}-u_{\psi}| < \infty$ if and only if the piecewise-linear map $f_{u}$ is a quasiconformal map.
\end{proposition}
\begin{proof}
	Proposition \ref{prop:bound} implies that there is an arbitrarily small $\eta=\eta(\Theta,\epsilon,u)>0$ such that the interior angles of the triangles are within $\eta$ and $\pi- \eta$. Notice that for any angle $\alpha,\beta,\gamma,\tilde{\alpha}, \tilde{\beta},\tilde{\gamma} \in (\eta, \pi - \eta)$
	\begin{align*}
		 \frac{4 \cdot \frac{\sin \tilde{\alpha} \sin \tilde{\beta} \sin \tilde{\gamma}}{\sin \alpha \sin \beta \sin \gamma} }{|\frac{\sin \tilde{\gamma}}{\sin \gamma} + \frac{\sin \tilde{\alpha} \sin \tilde{\beta}}{\sin \alpha \sin \beta}|^2 + |\frac{\cos \tilde{\alpha} \sin \tilde{\beta}}{\sin \alpha \sin \beta}- \frac{\cos \alpha  \sin \tilde{\gamma}}{\sin \alpha \sin \gamma}|^2} \geq \frac{4 \cdot\frac{\sin^3 \eta}{1}}{ 2 (\frac{1}{\sin^2 \eta} + 3 \frac{1}{\sin^4 \eta})}.
	\end{align*}
	Hence by Lemma \ref{lem:beltri}, $\mu(f)$ is uniformly bounded away from $1$ and hence $f$ is quasiconformal.
\end{proof}

Since $(\Omega,\sigma(\Theta,e^u R^\dagger))$ is a simply connected Riemann surface and admits a quasiconformal homeomorphism $f_{u}$ from the unit disk, we deduce that $(\Omega,\sigma(\Theta,e^u R^\dagger))$ is conformally equivalent to the unit disk and there exists a Riemann mapping $g_u: \mathbb{D} \to (\Omega,\sigma(\Theta,e^u R^\dagger))$. It induces a quasiconformal homeomorphism of the unit disk $h_u : \mathbb{D} \to \mathbb{D}$ 
\[
h_u := g_u^{-1} \circ f_{u}
\]
and its extension to the unit circle is a quasisymmetric homeomorphism, which represents an element in the universal Teichm\"{u}ller space. We claim that it indeed lies in the Weil-Petersson class. Since the Beltrami differential is invariant under post composition of a conformal map, it suffices to show that the Beltrami differential of $f_{\mathrm{PL}}$ is $L^2$-integrable over the unit disk with respect to the hyperbolic metric.

\begin{lemma}\label{lem:area}
	For the uniformized circle pattern on the unit disk, every circumcircle is compact. Furthermore there exists $A_0 >0$ such that the hyperbolic area of every circumdisk is bounded above by $A_0$. 
\end{lemma}
\begin{proof}
	Since each circle is surrounded by finitely many circles, the central circle has to be compact, i.e. not touching the unit circle nor intersecting it. In order to obtain an upper bound on the hyperbolic area of the circumcircles, we focus on one face together with its neighbouring faces and disregard all other faces. In this case, the central circumcircle achieves its maximum hyperbolic radius with given intersection angles when all the neighboring circumcircles become horocycles. In this configuration, the hyperbolic area of the central circumdisk is a continuous function of the intersection angles. Since the intersection angles vary in a compact set $[\epsilon_0, \pi-\epsilon_0]$ and the face degree is uniformly bounded, there exists an upper bound $A>0$ on the hyperbolic area of the circumdisks in the uniformized circle pattern on the disk.
\end{proof}

For the deformed circle pattern corresponding to $u \in P(\Theta,R^\dagger)$, we recall that the interior angles of the triangles are $\alpha_{\phi \psi}, \alpha_{\psi \phi}, \Theta_{\phi \psi}$. The angles $\alpha_{\phi \psi}, \alpha_{\psi \phi}$ depend on the ratios of the radii of the neighbouring circumcircles and the difference of $u$
\begin{align*}
	 \alpha_{\phi \psi} &=\cot^{-1}\left(\frac{e^{u_{\phi}-u_{\psi}}\frac{R_{\phi}}{R_{\psi}} - \cos \Theta_{\phi\psi}}{\sin \Theta_{\phi\psi}}\right), \\
	 \alpha_{\psi \phi} &= \cot^{-1}\left(\frac{e^{u_{\psi}-u_{\phi}}\frac{R_{\psi}}{R_{\phi}} - \cos \Theta_{\phi\psi}}{\sin \Theta_{\phi\psi}}\right) = \pi - \Theta_{\phi\psi} - \alpha_{\phi \psi},
\end{align*}
where 
\begin{align*}
\Theta_{\phi\psi} \in [\epsilon_0,\pi-\epsilon_0], \quad u_{\phi}-u_{\psi}\in [-\sqrt{\mathcal{E}(u)},\sqrt{\mathcal{E}(u)}], \quad \frac{R_{\phi}}{R_{\psi}} \in [\frac{1}{C},C]
\end{align*}
with $C=C(\Theta,\epsilon_0)$ being the constant in the Ring lemma \ref{lem:ring}. In particular, the constants $\sqrt{\mathcal{E}(u)}, C, \epsilon_0$ are independent of edges $ij$.

\begin{lemma}\label{lem:smalldv}
	We consider the following analytic function for $t \in [-\sqrt{\mathcal{E}(u)},\sqrt{\mathcal{E}(u)}]$, $r \in [\frac{1}{C},C]$ and $\theta \in [\epsilon_0,\pi-\epsilon_0]$
	\begin{align}
		m(t,r,\theta):= \frac{ \frac{\sin \tilde{\gamma}}{\sin \gamma} - \frac{\sin \tilde{\alpha} \sin \tilde{\beta}}{\sin \alpha \sin \beta} + \mathbf{i} (\frac{\cos \tilde{\alpha} \sin \tilde{\beta}}{\sin \alpha \sin \beta}- \frac{\cos \alpha  \sin \tilde{\gamma}}{\sin \alpha \sin \gamma})}{\frac{\sin \tilde{\gamma}}{\sin \gamma} + \frac{\sin \tilde{\alpha} \sin \tilde{\beta}}{\sin \alpha \sin \beta} - \mathbf{i} (\frac{\cos \tilde{\alpha} \sin \tilde{\beta}}{\sin \alpha \sin \beta}- \frac{\cos \alpha  \sin \tilde{\gamma}}{\sin \alpha \sin \gamma})} 
	\end{align}
where 
\begin{align*}
	\tilde{\alpha} &= \cot^{-1}\left(\frac{e^{t}r - \cos \theta}{\sin \theta}\right), \quad
	\tilde{\beta} = \pi-\theta-\tilde{\alpha}, \quad
	\tilde{\gamma} = \theta \\
	\alpha &= \cot^{-1}\left(\frac{r - \cos \theta}{\sin \theta}\right), \quad
	\beta = \pi-\theta-\alpha, \quad
	\gamma = \theta. 
\end{align*}
Then there exists a constant $C_1,C_2$ depending on $\sqrt{\mathcal{E}(u)},\Theta,\epsilon_0$ such that  for $t \in [-\sqrt{\mathcal{E}(u)},\sqrt{\mathcal{E}(u)}]$, $r \in [\frac{1}{C},C]$ and $\theta \in [\epsilon_0,\pi-\epsilon_0]$
\[
	C_1 |t| \geq |m(t,r,\theta)| \geq C_2 |t|.
\]
\end{lemma}
\begin{proof}
	Since $m$ is analytic on a compact set, we consider the values
	\begin{align*}
	C_1= \max_{t \in [-\sqrt{\mathcal{E}(u)},\sqrt{\mathcal{E}(u)}], r \in [\frac{1}{C},C], \theta \in [\epsilon_0,\pi-\epsilon_0]} |\partial_t m(t,r,\theta)|,\\
	C_2= \min_{t \in [-\sqrt{\mathcal{E}(u)},\sqrt{\mathcal{E}(u)}], r \in [\frac{1}{C},C], \theta \in [\epsilon_0,\pi-\epsilon_0]} |\partial_t m(t,r,\theta)|.	
	\end{align*}
	By the mean value theorem, we have for all $t \in [-\sqrt{\mathcal{E}(u)},\sqrt{\mathcal{E}(u)}]$, $r \in [\frac{1}{C},C]$ and $\theta \in [\epsilon_0,\pi-\epsilon_0]$,
	\[
	|m(t,r,\theta)-m(0,r,\theta)| = |t| \cdot |\partial_t m(\xi,r,\theta)| \leq C_1 |t|
	\]
	where $\xi$ is between $0$ and $t$. We observe that $m(0,r,\theta)=0$ and thus obtain the desired estimate. The lower bound can be obtained similarly.
\end{proof}

\begin{proof}[Proof of Theorem \ref{thm:projectionWP}]
    The first equivalence asserted in the theorem follows from Proposition \ref{prop:quasiLinfin}. 

	We now show that if $u$ has finite combinatorial Dirichlet energy, then the Beltrami differential of $f_{u}$ is $L^2$-integrable over the unit disk with respect to the hyperbolic metric. Since every Euclidean triangular face $\phi\psi i$ is covered by the adjacent circumdisks, its hyperbolic area $A_{\mathrm{hyp}}(\phi\psi i)$ is bounded above by $2A_0$ by Lemma \ref{lem:area}. Then with respect to the hyperbolic metric over the unit disk and with the constant $C_1$ from Lemma \ref{lem:smalldv}, we have
	\begin{align*}
		\iint_{\mathbb{D}} |\mu(f_u)|^2 d A_{\mathrm{hyp}}&=\sum_{\phi\psi \in E^*} \left( |\mu(f_u)_{\phi\psi i}|^2 A_{\mathrm{hyp}}(\phi\psi i) + |\mu(f_u)_{\phi\psi j}|^2 A_{\mathrm{hyp}}(\phi\psi j) \right)\\
		&\leq \sum_{\phi\psi \in E^*} C_1^2 |u_{\phi}-u_{\psi}|^2 \cdot 2A_0 \cdot 2 \\
		&= 4C_1^2 A_0 \sum_{\phi\psi \in E^*} |u_{\phi}-u_{\psi}|^2 
	\end{align*}
which is finite since $u \in \mathbf{D}(F)$.

    We now prove the converse direction with an additional assumption that the uniformized circle pattern has a uniform lower bound on the hyperbolic radii $r_{\min}>0$. Assume the Beltrami differential $\mu(f_{u})$ is $L^2$-integrable with respect to the hyperbolic metric on $\mathbb{D}$. 
	
	First, Lemma \ref{lem:smalldv} implies that there exists a constant $C_2 > 0$ such that $|\mu(f_u)_{\phi\psi i}| \geq C_2 |u_\phi - u_\psi|$ for all edges $\phi\psi \in E^*$.  
    
    Second, we require a lower bound for the hyperbolic area of the Euclidean triangles $\phi\psi i$. Using the explicit formulas for the Euclidean center $c$ and radius $R$ of a circle in Poincaré disk with hyperbolic center $a$ and hyperbolic radius $r$ \cite[Exercise 3]{BeardonMinda2005}, we can derive the ratio of the Euclidean radius to the Euclidean distance from the boundary
\[
\frac{R}{1-|c|} = \tanh(\frac{r}{2}) \frac{1+|a|}{1+|a|\tanh^2(\frac{r}{2})} .
\]
Observing the monotonicity of this function with respect to the location of the hyperbolic center $|a|$, one obtains the following lower bound
\[
 \frac{R}{1-|c|} \geq \tanh(r/2) \geq \tanh(r_{\min}/2).
\]
    
    Now consider a point $z$ in the Euclidean triangle $\phi\psi i$ under the uniformized circle pattern. Since the triangle is contained in the union of circumdisks, and the circumdisks are subsets of $\mathbb{D}$, we have \[ |z-z_\phi| \leq R_{\phi}+ R_{\psi} \leq (1+C)R_{\phi} \leq (C+1)(1-|z_\phi|).\] This implies \[ 1-|z| \leq 1-|z_\phi| + |z-z_\phi| \leq (C+2)(1-|z_\phi|).\] Consequently, the hyperbolic density satisfies a lower bound:
    \[ \frac{4}{(1-|z|^2)^2} \geq \frac{1}{(1-|z|)^2} \geq \frac{1}{(C+2)^2(1-|z_\phi|)^2}. \]
    On the other hand, the Euclidean area of $\triangle \phi\psi i$ is given by \[
	\frac{1}{2} R_\phi R_\psi \sin \Theta_{\phi\psi} \geq \frac{\sin \epsilon_0}{2C}  R^2_\phi.  \] 
    Combining these estimates, we obtain the lower bound for the hyperbolic area:
    \[
        A_{\mathrm{hyp}}(\phi\psi i)  = \iint_{\triangle \phi\psi i} \frac{4}{(1-|z|^2)^2} dx dy \geq \frac{\sin \epsilon_0}{2C(C+2)^2} \left(\frac{R_\phi}{1-|z_\phi|}\right)^2 \geq A_{\min}
    \]
    where $A_{\min} = \frac{\sin \epsilon_0}{2C(C+2)^2} \tanh^2(r_{\min}/2)$ is a positive constant independent of the edge $\phi\psi$. Thus, we have
	\begin{align*}
		\iint_{\mathbb{D}} |\mu(f_u)|^2 d A_{\mathrm{hyp}} &= \sum_{\phi\psi \in E^*} \left( |\mu(f_u)_{\phi\psi i}|^2 A_{\mathrm{hyp}}(\phi\psi i) + |\mu(f_u)_{\phi\psi j}|^2 A_{\mathrm{hyp}}(\psi\phi j) \right) \\ &\geq 2 C_2^2 A_{\min} \sum_{\phi\psi \in E^*} |u_\phi - u_\psi|^2.
	\end{align*}
    Hence,
	\[
	\sum_{\phi\psi \in E^*} |u_\phi - u_\psi|^2 \leq \frac{1}{2 C_2^2 A_{\min}} \iint_{\mathbb{D}} |\mu(f_u)|^2 d A_{\mathrm{hyp}} < \infty.
	\]
\end{proof}

\section*{Acknowledgment}
Figures \ref{fig:circle_patterns_disk}, \ref{fig:circle_patterns_deformed} and \ref{fig:unbounded} are generated by Ken Stephenson's software CirclePack \cite{Stephenson2005}. The author would like to thank Jean-Marc Schlenker and Yilin Wang for helpful discussions.

\bibliographystyle{siam}
\bibliography{uniteich}

\end{document}